\documentclass[10pt,bezier,amstex]{article}  
\usepackage{color}              
\usepackage{graphicx}
\usepackage[]{amsmath}
\usepackage{amssymb}
\usepackage{hyperref,bbm}
\definecolor{MyDarkBlue}{rgb}{0,0.08,0.50}
\definecolor{BrickRed}{rgb}{0.65,0.08,0}

\hypersetup{
colorlinks=true,       
    linkcolor=MyDarkBlue,          
    citecolor=BrickRed,        
    filecolor=red,      
    urlcolor=cyan           
}
\usepackage{multirow}
\usepackage{enumerate}
\usepackage[mathscr]{eucal}
\usepackage{graphicx}
\usepackage{latexsym}
\usepackage{amssymb}
\usepackage{psfrag}
\usepackage{epsfig}

\usepackage{amsthm,amsmath,amssymb,amsfonts,xspace,graphics,enumerate,setspace}

\include{psbox}

\usepackage{amsfonts}

\usepackage{graphicx}

\usepackage{amsfonts}

\usepackage{color}
\usepackage{xspace}

\newtheorem{Lemma}{Lemma}[section]
\newtheorem{Theorem}[Lemma]{Theorem}
\newtheorem{Assumption}[Lemma]{Assumption}
\newtheorem{Proposition}[Lemma]{Proposition}

\newtheorem{Corollary}[Lemma]{Corollary}

\newtheorem{remark}[Lemma]{Remark}

\newcommand{\U}{\mathbf{U}}
\newcommand{\Y}{\mathbf{Y}}

\newcommand{\s}{\quad}

\newcommand{\R}{\mathbb{R}}
\newcommand{\W}{{W}}

\newcommand{\non}{\nonumber}

\newcommand{\GFI}{GFI\xspace}
\newcommand{\GFD}{GFD\xspace}
\newcommand{\DGE}{DGE\xspace}

\newcommand{\beq}{\begin{eqnarray*}}
\newcommand{\eeq}{\end{eqnarray*}}
\newcommand{\beqn}{\begin{eqnarray}}
\newcommand{\eeqn}{\end{eqnarray}}

\newcommand{\bt}{\begin{Theorem}}
\newcommand{\et}{\end{Theorem}}

\newcommand{\bas}{\begin{Assumption}}
\newcommand{\eas}{\end{Assumption}}
\newcommand{\lt}{\left}
\newcommand{\rt}{\right}

\newcommand{\be}{\begin{equation}}
\newcommand{\ee}{\end{equation}}

\newcommand{\mb}[1]{{\mathbf #1}}

\newcommand{\X}{\mathbf{X}}

\newcommand{\GF}{\text{Generalized Fiducial}}

\numberwithin{equation}{section}

\definecolor{darkgreen}{rgb}{0,.4,0}
\definecolor{darkagenta}{rgb}{.5,0,.5}
\definecolor{darkred}{rgb}{1,0,0}
\definecolor{darkblue}{rgb}{0,0,.4}

\begin{document}
	\author{Abhishek Pal Majumder
	\thanks{
Department of Mathematical Science,
University of Copenhagen,
Universitetsparken 5, 
DK-2100 K\o{}benhavn \O{} ,
Email: {\tt Abhishek@math.ku.dk}
	}
	\and
	Jan Hannig
	\thanks{Department of Statistics and Operations Research
	The University of North Carolina
	304 Hanes Hall
	Chapel Hill, NC 27510. Email: {\tt jan.hannig@unc.edu}
	}
	\thanks{This work was supported in part by the National Science Foundation under Grant No. 1016441 and 1512945.}
	}

\title{Higher order asymptotics of Generalized Fiducial Distribution }

\maketitle

\begin{abstract}
Generalized Fiducial Inference (\GFI) is motivated by R.A. Fisher's approach of obtaining posterior-like distributions when there is no prior information available for the unknown parameter. Without the use of Bayes' theorem GFI proposes a distribution on the parameter space using a technique called increasing precision asymptotics  \cite{hannig2013generalized}. In this article we analyzed the regularity conditions under which the Generalized Fiducial Distribution (\GFD) will be first and second order exact in a frequentist sense. We used a modification of an ingenious technique named ``Shrinkage method" \cite{bickel1990decomposition}, which has been extensively used in the probability matching prior contexts, to find the higher order expansion of the frequentist coverage of Fiducial quantile. We identified when the higher order terms of one-sided coverage of Fiducial quantile will vanish and derived a workable recipe for obtaining such GFDs. These ideas are demonstrated on several  examples. 
\end{abstract}

\section{Introduction}

The philosophy of Generalized Fiducial Inference evolved from R.A. Fisher's fiducial argument. Fisher couldn't accept the Bayes/Laplace postulate for the non-informative prior. He argued
\begin{quote}
``{Not knowing the chance of mutually exclusive events and knowing the chance to be equal are two quite different states of knowledge}"\cite{syversveen1998noninformative}.
\end{quote}
Fisher only approved the usage of \textbf{Bayes' theorem} for the case of informative priors since imposing any measure on the parameter space is contrary to ``no-information" assumption. But Fisher's proposal created some serious controversies once his contemporaries realized that this approach often led to procedures that were not exact in frequentist sense and did not possess other properties claimed by Fisher. In a complete manner \cite{hannig2009generalized} gives a list of all references regarding this and subsequent Fiducial approaches.

 Much after Fisher, in context of generalized confidence interval Tsui, Weerahandi \cite{tsui1989generalized, weerahandi1995generalized} suggested a new approach for constructing hypothesis testing using the concept of generalized P-values. Hannig et al \cite{hannig2006fiducial} made a direct connection between fiducial intervals and generalized confidence intervals and proved asymptotic frequentist correctness of such intervals. These ideas took a general shape in \cite{hannig2009generalized} through applications in various parametric model formulations which is now termed as Generalized Fiducial Inference (in short \GFI). From Fisher \cite{fisher1935fiducial, fisher1933concepts} one of the goals of Fiducial inference had been to formulate a clear and definite principle that would guide a statistician to a unique fiducial distribution. GFI does not have such aim and is quite different from that perspective. It treats the techniques as a tool in order to propose a distribution on the parameter space when no prior information is available and uses this distribution to propose useful statistical procedures for uncertainty quantification like an approximate confidence interval, etc.

In last decades there had been a surge of parallel endeavors in modern modifications of fiducial inference. These approaches are well known under a common name: \textit{``distributional inference"} or \textit{fusion learning}. Main emphasis for these approaches was defining inferentially meaningful probability statements about subsets of the parameter space without the need for subjective prior information. They include the \textit{``Dempster Shafer theory" }(Dempster, \cite{dempster2008dempster}; Edlefsen, Liu and Dempster,\cite{edlefsen2009estimating}) and \textit{inferential models} (Martin, Zhang and Liu \cite{martin2010dempster}; Zhang and Liu \cite{zhang2011dempster}; Martin and Liu \cite{martin2014conditional, martin2013inferential, martin2013marginal}). There is another rigorous framework available called \textit{Objective Bayesian inference} that aims at finding nonsubjective model based priors. An example of a recent breakthrough in this area is the modern development of reference priors (Berger, \cite{berger1992development}; Berger and Sun\cite{berger2008objective}; Berger, Bernardo and Sun \cite{berger2009formal, berger2012objective}; Bayarri et al.\cite{bayarri2012criteria}). Another related approach is based on higher order likelihood expansions and implied data dependent priors (Fraser, Fraser and Staicu\cite{fraser2010second}; Fraser\cite{fraser2004ancillaries, fraser2011bayes}; Fraser and Naderi \cite{fraser2008exponential}; Fraser et al.\cite{fraser2010default}; Fraser, Reid and Wong\cite{fraser2005model}). A different frequentist approach namely confidence distributions looks at the problem of obtaining an inferentially meaningful distribution on the parameter space (Xie and Singh\cite{xie2013confidence}). Recently, Taraldsen and Lindqvist \cite{taraldsen2013fiducial} show how some simple fiducial distributions that are not Bayesian posteriors naturally arises within the decision theoretical framework. 

Arguably, Generalized Fiducial Inference has been on the forefront of the modern fiducial revival. The strengths and limitations of the fiducial approach are starting to be better understood; see especially Hannig \cite{hannig2009generalized,hannig2013generalized}. In particular, the asymptotic exactness of fiducial confidence sets, under fairly general conditions, was established in Hannig \cite{hannig2013generalized}; Hannig, Iyer and Patterson \cite{hannig2006fiducial}; Sonderegger and Hannig \cite{sonderegger2014fiducial}. 

{Main aim of this article is to further study exactness property of the Fiducial quantile in frequentist sense for uni-parameter cases with exploration of higher order asymptotics. From a different point of view it can be seen as a prudent way of selecting a data generating equation (to be defined shortly) so that the non-uniqueness issue of proposing Generalized Fiducial Distribution (in short \GFD) can be reduced partially.} We start with with the definition of \GFD. 

Denote the parameter space by $\Theta$. Let the data $\mb X$ be a $\mb S$ valued random variable. \GFD starts by expressing a relationship between the parameter and the data through a deterministic function $G:\mb M \times\Theta \to \mb S$ which we call \textit{data generating equation} (in short \DGE): 
\beqn
\mathbf{X}=\mb G(\mb U,\theta).\label{gen1}
\eeqn
Here $\mb U$ is a $\mb M$ valued random variable whose distribution doesn't depend on $\theta$. The distribution of the data {$\mathbf X$} is determined by { $\mb U$} via (\ref{gen1}). That is one can generate { $\mathbf X$} by generating {$\mb U$} and plugging it into the data generating equation.

 For example for one sample of $N(\theta,1)$ the \DGE is
$$G(U,\theta)= \theta+ \Phi^{-1}(U)$$
where $\Phi(.)$ cumulative Normal distribution function and $U \sim U(0,1)$. One can always find (\ref{gen1}) by following construction. 
 For a realization $\mb x_{0}:=(x_1,x_2,\ldots x_n)$ of $\mb X$ where $\mb X \sim F_{\theta}(.)$ for $F_{\theta}$ being a distribution function on $\R^{n}$ with $\theta\in\Theta$ being the unknown parameter denote the conditonal distributions of first, second and $n$-th co-ordinate (sequentially given the rest) by $F_{\theta,X_1}(\cdot), F_{\theta,X_2 | X_{1}}(\cdot), $ and $F_{\theta,X_n | (X_{1},X_{2},\ldots, X_{n-1})}(\cdot)$ respectively.  Then (\ref{gen1}) can be written as
\beqn
x_1 &=& F^{-1}_{\theta,X_1}(U_1)\non\\
x_2 &=& F^{-1}_{\theta,X_2\big| \{X_{1}=x_{1}\}}(U_2)\non\\
x_{n}&=&F^{-1}_{\theta,X_n\big| \{(X_{1},X_{2},\ldots, X_{n-1})=(x_1,x_2,\ldots x_{n-1})\}}(U_{n})
\eeqn
where $(U_{1},U_{2},\ldots,U_{n})$ iid copies of Uniform $(0,1)$ random variables. Note that in the above illustration changing the order of the variables $(X_{1},\ldots,X_{n})$ could give different data generating equations. 

After observing $\mb x_{0}$, given $U,$ define the inverse image $Q_{\mathbf x_{0}}(U)$ as
\[Q_{\mathbf x_{0}}(U):=\{\theta: G(U,\theta)=\mathbf x_{0}\}.\]
Fiducial approach instructs us to deduce a {distribution} for {$\theta$} from the randomness of {$U$}  and the \DGE 
via the inverse image $Q_{\mathbf x_{0}}(U)$, i.e.,  generate an independent copy  {$U^\star$} and invert the structural equation solving for { $\theta=Q_{\mathbf x_{0}}(U^\star)$} to obtain a random estimator of the parameter. 

Now in order to remove the possibility of non-existence of solution for some $U^*$, we will discard such values, i.e, condition the distribution of $U^\star$ given the fact the solution always exists, i.e $Q_{\mathbf x_{0}}(U)\neq \emptyset$. Consequently, the Fiducial distribution of $\theta$ given observed $\mathbf x_{0}$ should be heuristically (hence ill defined) the following conditional distribution 
\beqn
Q_{\mathbf x_{0}}(U^*)\bigg| \lt\{Q_{\mathbf x_{0}}(U^*)\neq \emptyset\rt\}.\label{fid1}
\eeqn

Immediately three relevant questions arise regarding the non-uniquenesses of Generalized Fiducial distribution (\ref{fid1}):
\begin{itemize}
\item \textbf{The choice among multiple solutions:}
It arises if the inverse image $Q_{\mathbf {x_{0}}}(U^*)$ has more than one element for $U^*$ and observed $\mathbf {x_{0}}$. This problems mainly occur in discrete distributions which we did not consider in this article (see \cite{hannig2013generalized}). 
\item \textbf{Borel Paradox:} Another important problem regarding computing the conditional probability in (\ref{fid1}) arises when the conditioning event $\{Q_{\mathbf x_{0}}(U^*)\neq \emptyset\}$ has probability $0.$ For example, suppose one observes $\X=\mb x_{0}:= (x_{1},\ldots,x_{n})'$  from the location normal model with the data generating equation $\X = \theta + \U_{n\times 1}$ where $\U\sim N_{n}(0,I_{n})$. The inverse image 
\[
 Q_{\mb x_{0}}(\mathbf{U}^*_{})= \begin{cases}
         (x_{1}-U^*_{1}) &\mbox{if $U^*_{2}-U^*_{1}=x_{2}-x_{1},U^*_{3}-U^*_{1}=x_{3}-x_{1}, \ldots ,U^*_{n}-U^*_{1}=x_{n}-x_{1}$,}\\
         \emptyset & \mbox{otherwise,}
         \end{cases}
 \] and the set $\{Q_{\mb x_{0}}(\mathbf{U}^*_{})\neq \emptyset\}$ has probability $0$ (an $n-1$ dimensional hyperplane in the $n$ dimensional Euclidean space). In that case the conditional probability distribution may not remain unique which in literature is termed as the Borel paradox.

 This problem can also be remedied by defining the \GFD as the distribution of the weak limit of the following quantity (in the display) conditioned on the event $\left\{\inf_{\theta} \big\|\mathbf {x_{0}} - G(U^*,\theta)\big\| \le \epsilon\right\}$ as $\epsilon \downarrow 0$.
\begin{eqnarray}
\arg \inf_{\theta}\big\|\mathbf {x_{0}} - G(U^*,\theta)\big\| \bigg| \left\{\inf_{\theta} \big\|\mathbf {x_{0}} - G(U^*,\theta)\big\| \le \epsilon\right\}.\label{fid2}
\end{eqnarray}
Let's assume that for each fixed $\theta \in \Theta$ the function $G(\cdot, \theta)$ is one-to-one and continuously differentiable denoting the inverse by $G^{-1}(x, \theta)$. If we use $L_{\infty}$ norm as $\|\cdot\|$ in the definition of \eqref{fid2}, from Theorem 3.1 of \cite{hannig2013generalized} it follows that the unique weak limit is a distribution on $\Theta$ with density
\beqn
f^{\mathcal{G}}\big(\theta \big|\mathbf {x_{0}}=\{ X_{1},X_{2},\ldots,X_{n}\}\big) = \frac{f_{\X}\lt(\mathbf{x_{0}}\big|\theta\rt) J_{n}(\mathbf{x_{0}},\theta)}{\int_{\mathbb{R}}f_{\X}\lt(\mathbf{x_{0}}\big|\theta'\rt) J_{n}(\mathbf{x_{0}},\theta')d\theta'},\label{fiddef}
\eeqn
where in the one parameter $(p=1)$ case, the Jacobian becomes 
\begin{eqnarray}
J_{n}(\X,\theta) \propto \sum_{i=1}^{n} \Bigg|\frac{\partial}{\partial \theta}G_i(\U,\theta)\bigg|_{\U = \mb G^{-1}(\X,\theta)}\Bigg|.\label{jacobian}
\end{eqnarray}
In what follows by Fiducial distribution (or density) we will mean the distribution(or density) of $\theta$ defined in (\ref{fid2}). Denote the probability distribution on $\Theta$ induced by the data generating equation $\mathcal{G}$ in (\ref{gen1}) by $P^{\mathcal{G}}(\cdot)$ whose density is (\ref{fiddef}).
\end{itemize}
As an example let $\X$ be a sample of $n$ iid observations from {\em Scaled Normal family} $N(\mu,\mu^{q}), \mu>0$ with $q>0$ known.  The simplest data generating equation comes from the relation $\X=G(\mathbf{U},\theta):=\mu+\mu^{\frac{q}{2}}\mathbf{U},$ where $\mathbf{U}=(U_1,U_2,\ldots,U_{n})$ is an array of $n$ i.i.d $N(0,1)$ random variables. Since the derivative of the $i$th coordinate $\frac{d}{d\mu}G_i(U_i,\mu)=1+\frac{q}{2}\mu^{q/2-1} U_i,$ and $U_i=\frac{x_i-\mu}{\mu^{q/2} },$ the Jacobian  in (\ref{jacobian}) simplifies to
\beqn 
J_n(\mathbf x,\mu)\propto  \sum\limits_{i=1}^n \left| 1+\frac{q(X_i-\mu)}{2\mu}\right|\label{ex1}
\eeqn
with the corresponding \GFD specified in (\ref{fiddef}).

In context of non-informative prior for any one-to-one function $\phi(.),$ inference of $\theta$ given $\X$ and inference of $\phi(\theta)$ given $\X$ should not be different since the \textit{ideal non-informative prior should not impose any extra information on $\Theta$}  \cite{syversveen1998noninformative}. Just like Posterior distribution of \textbf{Jeffrey's prior}, Generalized Fiducial distribution of $\theta$ as defined in (\ref{fid2}) exhibits this \textbf{parametrization invariance} property. 

%

\subsection{The Choice of Structural Equations}

While the \GFD is invariant to re-parametrization, it is not invariant to changes in the data generating equation.
We illustrate this issue on the following example:

 Let $\mb X=(X_1,\ldots,X_{n})$ be $n$ iid realizations from a distribution with density $f(\cdot\big|\theta)$ parametrized by one dimensional parameter $\theta\in \R$ and consider \DGE $\mb X=\mb G(\mb U,\theta)$. Suppose $T(\cdot)$ is absolutely continuous, one to one transformation and denote the derivative $w(x)=T'(x)$. Now considering the transformed data  $\mb Y:= T(\mb X) = T\circ \mb G(\mb U,\theta).$ Then (\ref{fiddef}) implies that the density of fiducial distribution of $\theta$ based on the transformed \DGE for $\mathbf y$ is
 \begin{equation}
f^{\mathcal{G}}(\theta\big|\mb Y=\mb y)= \frac{f_{\mb Y}\lt(\mathbf{y}\big|\theta\rt) J_{n}(\mathbf y,\theta)}{\int_{\Theta}f_{\mb Y}\lt(\mathbf{y}\big|\theta'\rt) J_{n}(\mathbf y,\theta')d\theta'} 
= \frac{\prod_{i=1}^{n}f_{}\lt(\mathbf{\mathbf{x}}_{i}\big|\theta\rt) J\lt(B(\mb x),\theta\rt)}{\int_{\Theta}\prod_{i=1}^{n}f_{}\lt(\mathbf{x}_{i}\big|\theta'\rt) J_n(T(\mb x),\theta')d\theta'}\label{tr1}
\end{equation}
with the Jacobian $J_n(T(\mb X),\theta)$ 
 \begin{equation}
\sum_{i=1}^{n} \Bigg|\frac{\partial}{\partial \theta}T_i\circ \mb G(\U,\theta)\bigg|_{\U = \mb G^{-1}(\X,\theta)}\Bigg|
=\sum_{i=1}^{n} \Bigg| w(X_{i})\frac{\partial}{\partial \theta} G_i(\U,\theta)\bigg|_{\U = \mb G^{-1}(\X,\theta)}\Bigg|.\label{tr}
\end{equation}
Note that (\ref{tr1}) differs from (\ref{fiddef}) computed based on the original \DGE only in the form of the Jacobian function (\ref{tr}). In particular (\ref{tr}) has an extra weight factor $w(x_i)$ absent in (\ref{jacobian}).
Intuitively fiducial distribution changes due to the deformation of the  neighborhoods of the data $\mathbf{x}$ through the transformation $T$ and the consequent change in the shape of the conditioning event $\inf_{\theta} \left\|T(\mathbf {x}) - T\circ \mb G(\mb U^*,\theta)\right\| \le \epsilon$ in the definition (\ref{fid2}).

If the \DGE for each individual $X_i,\ i=1,\ldots, n$ is the inverse cdf $X_i=G(U_i,\theta)=F^{-1}(U_i,\theta)$ then (\ref{tr}) becomes
\begin{equation}
J_{n}(T(X),\theta) \propto   \sum_{i=1}^{n}\bigg|w(X_{i}) \frac{\frac{\partial F_{i}(X_{i},\theta)}{\partial \theta}}{f_{}(X_{i},\theta)}\bigg|.\label{examplejac}
\end{equation}

We conclude that smooth and one-one transformations of the original \DGE results in introduction of a weight $w(\cdot)$ in the Jacobian function. The question of interest is, \textit{what is an ``ideal" transformation $T(\cdot)$ for which the Fiducial distribution enjoys some ``desirable" properties.} In particular we will consider \DGE {desirable} if it has some good frequentist properties.

 In what follows we first give an ideal recipe for a special case when some strong monotonicity conditions are satisfied. In the absence of those conditions, we study quality of the \GFD through higher order asymptotics.  The main goal of this article  is to obtain transformation $T(\cdot)$ so that the data generating equation $``T(\mb X)= T\circ G(\mb U, \theta)"$ will give a first order probability matching Fiducial distribution (to be defined in (\ref{matching})). We will conclude with some examples for which we derive higher order matching fiducial distribution and study its properties using small sample situations.

\section{Why Fisher Might Have Thought Fiducial Distribution Exact and Unique?}
Fisher \cite{fisher1935fiducial} developed  the Fiducial idea in conjunction with the concept on the minimal sufficient statistics. Using the same motivation we state the following theorem considering \GFD based on $S=G_S(\mb{U},\theta)$ the data generating equation for a one dimensional statistics $S$.  

Denote the distribution function of $S$ by $F_S(s,\theta)$, the inverse image of the \DGE by $Q_{s}(\mb u)=\{\theta\,:\, s=G_S(\mb u,\theta)\}$ and an independent copy of $\mb U$ by $\mb U^\star$ .
The \GFD as defined in (\ref{fid2}) is the conditional distribution
 $Q_{s}(\mb U^\star)\mid \{Q_{s}(\mb U^\star)\neq\emptyset\}$ provided $P(Q_{s}(\mb U^\star)\neq\emptyset)>0$, In this simplest of settings we typically observe that the $P(Q_{s}(\mb U)\neq\emptyset)=1$ in which case \GFD is the unconditional distribution of $Q_{s}(\mb U^\star)$

\bt\label{e3.1}
  Let us assume that 1) for all $\mb u,$ the function $G_S(\mb{u},\theta)$ is non-decreasing in $\theta$ and 2) for all $\mb u$ and $\theta$ we have $Q(s,\mb u)\neq\emptyset$. Then the inverse image $Q(s,\mb u)$ is an interval with bounds $Q_{s}^{-}(\mb u)\leq Q_{s}^+(\mb u)$. Additionally, for any $s_0$ and $\theta_0$ we have $   P(Q_{s_{0}}^{+}(\mb U^\star) \leq  \theta_0)=1-\lim_{\epsilon\downarrow 0} F_S(s_0,\theta_0+\epsilon)$
and $    P(Q_{s_{0}}^{-}(\mb U^\star) \leq \theta_0)= 1-\lim_{\epsilon\downarrow 0} F_S(s_0-\epsilon,\theta_0).$  

If additionally 3) for all $\theta_0$ and $s_0$ the $P_{\theta_0}(S=s_0)=F_S(s_0,\theta_0)-\lim_{\epsilon\downarrow 0} F_S(s_0-\epsilon,\theta_0)=0$ then the distribution function $F_S(s,\theta)$ is continuous as a function of $\theta$, $Q_{s_{0}}^{+}(\mb U^\star) = Q_{s_{0}}^{-}(\mb U^\star)$ with probability $1$, and 
\begin{equation}\label{contSfid}
   P(Q_{s_0}(\mb U^\star) \leq  \theta_0)=1-F_S(s_0,\theta_0).
\end{equation}
\et
\begin{proof}[Proof of Theorem \ref{e3.1}]
 The fact that $Q(s,\mb u)$ is an interval follows by monotonicity. Consider an iid sample $\mb U_1^\star,\ldots ,\mb U_n^\star.$ By SLLN we have
 \beqn
  P(Q_{s_0}^+(\mb U^\star) <  \theta_0) &=& \lim_{n\to\infty} \frac 1n \sum_{i=1}^n I_{\{Q^+_{s_0}(\mb U_i^\star)<\theta_0\}}                                          =\lim_{n\to\infty} \frac 1n \sum_{i=1}^n I_{\{G_S(\mb U_i^\star,\theta_0)>s_0\}}
                                                   =1-F_S(s_0,\theta_0),\nonumber
 \eeqn
 where the second equality follows from monotonicity of $G_S$. The result follows by taking a limit. 
 (Notice that the distribution function $F_S(s,\theta)$ is non-increasing in $\theta$.)
 Similarly
\beqn
  P(Q^-_{s_0}(\mb U^\star) >  \theta_0) = \lim_{n\to\infty} \frac 1n \sum_{i=1}^n I_{\{Q^-_{s_0}(\mb U_i^\star)>\theta_0\}}
                                                   =\lim_{n\to\infty} \frac 1n \sum_{i=1}^n I_{\{G_S(\mb U_i^\star,\theta_0)<s_0\}}
                                                   = \lim_{\epsilon\downarrow 0} F_S(s_0-\epsilon,\theta_0).\nonumber \eeqn 
 Finally, if  $\lim_{\epsilon\downarrow 0} F_S(s_0,\theta_0+\epsilon)-F_S(s_0,\theta_0)<0$ then $P_{\theta_0}(S=s_0)>0$. The rest of the proof follows by simple comparison.
\end{proof}

\begin{remark}
To understand the main message of Theorem~\ref{contSfid} assume 1) and 2) and denote by $C^+_\alpha(s)$ the $1-\alpha$ quantile of $Q_{s}^{+}(\mb U^\star)$, i.e.  $C^+_\alpha(s)=\sup_c\{P(Q_{s}^{+}(\mb U^\star) < c)\leq 1- \alpha\}$. 
Since \[\{s: \theta\leq C^+_\alpha(s)\}\supset \{s: P(Q_{s}^{+}(\mb U^\star) < \theta)\leq 1-\alpha\}=\{s: F_S(s,\theta)\geq \alpha\},\]
then $P_\theta(\theta\leq C^+_\alpha(s))\geq P_\theta(F_S(S,\theta)\geq\alpha)\geq 1-\alpha$ and the 
set $(-\infty, C^+_\alpha(s))$ forms a $(1-\alpha)$ level upper confidence bound. Notice that the coverage is guaranteed to be either exact or conservative.

Similarly, if $C^-_\alpha(s)=\inf_c\{P(Q_{s}^{-}(\mb U^\star) > c)\leq 1- \alpha\}$ then $P_\theta(\theta\geq C^-_\alpha(S))\geq 1-\alpha$ and 
the  set $(C^-_\alpha(s),\infty)$ forms a $(1-\alpha)$ level upper confidence bound.

If additionally 3) is satisfied, then the coverage of the one sided confidence bounds is exact (not conservative). We will say that in this case the \GFD is exact.

Finally we remark that Theorem~\ref{contSfid} also implies that the \GFD is the same for all possible \DGE{}s satisfying 1) and 2).
\end{remark}

\begin{remark}
If $G_S(\mb{u},\theta)$ is non-increasing in $\theta$ then a similar theorem can be proved by a simple re-parametrization.
\end{remark}

Now we will generalize Theorem \ref{e3.1} beyond the existence of $1$-dimensional sufficient statistics under the following assumption:
\bas\label{Ascor}
Let us consider a data generating equation $  \mb X=\mb G(\mb U,\theta).$ Assume that there exists a \textit{one-one} $\mathcal{C}^{1}$ \textit{ transformation}  $(S(\mb X), A(\mb X))$ of $n$ dimensional $\mb X,$ such that $S(\mb X)$ is one dimensional and $\mb A(\mb X)$ is an $(n-1)$ dimensional vector of ancillary statistics (more precisely the function $\mb A(\mb G(\mb U,\theta))$ is invariant in $\theta$).
\eas
After the transformation $(S,\mb A)$ on the initial data generating equation $\mb X=\mb G(\mb U,\theta),$ the new one can be written as 
\begin{equation}\label{seS}
 s=G_S(\mb{U},\theta):=S\circ G(\mb{U},\theta),\s\text{and} \s \mb a=\mb G_{\mb A}(\mb{U}):=A\circ G(\mb{U},\theta).
\end{equation}

If $P(Q_{s}(\mb U)\neq\emptyset)=1$, then 
a simple calculation shows that the \GFD as defined in (\ref{fid2}) is the distribution of $Q_{s}(\mb U_a^\star)$ where 
$U_a^\star$ has as its distribution the conditional distribution of $\mb U\mid \{\mb G_{\mb A}(\mb U) =\mb a\}$.
Denote the conditional distribution function  $S(\mb X)\mid \{\mb A(\mb X)=\mb a\}$ by $F_{S\mid \mb a}(\cdot,\theta)$. 
The following corollary is immediate:

\begin{Corollary}\label{cor1}
Under Assumption \ref{Ascor}, suppose the $G_S(\mb{u},\theta)$ satisfies conditions in Theorem~\ref{contSfid}. Then the 
conclusions of Theorem~\ref{contSfid} remain satisfied with $F_S(s,\theta)$ replaced by $F_{S\mid \mb a}(\cdot,\theta)$.
\end{Corollary}

\subsection{Examples}\label{s:examples}
In this section we consider a few simple examples of one-parameter problems. The first two are examples that satisfy conditions of Corollary~\ref{Ascor} and therefore lead to exact \GFD. The other two examples do not satisfy these conditions and we need a way of selecting between the different potential \DGE{}s.

\begin{enumerate}[(A)]
\item \textbf{Location family}: Let $X_{i}=\theta + U_i,\ i=1,\ldots, n$ where $U_i$ are iid. The following one-one transformation
\[T: \X \to (\bar{X}_{n},(X_{1}-\bar{X}_{n},\ldots,X_{n-1}-\bar{X}_{n})):= (S(\X),\mb A(\X))\] is one-one and $\mb A$ is ancillary.  By Corollarry~\ref{cor1} the \GFD based on $S(\X)\mid A(\X)$ is exact. Moreover, the Jacobian $J_n(s,\theta)\propto 1$ and the \GFD is the same as posterior for Jeffreys prior.

\textbf{Scale family} Let $X_{i}=\theta U_i,\ i=1,\ldots, n$ where $U_i>0$ are iid. Set $\tilde{X}_{n}$ as the geometric mean of $\mb X$
and notice that the one-to-one
\[T_{1}: \X \to \lt(\tilde{X}_{n},\left(\frac{X_{1}}{\tilde{X}_{n}},\ldots,\frac{X_{n-1}}{\tilde{X}_{n}}\right)\rt):= (S_{1}(\X),\mb A_{1}(\X))\]
again satisfies Assumption~\ref{Ascor}.
Consequently, the fiducial distribution based on $\tilde{X}_{n}$ conditional on $A_{1}(\X)$ is exact. Again, the Jacobian $J_n(s,\theta)\propto \theta^{-1}$ and the \GFD is the same as posterior for Jeffreys prior.

For a concrete examples, models $U(\theta,\theta+1)$, Cauchy with unknown location parameter $\mu$ and known scale, or 
Cauchy with scale parameter $\sigma$ and known location that do not have a one dimensional sufficient statistic still all have exact \GFD.

\item {\bf Exponential Family:} Let $\mb X=(X_1,\ldots, X_n)$ be i.i.d.\ sample from an natural exponential family with one parameter $\eta$;  
$f_{X}(x|\eta)=h(\mb x)e^{\eta S(x)-C(\eta)}$.  Since $S_n(\mb X)=\sum_{i=1}^n S(X_i)$ is complete sufficient statistics it isi ndependent of any ancillary statistics  by Basu's Theorem \cite{CasellaBerger2002}. Thus we can base our inference only on $S(\mb X)$ using directly Theorem~\ref{contSfid}. 

To this end, let us assume the inverse distribution function \DGE  $X_i=F^{-1}(U_i,\eta)$ and assume 
 $S(x)$ is a smooth one-to-one function on the domain of $X$. The verify the monotonicity condition of Theorem \ref{e3.1} we need to verify that for each $x$ the distribution function $F_X(x,\eta)$ is monotone onto $(0,1)$ function of $\eta$. For more detailed discussion of propertied of \GFD for exponential family consult \cite{veronese2014fiducial}.

If $X$ is continuous the density of the fiducial distribution is given by (\ref{examplejac}) with $w(x)=S'(x)$ and
\[
\frac{\partial F(x,\eta)}{\partial \eta}=E_{\eta}[S(X)1_{\{X\le x\}}] - C'(\eta)F(x,\eta),
\]
which should not alter its sign for every $x,  \eta$. 

As an example consider the Gamma$(\theta,1)$, i.e., $f_{\theta}(x)=\frac{e^{-x}x^{\theta-1}}{\Gamma(\theta)}.1_{\{x>0\}}.$ The statistic $S(X)=\log X$ is smooth and one-one. The distribution function is incomplete Gamma function which, for each $x$, is a function decreasing in $\theta$ onto $(0,1)$. Thus the \GFD is exact. Notice that the weight  function in (\ref{examplejac}) is $w(x)=x^{-1}$. 
As opposed to the location and scale parameter model discussed above, the \GFD for Gamma distribution does not coincide with a Bayesian posterior for any prior.

%

\item  \textbf{Scaled normal family} $N(\mu,\mu^{q})$, $\mu>0$ with $q>0$ known. When $q=2$ the model is a scale family model of Example (A) and \GFD is exact. For $q\neq2,$ the assumption of Corollary \ref{cor1} will not hold. There are several \DGE one can consider for this model. 

For sample size $n>1,$ we considered the fiducial distribution based on the simplest data generating equation in (\ref{ex1}). Two more choices are based on transformations of the two dimensional minimal sufficient statistics:
\begin{align}
 (\bar{X}_{n},S_n)&=\left(\mu+\mu^{q/2} Z,\mu^{q/2} U^{1/2}\right), \s \label{ex2s}  \\
 (\bar{X}^{2}_{n}\text{sgn}(\bar{X}_{n}),q S^{2}_n)&=\left((\mu+\mu^{q/2} Z)^2\text{sgn}(\mu+\mu^{q/2} Z),q\mu^{q}U\right)\label{ex2}
\end{align}
where $Z\sim N(0,{1}/{n})$ independent of $U\sim \chi^{2}_{n-1}/(n-1)$ and $\text{sgn}(x)=1_{[x>0]}-1_{[x<0]}.$ 
The corresponding Jacobians are
\[
   \s J_{n,2}(\mathbf x,\mu)=\left| 1+\frac{q(\bar x-\mu)}{2\mu}\right|+\frac{q s_n}{2\mu},\s\text{and}\s J_{n,3}(\mathbf x,\mu)=2\bar x_n \left| 1+\frac{q(\bar x-\mu)}{2\mu}\right|+ \frac{q^2 s_n^2}{\mu}.
\] 

\item \textbf{Correlation coefficient $\rho\in (-1,1)$ of a Bivariate normal model:} Suppose $(X_1,Y_1),\ldots,(X_n, Y_n)$ be i.i.d. N{\tiny $\left(\begin{pmatrix} 0\\0\end{pmatrix}
           , \begin{pmatrix}
              1 &\ \rho\\
              \rho & 1
              \end{pmatrix}\right)$}. 
               This model has been first proposed by Basu \cite{Basu1964} as an example of a distribution without maximal ancillary statistic.
              Here we consider  three potential data generating equations:

The simplest symmetric data generating equation models the data directly
\[
 (X_i,  Y_i)=B(Z_i,\rho Z_i +\sqrt{1-\rho^2} W_i) + (1-B)(\sqrt{1-\rho^2} Z_i + \rho W_i,W_i),\ i=1,\ldots,n.\] 
 where $Z_i, W_i, i=1,\ldots n$ are iid $N(0,1)$ and $B$ is a single independent Bernoulli(1/2). The Jacobian for this \DGE is
 \[J_{n,1}((\X,\Y),\rho)=\frac{\sum_{i=1}^n |X_i- \rho Y_i|+|\rho X_i- Y_i|}{2(1-\rho^2)}.\]

 We also construct data generating equations based on transformations of the minimal sufficient statistics. Denote 
$V_1:= \frac{1}{2n}\sum_{i=1}^{n}(X_{i}+Y_{i})^{2},V_2 := \frac{1}{2n}\sum_{i=1}^{n}(X_{i}-Y_{i})^2$ and set $U_1,U_2$ as iid $\chi_{n}^{2}/{n}$ . This allows us to form \DGE{}s
\begin{equation}
(V_{1},V_{2})=((1+\rho)U_{1}, (1-\rho)U_{2})\s \text{ and }\s\lt(\frac{1}{V_{1}},\frac{1}{V_{2}}\rt)=\lt(\frac{1}{(1+\rho)U_{1}}, \frac{1}{(1-\rho)U_{2}}\rt)\label{bivnorm}
\end{equation}
with corresponding Jacobians 
\[J_{n,2}((\X,\Y),\rho)=\frac{V_1}{1+\rho}+\frac{V_2}{1-\rho}\s\text{and}\s J_{n,3}((\X,\Y),\rho)=\frac{1}{V_{1}(1+\rho)}+\frac{1}{V_{2}(1-\rho)}.\] 
\end{enumerate}
We discuss a way of selecting between the generating equation in Examples (C) and (D)  in the next section.

\section{Probability Matching Data Generating Equation:}
We define $\mathcal{G}_{s}$ as the \textbf{Probability Matching Data Generating Equation} of order $s\in\mathbb{N}$ if
\beqn
 P_{\theta_{0}}\lt[\theta_{0}<\theta^{1-\alpha}(\X,\mathcal{G}_{s})\rt]=P^{\mathcal{G}_{s}}(\theta < \theta^{1-\alpha}(\X,\mathcal{G}_{s})\mid \X)+o(n^{-\frac{s}{2}})\label{matching}
\eeqn
where $\theta^{1-\alpha}(\X,\mathcal{G}_{s})$ is the upper $(1-\alpha)$-th quantile of the $\GF$ Distribution $P^{\mathcal{G}_{s}}(\cdot\mid \X).$ In other words it characterizes the corresponding data generating equation, so that the frequentist coverage of the $(1-\alpha)$-th Fiducial quantile matches $(1-\alpha)$ upto rate  $o(n^{-\frac{s}{2}}).$ 

We plan to guide our choice of \DGE based on the frequentist coverage in (\ref{matching}). The reason for this is that the higher $s$ is the better the well one sided quantile of \GFD behaves \textit{asymptotically in frequentist sense}. 
This choice has been motivated by
probability matching priors in {non-subjective bayesian paradigm}.
\cite{Ghosh2011, syversveen1998noninformative, WelchPeers1963}. In fact one criteria for  judging a quality of a non-informative prior is the frequentist coverage of $(1-\alpha)$th posterior regions.  An ideal non-informative prior should match all order terms at the true parameter value but constructing is often impossible. In \GFI, the challenge translates into finding the data generating equation for which the fiducial quantile has the exact ideal coverage (i.e least influence on the parameter space). Similar to Bayesian paradigm achieving exactness is often impossible (but see Corollary~\ref{Ascor}). Finding \DGE  that has either first or second order matching quantiles is a more generally achievable goal. 

In one parameter models it is well-known that when regularity conditions are satisfied, Jeffreys prior is the only non-data dependent prior with posterior that is first order matching \cite{datta2004probability, WelchPeers1963}. Notice that because the Jacobian function (\ref{tr}) is data dependent, \GFD can still be second order matching even though Jeffreys prior is only rarely second order matching.

\subsection{Main Results}

In this section we present the main theoretical contribution  of this paper, Theorem~\ref{uni}. The detailed statement of the assumptions are listed in Appendix~\ref{s:assumptions}. 

Assumption~\ref{str} is a standard set of conditions to ensure a valid higher order likelihood expansion of the loglikelihood $L_{n}(\theta):=\frac{1}{n}\sum_{i=1}^{n}  \log f(X_{i},\theta)$  \cite{datta2004probability}. Assumptions~\ref{ext} is commonly used to control the tail behavior of the log-likelihood and \GFD \cite{hannig2013generalized}.

Assumption~\ref{ext2} controls the behavior of  Jacobian $J_{n}(\X,\theta)$ as $n\to\infty$.
For example, it assumes the locally uniform (in $\theta$ on a neighborhood of $\theta_0$) convergence of
\[J_{n}^{(i)}(\X,\theta) \to J^{(i)}(\theta_0,\theta),\]
where $J_{n}^{(i)}(\X,\theta)=\frac{\partial^{i} J_{n}(\X,\theta)}{\partial\theta^i}$ and $J^{(i)}(\theta_0,\theta)=\frac{\partial^{i} J(\theta_0,\theta)}{\partial\theta^i}$ are the derivatives of the Jacobian function and its limit. 
Finally,  Assumption~\ref{as4} is needed for the regularity of the second order term and is not necessary if only first order expansion was needed. For example it assumes the existence of
\[
a_{i}(\theta_{0})=\lim_{n\to \infty}E_{\theta_{}}\sqrt{n}\lt[J_{n}^{(i)}(\X,\hat{\theta}_{n}) - J^{(i)}(\theta_{0},\hat{\theta}_{n})\rt],
\]
where $\hat\theta$ is the MLE of $\theta$.

Let $\phi(z)$ and $z_{\alpha}$ be the density and  $(1-\alpha)$-th quantile of Normal distribution. Denote the Fisher information as $I_\theta$ and define $m_3(\theta)=E_{\theta}\lt[\frac{\partial^3}{\partial\theta^3}\log f(X, \theta)\rt]$. Now we will state the main result for first and second order terms in expansion of the coverage of one sided fiducial intervals. 
\bt\label{uni}
Suppose Assumptions  \ref{str},\ref{ext},\ref{ext2},\ref{as4} hold with $m=2.$ Then for fixed values of the true parameter $\theta_{0}\in\Theta\subset\mathbb{R}$ 
\beqn
P_{\theta_{0}}\bigg[\theta_{0}\le \theta^{1-\alpha}(\mathcal{G},\X,n)\bigg]&=& \big(1-\alpha\big) + 
\frac{\phi(z_\alpha)\Delta_1(\mathcal{G})}{\sqrt{n}}+\frac{z_\alpha\phi(z_\alpha)\Delta_2(\mathcal{G})}{n}+o\bigg(\frac{1}{n}\bigg),\s\s\text{where}\non
\eeqn
\beqn
\Delta_1 (\mathcal{G})&:=& \bigg[I_{\theta_{0}}^{-\frac{1}{2}} \frac{\partial}{\partial \theta}\log J(\theta_0,\theta) + \frac{\partial }{\partial \theta}I^{-\frac{1}{2}}_{\theta}\bigg]\Bigg|_{\theta=\theta_{0}}, \label{et2a}\s\s\s\\
\Delta_2(\mathcal{G}) &:=& J(\theta_0,\theta_0)^{-1}\bigg[\frac{1}{6}\frac{\partial}{\partial \theta}\lt\{I^{-2}_{\theta}J(\theta_0,\theta)m_3(\theta)\rt\}- \frac{1}{2}\frac{\partial^2}{\partial \theta^2}\left\{J(\theta_0,\theta)I^{-1}_{\theta}\right\}\bigg]\Bigg|_{\theta=\theta_0}\non\\&&+\frac{I^{-\frac{1}{2}}_{\theta_{0}}}{z_{\alpha}J(\theta_0,\theta_0)}\bigg[a_{1}(\theta_{0}) - a_{0}(\theta_{0}) \frac{\partial}{\partial \theta}\log J(\theta_0,\theta)\bigg]\Bigg|_{\theta=\theta_0}.\,\label{secondord}
\eeqn
\et

\begin{Corollary}\label{cor3}
Under Assumptions  \ref{str},\ref{ext},\ref{ext2} with $m=1,$ $\mathcal{G}_{1}$ will be the first order Probability Matching \DGE if $\Delta_1 (\mathcal{G}_{1}) =0.$

Under Assumptions  \ref{str},\ref{ext},\ref{ext2},\ref{as4} with $m=2,$ $\mathcal{G}_{2}$ will be the second order Probability Matching  if $\s\s\Delta_1 (\mathcal{G}_{2}) =0, \s\text{and}\s \Delta_2(\mathcal{G}_{2}) =0.$
\end{Corollary}

The detailed proofs are in Appendix~\ref{s:proofs}. 
Broadly speaking our proof follows a similar approach as in the Bayesian probability matching literature \cite[Chapter 5]{ghosh1994higher}.
In contrast to the Bayesian literature \cite{datta2004probability} the main difference is due to the form of the data dependent Jacobian $J_{n}(\X,\theta)$. 
This is caused by the presence of the extra terms 
 in the expansion of the Jacobian function $J_{n}(\X,\theta)$. These new challenges demonstrate themselves especially in the proof of Lemma~\ref{lem1}.

\section{Recipe For Creating Higher Order Probability Matching Data Generating Equation} 

In this section we provide guidelines on how to identify \DGE with desired matching properties.
 \begin{enumerate}[(a)]
\item \textit{Start with sufficient statistic  $\mb S=(S_{1},S_{2},\ldots,S_{m})$.  From computational point of view prefer $S_{1},S_{2},\ldots,S_{m}$ independent of each other.}
\item \textit{Reverse engineer a smooth transformation $\mb g(\mb S)$ that meets conditions of Corollary~\ref{cor3}.}
\end{enumerate}
Denote
$
 \mathcal{A}_{\mathcal G}:=\{\text{Space of all $C^{1}$ one-one transformations from $R_S$ into $\mathbb R^m$}\},
$ 
where $P_{\theta_0}(\mb S\in R_S)=o(e^{-an})$ for some constant $a>0$. ($R_S$ is the range of $\mb S$ up to an exponential term that does not affect polynomial rates of convergence.)
From Corollary \ref{cor3} we define the set of transformations yielding first and second order probability matching data generating equations respectively as:
\beqn
{\mathcal{A}^{(1)}_{\mathcal{G}}=\{A\in \mathcal{A}_{\mathcal{G}}: \Delta_{1}(\mathcal{G})=0\}, \s\text{and}\s \mathcal{A}^{(2)}_{\mathcal{G}}=\{A\in \mathcal{A}^{(1)}_{\mathcal{G}}:\Delta_{2}(\mathcal{G})=0\}.}\non
\eeqn

We will find the class $\mathcal{A}^{(1)}_{\mathcal{G}}$ for the
two motivating examples.
\begin{enumerate}[(a)]
\item  N{\tiny $\left(\begin{pmatrix} 0\\0\end{pmatrix}
           , \begin{pmatrix}
              1 &\ \rho\\
              \rho & 1
              \end{pmatrix}\right)$}. The generating equation of minimal-sufficient statistics is
\[(S_{1},S_{2})=\bigg(\frac{1}{n}\sum_{i=1}^{n}(X_{i}+Y_{i})^{2},\frac{1}{n}\sum_{i=1}^{n}(X_{i}-Y_{i})^{2}\bigg)=\big((1+\rho)U_{1}, (1-\rho)U_{2}\big).\]
For $A\in  \mathcal{A}_{\mathcal{G}},$ the Jacobian for the transformed $(A_{1}(S_{1}),A_{2}(S_{2}))=\big(A_{1}((1+\rho)U_{1}),A_{2}( (1-\rho)U_{2}\big))$ will be 
\beqn
J^{A}_{n}(\X,\rho)=A'_{1}(S_{1})\frac{S_{1}}{1+\rho}+A'_{2}(S_{2})\frac{S_{2}}{1-\rho} \longrightarrow J^{A}(\rho_{{0}},\rho):=A'_{1}(1+\rho_{{0}})\frac{1+\rho_{{0}}}{1+\rho}+A'_{2}(1-\rho_{{0}})\frac{1-\rho_{{0}}}{1-\rho}\non
\eeqn
as sample size $n\to\infty$.

 The first order class $\mathcal{A}^{(1)}_{\mathcal{G}}$ is found by solving $\Delta_{1}(\mathcal{G})=0,$ which is equivalent to
 \[
 \frac{\frac{\partial}{\partial \rho}J^{A}(\rho_{{0}},\rho)}{J^{A}(\rho_{{0}},\rho_{0})}\Bigg|_{\rho=\rho_{{0}}}
 = \frac{(1+\rho_{{0}}) -(1-\rho_{{0}})\bigg[\frac{A'_{1}(1+\rho_{{0}})}{A'_{2}(1-\rho_{{0}})}\bigg]}{(1-\rho^{2}_{{0}})\bigg(1+\bigg[\frac{A'_{1}(1+\rho_{{0}})}{A'_{2}(1-\rho_{{0}})}\bigg]\bigg)}
 =\frac{3\rho_{{0}}+\rho^{3}_{{0}}}{(1-\rho^{2}_{{0}})(1+\rho^{2}_{{0}})}
  =\frac{1}{2}\frac{I'_{\rho_{{0}}}}{I_{\rho_{{0}}}}
 \]
Consequently
\beqn
{\mathcal{A}^{(1)}_{\mathcal{G}}}= \bigg\{A:=\big(A_{1}(\cdot),A_{2}(\cdot)\big) \in \mathcal{A}_{\mathcal{G}}:\,\,{ A'_{1}(1+{\rho})=A'_{2}(1-{\rho}).\frac{(1-{\rho})^2}{(1+{\rho})^{2}}}, \text{for }|\rho|<1\bigg\}.\s\s\label{class21}
\eeqn
The second proposal in (\ref{bivnorm})  $ A_{1}(x)=A_{2}(x)=\frac{1}{x},$ belongs to this class ${\mathcal{A}^{(1)}_{\mathcal{G}}}$ hence it is first order matching by Theorem \ref{uni} provided we can verify its conditions.

To verify Assumption~\ref{ext2} note that
\beqn
J^{A}_{n}(\X,\hat{\rho}_{n})- J^{A}(\rho_{{0}},\hat{\rho}_{n})&=&[A'_{1}(S_{1}) - A'_{1}(1+\rho_{0})]\frac{S_{1}}{1+\hat{\rho}_{n}}\non\\&&+A'_{1}(1+\rho_{0})\bigg[\frac{S_{1}-(1+\rho_{0})}{1+\hat{\rho}_{n}}\bigg]+[A'_{2}(S_{2})-A'_{2}(1-\rho_{0})]\frac{S_{2}}{1-\hat{\rho}_{n}}\non\\&&+\frac{A'_{2}(1-\rho_{0})}{1-\hat{\rho}_{n}}[S_{2} - (1-\rho_{0})]\non
\eeqn
and 
\beqn
&&J^{A(1)}_{n}(\X,\rho)- J^{A(1)}(\rho_{{0}},\rho)=[A'_{1}(S_{1}) - A'_{1}(1+\rho_{0})] \frac{S_{1}}{1+\hat{\rho}_{n}}+A'_{1}(1+\rho_{0})\times\non\\&& \bigg[\frac{S_{1}-(1+\rho_{0})}{1+\hat{\rho}_{n}}\bigg]+[A'_{2}(S_{2})-A'_{2}(1-\rho_{0})]\frac{S_{2}}{1-\hat{\rho}_{n}}+\frac{A'_{2}(1-\rho_{0})}{1-\hat{\rho}_{n}}[S_{2} - (1-\rho_{0})].\non
\eeqn
Since $S_{1}, S_{2}$ both converge to $(1+\rho_{0}), (1-\rho_{0})$ respectively; using smoothness of $A_{1},A_{2}$ by applying Delta method and Slutsky's theorem one can show that $\sqrt{n}[J^{A}_{n}(\X,\hat{\rho}_{n})- J^{A}(\rho_{{0}},\hat{\rho}_{n})]$ $\big($similarly for $\sqrt{n}(J^{A(1)}_{n}(\X,\rho)- J^{A(1)}(\rho_{{0}},\rho))\big)$ is $O_{{P_{\rho_{0}}}}(1).$ 

Next we will consider whether our \DGE is also second order matching.
To verify Assumption~\ref{as4} note that for $i=0,1$ each of the four terms of $\sqrt{n}[J^{A(i)}_{n}(\X,\hat{\rho}_{n})- J^{A(i)}(\rho_{{0}},\hat{\rho}_{n})]$ by Slutsky's theorem converges to normal with mean zero. Since $S_{1}$ is chi-square so using its exponential concentration property one can prove uniform integrability of each of those terms. So $E_{\rho_{0}}\sqrt{n}[J^{A(i)}_{n}(\X,\hat{\rho}_{n})- J^{A(i)}(\rho_{{0}},\hat{\rho}_{n})]$ asymptotically will converge to the mean of its weak limit which is 0.  Next we compute $\Delta_{2}(\mathcal{G})$ in (\ref{secondord}). 

The last term  in (\ref{secondord}) is $0$ since both of $a_{1}(\rho_{0})=0, a_{2}(\rho_{0})=0.$ Surprisingly, a straightforward calculus exercise reveals that $\Delta_{2}(\mathcal{G})=0$ for any  $A\in {\mathcal{A}^{(1)}_{\mathcal{G}}}$ and consequently 
\[
P_{{\rho_{{0}}}}\bigg[\rho_{{0}}\le \rho^{1-\alpha}(\mathcal{G_{A}},\X,n)\bigg]- \big(1-\alpha\big) = o\bigg(\frac{1}{n}\bigg).
\]

\item $N(\mu,\mu^{q})$ for $\mu>0.$ The data generating equation is $$(S_{1},S_{2})=(\bar{X}_{n},S_n):=\bigg(\mu+\mu^{q/2} Z,\mu^{q/2} U^{1/2}\bigg).$$ Again for $A\in \mathcal{A}_{\mathcal{G}},$ the transformed Jacobian will be
\[
 J^{A}_{n}(\X,\mu)=A'_{1}(\bar{X}_{n})\left| 1+\frac{q(\bar{X}_{n}-\mu)}{2\mu}\right| +A'_{2}(S_{n})\frac{q S_{n}}{2\mu} \longrightarrow  J^{A}(\mu_{0},\mu) := A'_{1}(\mu_{0})\left| 1+\frac{q(\mu_{0}-\mu)}{2\mu}\right| +\frac{q}{2}A'_{2}(\mu_{0}^{\frac{q}{2}})\frac{\mu_{0}^{\frac{q}{2}}}{\mu}
\]
Using $ J^{A}(\mu_{0},\mu)$ in the equation  $\Delta_{1}=0,$ one has the following characterization 
\[{\mathcal{A}^{(1)}_{\mathcal{G}}= \big\{A=(A_{1},A_{2})\in \mathcal{A}_{\mathcal{G}}:  A'_{2}(\mathbf{x^{\frac{q}{2}}})=A'_{1}(\mathbf{x}).q\mathbf{x^{\frac{q}{2} - 1}},\text{ for }x> 0\big\}}\s\]
 which is satisfied by our third choice$A_{1}(\mathbf{x})=\mathbf{x}^{2}, A_{2}(\mathbf{y})=q\mathbf{y}^{2}$ in  (\ref{ex2}).

Similar to the arguments for the bivariate normal cases one can argue that assumptions of Theorem~\ref{uni} hold and $a_{1}(\mu_0), a_{0}(\mu_{0})$. Again some straightforward calculus reveals that for any $A\in {\mathcal{A}^{(1)}_{\mathcal{G}}}$ 
\[
\Delta_2(\mathcal{G})= \frac{q(q-2)\mu_{0}^{q+2}\lt(2\mu^{2}_{0}+\mu^{q}_{0}q(q-1)\rt)}{\lt(2\mu^{2}_{0}+\mu^{q}_{0}q^{2}\rt)^{3}}
\]
and consequently 
\beqn
P_{{\mu_{{0}}}}\bigg[\mu_{{0}}\le \mu^{1-\alpha}(\mathcal{G_{A}},\X,n)\bigg]- \big(1-\alpha\big) =
\frac{z_\alpha\phi(z_\alpha)\Delta_2(\mathcal{G})}{n}+o\bigg(\frac{1}{n}\bigg).
\eeqn
We plot the contour plot of $\Delta_2(\mathcal{G})$ as a function of $\mu$ and $q$ in Figure~\ref{fig1}.
The plot shows that for most values of $q$ the $\Delta_2(\mathcal{G})$ takes on small positive values. This corresponds to second order conservative coverage. When $q$ and $\mu$ are both small, the behavior of $\Delta_2(\mathcal{G})>0$ is very erratic with potential for both positive and negative relatively large values.
\begin{figure}[t!]\centering
\includegraphics[width=90mm]{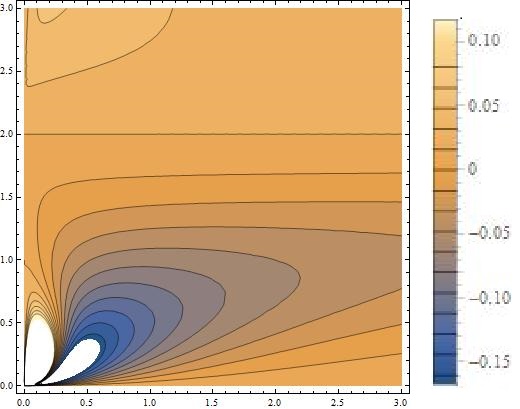}
\caption{Contour plot of $\Delta_2(\mathcal{G})$ with range of $\mu$ (horizontal axis) and $q$ (vertical axis) both $\in (0,3)$ for $N(\mu,\mu^{q})$ example. \label{fig1}
}
\end{figure}

\end{enumerate}
\begin{remark}\label{R4}
This suggested framework is oriented on the data generating equation based on minimal sufficient statistics. One can try transformations based on the simple data generating equation. Further characterizations are possible by imposing two different transformations on two disjoint parts of the data and then getting higher order conditions on the transformations based on the probability matching results.
\end{remark}

\section{Small Sample Simulation}
In this section we describe a result of a small scale simulation study used to evaluate a small sample performance of the various choices of the data generating. We consider two basic examples introduced in Section~\ref{s:examples}; the Bivariate normal  N{\tiny $\left(\begin{pmatrix} 0\\0\end{pmatrix}
           , \begin{pmatrix}
              1 &\ \rho\\
              \rho & 1
              \end{pmatrix}\right)$} example due to Basu \cite{Basu1964} and $N(\mu,\mu^{q})$ with $\mu>0$ unknown and $q=1$ known.

We  compare the following four procedures (listed in the order they appear in the tables): simple fiducial distribution based on modeling the observed data directly (FS), first order matching Fiducial distribution (F1), Posterior based on Jeffrey's prior (BJ), and the second order matching prior (B2) \cite{ong2010data}. All simulation results are computed from 5000 replications.

Tables~\ref{t:rhoOne},~\ref{t:rhoTwo},~\ref{t:rhoLength},~\ref{t:rhoMAD} report simulation results for the Bivariate normal  N{\tiny $\left(\begin{pmatrix} 0\\0\end{pmatrix}
           , \begin{pmatrix}
              1 &\ \rho\\
              \rho & 1
              \end{pmatrix}\right)$} example due to Basu \cite{Basu1964}.
We select  $\rho\in (0.05,0.25,0.5,0.75,0.9)$ and $n=2,3, 4, 5, 10, 100$.  To speed up computations we used$\frac{1}{n}\sum_{i=1}^{n}X_{i}Y_{i}$ instead of the MLE in the second order Bayes method. Since $$\sqrt{n}\lt(\frac{1}{n}\sum_{i=1}^{n}X_{i}Y_{i}-\rho\rt)\sim N(0, \frac{\rho^2 +1}{n}),$$ so using $\frac{1}{n}\sum_{i=1}^{n}X_{i}Y_{i}$ instead of the $\hat{\rho}_{n}$ will have an effect of an order $O(e^{-an})$ which is acceptable for analysis of any polynomial order asymptotics.

Tables~\ref{t:qOne},~\ref{t:qTwo},~\ref{t:qLength},~\ref{t:qMAD} report simulation results for the $N(\mu,\mu^{q})$ with $\mu>0$ unknown and $q=1$.
We select  $\mu\in(0.1,0.5,1,3,5)$ and $n=2,3, 4, 5, 10, 100$. 

Tables~\ref{t:rhoOne}, \ref{t:qOne} report the empirical coverage of the one sided confidence limits for $\alpha=0.025,0.05,0.5,0.95,0.975$. 
Tables~\ref{t:rhoTwo}, \ref{t:qTwo} and \ref{t:rhoLength}, \ref{t:qLength} report the coverage and average length (respectively) of the 
 $95\%$ and $90\%$ two sided equal tailed confidence interval.
Finally, Tables~\ref{t:rhoMAD}, \ref{t:qMAD} report the mean absolute deviation of the point estimators (median) based on each of the distributions (FS, F1, BJ, B2).

The results exhibit mixed behavior of the methods. We observe that FS, the simple fiducial that is not first order matching, was the only method that maintained at least the stated coverage throughout the experiments. The other three methods were sometimes suffering under-coverage for the smallest sample sizes. Higher order matching fiducial gave the shortest confidence intervals for the correlation coefficient problem while the Jeffreys prior gave the shortest intervals (followed closely by the higher order fiducial) for the scaled normal. Overall, even though higher order methods are theoretically superior, it was not clear based on our simulation whether they provide any benefits for smallest sample sizes ($n=2, 3, 4, 5$). More research on this topic will need to be done.

\section{Conclusion and Open questions}

The study of non-informative priors and related variants of distributional inference has a long history and perhaps the best conclusion was made by Kass,Wasserman \cite{kass1996selection}:
\begin{quote}
{``$\ldots$ research on priors chosen by formal rules are serious and may not be dismissed lightly: When sample sizes are small (relative the number of parameters being estimated), it is dangerous to put faith in any default solution; but when asymptotics take over, Jeffreys rules and their variants remain reasonable choices."}
\end{quote}

In this manuscript we found Generalized Fiducial Distribution for $\Theta$ which are (when they are not exact) at least first order matching and sometimes even second order matching even if Jeffreys prior is only first order matching (the bivariate normal example). 
We have conducted a small scale simulation study that showed acceptable performance.

Following are some other open directions left for future work.
\begin{enumerate}
\item \textbf{Non-regular cases:} When the true distribution is supported on $(a(\theta),b(\theta))$ with $|a'(\theta)|\le|b'(\theta)|$ (for example $U(\theta,\theta^{2})$ for $\theta>1$) then the condition ``$\{x:f(x\mid\theta)>0\}$ doesn't depend on $\theta$"  of Theorem \ref{uni},  gets violated and the expansion of the fiducial distribution will not converge to Normal distribution anymore. For $U(a(\theta),b(\theta))$ the fiducial distribution of $\theta$ on the basis of $n$ iid observations  (under assumption both $a(\theta)$ and $b(\theta)$ are increasing and continuous in $\theta$)  is
$$f(\theta\mid \X)\propto \frac{a'(\theta)-a(\theta)[\log b(\theta)]'+\bar{X}_{n}[\log b(\theta)]'}{b(\theta)^{n}}.1_{\{a(\theta)-b(\theta)< X_{(1)},\s a(\theta)+b(\theta)>X_{(n)}\}}$$
where $(X_{(1)},X_{(n)})=(\min \X,\max \X),\bar{X}_{n}=\text{mean}(\X).$   Even in probability matching prior context under a more restrictive condition. We expect a similar result to hold in Fiducial context but with a change that should come from the terms similar to $(W_{n}^{(i)}(\X))^{2}.$
\item \textbf{Multi-parameter context:} Proving analogue version of Theorem \ref{uni} in multi-parameter cases where there is only one parameter of interest and rest are nuisance  is more involved. Generally the Jacobian becomes a \textit{U-Statistics}.  Since higher order expansion of fiducial quantile requires  convergence of fluctuation of scaled Jacobian (like Assumption \ref{as4}), deriving concentration properties of  U-Statistics (that the Jacobian for multi-parameter case resembles) is essential  and challenging.
\end{enumerate}

\appendix

\section{Assumptions}\label{s:assumptions}
We start by reviewing some standard notation used below. We then state the assumptions. In the next subsection we discuss how these can be verified.

We consider one dimensional parameter space $\Theta$ containing true $\theta_{0}$. Define $ l(\theta\mid X_i) = \log f(X_{i},\theta)$ as the log-likelihood of $\theta$ given one sample point $X_i$. Denote $l^{(m)}(\theta\mid X_1)$ as the $m$-th derivative of the log likelihood function $l(\theta\mid X_1)$ with respect to $\theta$. Define $L_{n}(\theta):=\frac{1}{n}\sum_{i=1}^{n}l(\theta | X_{i})$ as the likelihood of $\theta$ given $\X$ (which is scaled by $\frac{1}{n}$) and $c:=-\frac{1}{n}\sum_{i=1}^{n} \frac{\partial^2 l(\theta|X_{i})}{\partial \theta^2}.$ We denote the quantity $\sqrt{nc}(\theta-\hat{\theta}_{n})$ by $y$. Note that  $c=L_{n}^{(2)}(\hat{\theta})$, and define  $ a = L_{n}^{(3)}(\hat{\theta}),\quad a_{4} = L_{n}^{(4)}(\hat{\theta})$,   where $\hat{\theta}$ is the maximum likelihood estimate of $\theta_{0}$ (or any solution of $L'_{n}(\theta)=0$). Denote $ \sqrt{nc}(\theta - \hat{\theta})$ by $y$ whose Fiducial expansion will be needed for asymptotic analysis. From now on $\phi(x)$ will denote the density of the Gaussian distribution function (i.e $\frac{1}{\sqrt{2 \pi e}} e^{-\frac{x^{2}}{2}}$). For sake of generality from this section onwards by $J_{n}(\X,\theta)$ we denote \textit{any Jacobian} that appears in the Generalized Fiducial distribution driven by the corresponding data generating equation $\mathcal{G}$.   For $m\ge 1,$ we denote $\frac{\partial J_{n}(X,\theta)}{\partial \theta}\big|_{\theta =\hat{\theta}_n}, \frac{\partial^2 J_{n}(X,\theta)}{\partial \theta^2}\big|_{\theta =\hat{\theta}_n}$,$ \frac{\partial^m J_{n}(X,\theta)}{\partial \theta^m}\big|_{\theta =\hat{\theta}_n}$ by $J'_n(X,\hat{\theta}_{n}),J''_n(X,\hat{\theta}_{n})$, and $J^{(m)}_n(X,\hat{\theta}_{n})$ respectively. We know by virtue of SLLN  pointwise for each $\theta,\s$ $m$-th derivative (w.r.t $\theta$) of the simple Jacobian $J^{(m)}_n(X,\theta)$ in (\ref{jacobian}) scaled by $\frac{1}{n}$ converges to $J^{(m)}(\theta_0,\theta):= \frac{\partial^{m} E_{\theta_{0}}[J(\X,\theta)]}{\partial \theta^m}$ almost surely as $n\to \infty$.
Finally denote $g_1:= J^{(1)}(\theta_0,\theta_0),\s g_2:=J(\theta_{0},\theta_{0})$.
\bigskip

\bas\label{str}
\begin{enumerate}[(a)]
\item The distribution $F(\cdot | \theta)$ are distinct for $\theta\in\Theta$.
\item The set $\{x:f(x |\theta)>0\}$ is independent of the choice of $\theta$.
\item The data $\X=\{ X_1,\ldots,X_{n}\}$ are iid with probability density $f(x |\theta).$
\item There exists $m\ge 1$, such that in a neighborhood $B(\theta_{0},\delta)$ of the true value $\theta_0$, all possible $(m+3)$ ordered partial derivatives $\frac{\partial^{m+3} f(x |\theta)}{\partial \theta^{m+3}}$ exist. For all $i=1,\ldots,m+2;$ the quantities $E_{\theta_{0}}l^{(i)}(\theta_{0}| X_{i})$ are all  finite.
\item There exists a function $M(x)$ such that 
$$\sup_{\theta \in B(\theta_0,\delta)}\big|\frac{\partial^{(m+3)} }{\partial \theta^{(m+3)}}\log f(x\mid \theta)\big|\le M(x)\s\text{and} \s E_{\theta_0}M(X)<\infty.$$ 
\item The information $I(\theta)$ is positive for all $\theta\in B(\theta_{0},\delta)$
\end{enumerate}
\eas

\bas\label{ext}
\begin{enumerate}[(a)]
\item for any $\delta>0,$ there exists an $\epsilon>0$ such that
$$P_{\theta_{0}}\left\{\sup_{\theta \in B(\theta_0,\delta)^{c}}[  L_{n}(\theta) - L_{n}(\theta_0)] \le -\epsilon  \right\}\to 1\s \text{as} \s n\to \infty.$$
\item Let $\mb x=(x_{1},\ldots,x_{n}).$ There exists $s\in\mathbb{N},$ such that $J_{n}(\mb x,\theta)=\sum^{n}_{i=1}J_{i}(\mb x,\theta),$ where for all $\mb x \in \R^n,$ $J_{i}(\mb x,\theta)$ satisfies $$\sup_{i=1,\ldots,n}n^{-s}\int_{\R}J_{i}(\mb x,\hat{\theta}_{n}+\frac{y}{\sqrt{nc}})f_{}(x_{i},\hat{\theta}_{n}+\frac{y}{\sqrt{nc}})dy\,\, <\,\,\infty \s\s\text{a.s}\s  P_{\theta_{0}}.$$

\item The density $f$ satisfies the following property: There exists a constant $c\in [0,1)$
\beqn 
\inf_{\theta \in B(\theta_0,\delta)^{c}} \frac{\min_{i=1,\ldots ,n} \log f(X_{i},\theta)}{n[L_{n}(\theta)-L_n(\theta_0)]} \xrightarrow{P_{\theta_{0}}} c.\label{ex}
\eeqn
\end{enumerate}
\end{Assumption}

\bas\label{ext2}  There exists a function $J(\cdot,\cdot):\Theta\times \Theta \to \R$ with its $i$-th derivative with respect to second argument  $\frac{\partial^{i} J(\theta_1,\theta)}{\partial^{i} \theta} ,$ denoted by $J^{(i)}(\theta_1,\theta)$  (where $J^{(0)}(\theta_1,\theta):=J(\theta_1,\theta));$ such that following conditions hold.
\begin{enumerate}[(a)]
\item There exists $m\ge 1,$ such that for each $i=0,\ldots,m+1$ the Jacobian $J_{n}^{(i)}(\X,\theta),$ satisfies
\beqn\label{wald}
\sup_{\theta \in B(\theta_0,\delta)}\bigg|J^{(i)}_{n}(\X,\theta)- J^{(i)}(\theta_0,\theta)\bigg| \to 0 \s\s a.s. \s P_{\theta_0}.
\eeqn

Namely a uniform convergence result holds over a neighborhood of true parameter value $\theta_0$ for each $ i=0,\ldots,m+1$ uniformly as $n\to\infty$.  
\item The function $J(\cdot,\theta)$ doesn't vanish in $\theta \in B(\theta_0,\delta)$ for any $\delta>0$.
\item For $i=0,\ldots,m+1$ the quantities
\beqn
\sqrt{n}\lt[J_{n}^{(i)}(\X,\hat{\theta}_{n}) - J^{(i)}(\theta_{0},\hat{\theta}_{n})\rt]&=&O_{P_{\theta_{0}}}(1).
\eeqn
\end{enumerate}
\eas

\bas\label{as4}
There exists $m\ge 0,$ for which following hold:
\begin{enumerate}[(a)]
\item \textbf{Integrability Condition:} For $i=0,\ldots,m$ and any $\delta>0,$ one has for all $\theta\in\Theta$
$$E_{\theta}\lt[n\lt[J_{n}^{(i)}(\X,\hat{\theta}_{n}) - J^{(i)}(\theta_{},\hat{\theta}_{n})\rt]^{2}\rt]=O(1)$$ 
where the finite constant may depend on $\theta.$
\item For $i=0,\ldots,m$ there exist continuous functions $a_{i}(.)$ such that
\beqn
a_{i}(\theta_{}):=\lim_{n\to \infty}E_{\theta_{}}\sqrt{n}\lt[J_{n}^{(i)}(\X,\hat{\theta}_{n}) - J^{(i)}(\theta_{},\hat{\theta}_{n})\rt].
\eeqn
\item  For $i=0,\ldots,m$ the functions $J^{(i)}(\theta_{0},\cdot)$ are locally Lipschitz.
\end{enumerate}
\eas

\subsection{Remarks}
Here we state several remarks discussing the assumptions above.

\begin{remark} 
\begin{enumerate}[(a)]
\item Assumption \ref{ext}(b) can be verified with $s=0$ for the simple Jacobian structure of the form in (\ref{examplejac}) by taking $J_{i}(\mb x,\theta):=\bigg|w(x_{i}) \frac{\frac{\partial F_{i}(x_{i},\theta)}{\partial \theta}}{f_{}(x_{i},\theta)}\bigg|.$ Since 
\beqn
\int_{\R}J_{i}(\X,\theta)f(X_i,\theta)d\theta= \big|w(X_i)\big| \lt[ F_{i}(X_i,\infty)-F_{i}(X_i, -\infty)\rt]<\infty \s\s \text{a.s}\s  P_{\theta_{0}}. 
\eeqn
Note that any polynomial exponent of $n$ can replace the condition``$n^{s}$ for some $s>0$."
\item If $X_{1},X_{2},\ldots,X_{n}$ are iid realizations from density $f(\cdot|\theta_0)$ then both the numerator and denominator of the left hand side of (\ref{ex}) converge to $-\infty$ with rate $-C_1 \log n$ and $-C_2 n$ respectively. So in that case Assumption \ref{ext}(c) is strongly implied by $c=0$  if $\frac{C_1}{C_2}$ is uniformly bounded for $n\ge 1.$ 
\end{enumerate}
\end{remark}

\begin{remark}
Following comments are on Assumption \ref{ext2}:
\begin{enumerate}[(a)]
\item In (\ref{wald}) a difference from Bayesian paradigm is the extra (Assumption \ref{ext2} for $i=(m+1)$th order) smoothness condition for data dependent $J_{n}(\X,\theta)$ which is needed to apply the uniform law of large number in a neighborhood of $\theta_{0}$.
\item It follows from Wald's theorem that Assumption \ref{ext2}(a) holds for the simple Jacobian in (\ref{examplejac}) if following are satisfied for each $i=0,\ldots,m+1:$
\begin{enumerate}[(1)]
\item For each $x$,  $J^{(i)}(x,\cdot)$ is continuous in $\theta\in[\theta_0 -\delta,\theta_0 +\delta]$.
\item For each $\theta\in[\theta_0 -\delta,\theta_0 +\delta]$, $J^{(i)}(\cdot,\theta)$ is a strictly positive measurable function of $x$.
\item There exists a $\delta>0$ such that,  $$E_{\theta_{0}}\left(\sup_{\theta \in B(\theta_0,\delta)}\bigg|J^{(i)}(X_1,\theta)\bigg|\right)< \infty.$$
\end{enumerate}
\end{enumerate}
\end{remark} 

\begin{remark}
Following are some remarks on Assumption \ref{as4}:
\begin{enumerate}[(a)]
\item  Assumption \ref{as4}(a) is stronger than Assumption \ref{ext2}(c). We mentioned the latter to emphasis on the fact that it is sufficient for only first order term.
\item Note that in all situations where $\X$ is a collection of $n$ random samples, usually we have $a_{1}(\theta_{0}) - a_{2}(\theta_{0})\frac{\partial}{\partial\theta}\log J(\theta_0,\theta)=0$ because of arguments similar to contiguity. In that case $\Delta_1(\mathcal{G}) $ and the second term of $\Delta_2(\mathcal{G})$ are both first and second order terms for the asymptotic expansion of $P_{\theta_{0}}\lt[\theta_{0}\le \theta^{1-\alpha}(\pi,\X)\rt]$ where $\theta^{1-\alpha}(\pi,\X)$ is the $(1-\alpha)$th Posterior quantile based on the prior :$$\pi(\cdot) \propto J(\theta_{0},\cdot) \s\s\text{where $\theta_0$ is the true parameter value}.$$  
\end{enumerate}
\end{remark}

\begin{remark}[Higher order expansions] In general we have $J(\theta_{0},\theta)$ to be $\lim_{n\to \infty}E_{\theta_{0}}J_{n}(\X,\theta)$, implying $a_i(\theta_0)$ will be $0$ for $i=0,1$. But  in Theorem \ref{uni} we kept it general since data generating structural equation is not-unique. So conditions for the first two order terms really will not differ from the conditions in \textbf{probability matching priors}. We will remark about higher (third) order term and sketch the . Along with Assumptions \ref{str}-\ref{as4} suppose further  following two assumptions hold with $m=3$:
\begin{enumerate}[(1)]
\item  For $i=0,\ldots,m$ and any $\delta>0,$ one has for all $\theta\in\Theta$
$$\lim_{n\to \infty}n^{\frac{3}{2}}E_{\theta}\lt[\lt[J_{n}^{(i)}(\X,\hat{\theta}_{n}) - J^{(i)}(\theta_{},\hat{\theta}_{n})\rt]^{3}\rt]<\infty$$ 
where the finite constants may depend on $\theta.$
\item Define the following quantities given they exist: \beqn
a_{i}^{(1)}(\theta_{0})&:=&\lim_{n\to \infty}nE_{\theta_{0}}\lt[J_{n}^{(i)}(\X,\hat{\theta}_{n}) - J^{(i)}(\theta_{0},\hat{\theta}_{n})\rt]^{2},\non\\a_{0,1}^{()}(\theta_{0})&:=&\lim_{n\to \infty}nE_{\theta_{0}}\lt[J_{n}^{}(\X,\hat{\theta}_{n}) - J^{}(\theta_{0},\hat{\theta}_{n})\rt]\lt[J_{n}^{(1)}(\X,\hat{\theta}_{n}) - J^{(1)}(\theta_{0},\hat{\theta}_{n})\rt].\non
\eeqn
\end{enumerate}
Then analogue to the Theorem \ref{uni} a \textbf{third order representation} holds:
\beqn
P_{\theta_{0}}\bigg[\theta_{0}\le \theta^{1-\alpha}(\mathcal{G},\X,n)\bigg]- \big(1-\alpha\big) &=& 
\frac{c_{1}\Delta_1(\mathcal{G})}{\sqrt{n}}+\frac{c_{2}\Delta_2(\mathcal{G})}{n}+\frac{c_{3}\Delta_3(\mathcal{G})}{n^{3/2}}+o\bigg(\frac{1}{n^{3/2}}\bigg),\,\,\,\,\,\,\s\,\,\non\\
\text{where }\,\,\,\Delta_3(\mathcal{G}) =\bigg[\frac{a_{i}^{(1)}(\theta_{0})g_1}{g^{2}_{3}} -\frac{a_{0,1}^{()}(\theta_{0})}{g_2}\bigg]&+&I^{-\frac{1}{2}}_{\theta_{0}}\bigg[\frac{a_{1}(\theta_{0})}{z } - \frac{a_{2}(\theta_{0})g_1}{z g_{2}}\bigg]+\bigg\{\text{Third order }\non\\ \text{Probability matching prior term with prior}& J(\theta_{0},\theta)&\text{ at $\theta=\theta_{0}$}\bigg\}.\,\,\,\,\non
\eeqn
Define 
\[\W^{(m)}_{n}(\X) := \sqrt{n}\lt(\frac{J^{(m)}_{n}(\X,\hat{\theta}_{n})}{J_{n}(\X,\hat{\theta}_{n})}-  \frac{J^{(m)}(\theta_0,\hat{\theta}_{n})}{J(\theta_0,\hat{\theta}_{n})}\rt).\].

This extra additive quantity $ \bigg(\frac{a_{i}^{(1)}(\theta_{0})g_1}{g^{2}_{3}} -\frac{a_{0,1}^{()}(\theta_{0})}{g_2}\bigg)$ in the display of $\Delta_3(\mathcal{G}) $ will come due to the following Taylor's expansion of $\frac{T_1}{T_2}$ around $\frac{g_1}{g_2}$ where $T_1:= J'_{n}(\X,\hat{\theta}_{n}),T_2:=J_{n}(\X,\hat{\theta}_{n})$ and their corresponding limits $g_1:= J'(\theta_{0},\theta_{0}),g_2:=J(\theta_{0},\theta_{0})$
\beqn
\frac{T_1}{T_2}=\frac{g_1}{g_2}&+&(T_1-g_1)\frac{1}{g_2}- (T_2-g_2)\frac{g_1}{g^{2}_{2}}+ \lt((T_2 - g_2)^{2}\frac{g^{}_{1}}{g_{2}^{3}}- (T_1-g_1)(T_2-g_2)\frac{1}{g_{2}}\rt)\label{remeq}\\&+&O\Bigg(\bigg((T_{1} -g_{1})\frac{\partial}{\partial x_1}+(T_{2} -g_{2})\frac{\partial}{\partial x_2}\bigg)^{3}\bigg(\frac{x_{1}}{x_{2}}\bigg)\Bigg|_{x_1\in(T_{1},g_1),x_{2}\in (T_{2},g_2)}\Bigg)\s\s\text{implying}\non\\
E_{\theta_{0}}\lt[W_{n}^{(1)}(\X)\rt]&=& \bigg(\frac{a_{1}(\theta_{0})}{g_{2}}- \frac{a_{0}(\theta_{0})g_{1}}{g^{2}_{2}}\bigg)+ \frac{1}{\sqrt{n}}\bigg[\frac{a_{i}^{(1)}(\theta_{0})g_1}{g^{2}_{3}} -\frac{a_{0,1}^{()}(\theta_{0})}{g_2}\bigg]+O\bigg(\frac{1}{n}\bigg)\label{ett1}.
\eeqn
 keeping an extra order term. Note that for simple data generating equation $J(\theta_{0},\theta)= E_{\theta_{0}}J_{n}(\X,\theta),$ along with the empirical structure $J_{n}(\X,\theta) = \frac{1}{n} \sum_{i=1}^{n}J(X_{i},\theta)$ then one gets $$\bigg[\frac{a_{i}^{(1)}(\theta_{0})g_1}{g^{2}_{3}} -\frac{a_{0,1}^{()}(\theta_{0})}{g_2}\bigg]=\bigg[\frac{Var_{\theta_{0}}(J(X_{1},\theta))g_1}{g^{2}_{3}} -\frac{Cov_{\theta_{0}}(J(X_{1},\theta),J'(X_{1},\theta))}{g_2}\bigg]\bigg|_{\theta=\theta_{0}}$$
which appears as an extra in the third order term. The difference of Fiducial cases will be different from Bayesian paradigm likewise in the further order of terms, starting from $3$rd order due to the presence of $\{W_{n}^{(m)}(\X), m\ge 1\}$ and their respective higher order expansions. 
\end{remark}


\section{Proofs}\label{s:proofs}

We will proceed through a number of steps. Below we will be using notion introduced in previous section. First we prove a lemma on the expansion of the fiducial density and then we will give an asymptotic expansion of the Fiducial quantile in Corollary \ref{Cor1}. After that in order to get the frequentist coverage of the quantile with the obtained expression from Corollary \ref{Cor1}, we will proceed with Shrinkage method. 

The main challenges arise from the fact that there are fiducial specific terms in the expansion of the Jacobian function $J_n(\X,\theta)$. In particular, for some $\theta'\in (\hat{\theta}_{n},\hat{\theta}_{n}+\frac{y}{\sqrt{nc}}),$ 
\beqn
J_{n}(\X,\hat{\theta}_{n} +\frac{y}{\sqrt{nc}})&=& J_{n}(\X,\hat{\theta}_{n})+J'_{n}(\X,\hat{\theta}_{n})\frac{y}{\sqrt{nc}}+J''_{n}(\X,\hat{\theta}_{n})\frac{y^2}{2nc}+J_{n}'''(\X,\theta')\frac{y^3}{6(nc)^{3/2}}\s\s\non\\
&=&J_{n}(\X,\hat{\theta}_{n})\lt[1+\frac{J'_{n}(\X,\hat{\theta}_{n})}{J_{n}(\X,\hat{\theta}_{n})}\frac{y}{\sqrt{nc}}+\frac{J''_{n}(\X,\hat{\theta}_{n})}{J_{n}(\X,\hat{\theta}_{n})}\frac{y^2}{2nc}+\frac{J_{n}'''(\X,\theta')}{J_{n}(\X,\hat{\theta}_{n})}\frac{y^3}{6(nc)^{3/2}}\rt]\non\\
&=&J_{n}(\X,\hat{\theta}_{n}) \lt[1+\frac{J'(\theta_0,\hat{\theta}_{n})}{J(\theta_0,\hat{\theta}_{n})}\frac{y}{\sqrt{nc}}+\frac{1}{n}\lt(\W^{(1)}_{n}(\X)\frac{y}{\sqrt{c}} +   \frac{J''(\theta_0,\hat{\theta}_{n})}{J(\theta_0,\hat{\theta}_{n})}\frac{y^2}{2c}\rt)\rt]\non\\&&+\frac{J_{n}(\X,\hat{\theta}_{n}) }{n^{3/2}}\lt(\W^{(2)}_{n}(\X)\frac{y^2}{2c} +\frac{J_{n}'''(\X,\theta')}{J_{n}(\X,\hat{\theta}_{n})}\frac{y^3}{6(c)^{3/2}}\rt),\label{jaco}
\eeqn
where
\[\W^{(m)}_{n}(\X) := \sqrt{n}\lt(\frac{J^{(m)}_{n}(\X,\hat{\theta}_{n})}{J_{n}(\X,\hat{\theta}_{n})}-  \frac{J^{(m)}(\theta_0,\hat{\theta}_{n})}{J(\theta_0,\hat{\theta}_{n})}\rt).\]

 In order to lessen notational burden we also denote
\[K(\theta_0,\X,y):=\lt[1+\frac{J'(\theta_0,\hat{\theta}_{n})}{J(\theta_0,\hat{\theta}_{n})}\frac{y}{\sqrt{nc}}+\frac{1}{n}\lt(\W^{(1)}_{n}(\X)\frac{y}{\sqrt{c}} +   \frac{J''(\theta_0,\hat{\theta}_{n})}{J(\theta_0,\hat{\theta}_{n})}\frac{y^2}{2c}\rt)\rt].\]

Under the Assumption \ref{str} with $m=2$ one has the following expansion ( Consequence of Taylor's theorem) holds for some $\theta'\in (\hat{\theta}_{n},\hat{\theta}_{n}+\frac{y}{\sqrt{nc}}):$
\beqn
n[L_{n}(\hat{\theta}_{n} +\frac{y}{\sqrt{nc}}) - L_{n}(\hat{\theta}_{n})] &=& - \frac{y^2}{2}+ \frac{1}{6}\frac{y^3 L_{n}^{(3)} (\hat{\theta}_{n})}{\sqrt{n}c^{\frac{3}{2}}}+\frac{1}{24}\frac{y^4}{nc^2}L_{n}^{(4)}(\hat{\theta}_{n})+\frac{1}{120}\frac{y^5}{n^{3/2}c^2}L_{n}^{(5)}(\theta')\nonumber\\
&:=&- \frac{y^2}{2}+ R_{n}(\hat{\theta}_{n})+\frac{1}{120}\frac{y^5}{n^{3/2}c^2}L_{n}^{(5)}(\theta'),\label{likexp}
\eeqn
where $R_{n}(\theta):=\frac{1}{6}\frac{y^3 L_{n}^{(3)} (\theta)}{\sqrt{n}c^{\frac{3}{2}}}+\frac{1}{24}\frac{y^4}{nc^2}L_{n}^{(4)}(\theta).$

\begin{Lemma}\label{lem1}Suppose Assumptions  \ref{str},\ref{ext},\ref{ext2} hold with $m=2.$ Following quantity is 
\beqn
 I_{\R}&:=&n\int_{\R}\Bigg|J_{n}(\mathbf{X},\hat{\theta}_{n} +\frac{y}{\sqrt{nc}})e^{n[L_{n}(\hat{\theta}_{n} +\frac{y}{\sqrt{nc}})-L_{n}(\hat{\theta}_{n} )] }  \non\\ &&-J_{n}(\mathbf{X},\hat{\theta}_{n})e^{-\frac{y^2}{2}}\lt(1+R_{n}(\hat{\theta}_{n})+\frac{R_{n}(\hat{\theta}_{n})^2}{2}\rt) K(\theta_0,\X,y)\Bigg|dy = o_{P_{\theta_0}}(1).
\eeqn
\end{Lemma}
\begin{proof}
We will proceed traditionally by breaking the integral in three disjoint regions. Denoting $I_\R$ as the integral appeared in the left hand side of the Lemma,  we have 
\beqn\label{dec}
I_\R\le I_{A_1}+I_{A_2}+I_{A_3}
\eeqn
where $A_1=\{y:|y|<C \log\sqrt{n}\}, A_2 = \{y:C \log\sqrt{n}\le|y|\le \delta \sqrt{n}\}, A_3=\{y:|y|>\delta \sqrt{n}\}.$ The choice of $C,\delta$ will be specified later. The third term of (\ref{dec}) can be written as  $I_{A_3}\leq  I_{A_3}^{1}+ I_{A_3}^{2}$ where
\beqn
I_{A_3}^{1} &=& n\int_{A_3}J_{n}(\mathbf{X},\hat{\theta}_{n} +\frac{y}{\sqrt{nc}})e^{n[L_{n}(\hat{\theta}_{n} +\frac{y}{\sqrt{nc}})-L_{n}(\hat{\theta}_{n})]}dy,\non\\ I_{A_3}^{2}&=& n\int_{A_3}J_{n}(\mathbf{X},\hat{\theta}_{n})e^{-\frac{y^2}{2}}\lt(1+R_{n}(\hat{\theta}_{n})+\frac{R_{n}(\hat{\theta}_{n})^2}{2}\rt) K(\theta_0,\X,y)dy. \nonumber
\eeqn
Expanding  $I^{1}_{A_3},$ one gets
\beqn
 I^{1}_{A_3}&=&n\sum_{i=1}^{n}\int_{A_3}J_{i}(\X,\hat{\theta}_{n} +\frac{y}{\sqrt{nc}})f(X_i,\hat{\theta}_n+\frac{y}{\sqrt{nc}})e^{n[L_{n}(\hat{\theta}_{n} +\frac{y}{\sqrt{nc}})-L_{n}(\hat{\theta}_{n})]-\log f(X_{i},\hat{\theta}_n+\frac{y}{\sqrt{nc}})}dy\nonumber\\
&=&n\sum_{i=1}^{n}\int_{A_3}J_i(\X,\hat{\theta}_{n} +\frac{y}{\sqrt{nc}})f(X_i,\hat{\theta}_n+\frac{y}{\sqrt{nc}})e^{n[L_{n}(\hat{\theta}_{n} +\frac{y}{\sqrt{nc}})-L_{n}(\hat{\theta}_{n})]\Big[1-\frac{\log f(X_{i},\hat{\theta}_n+\frac{y}{\sqrt{nc}})}{n[L_{n}(\hat{\theta}_{n} +\frac{y}{\sqrt{nc}})-L_{n}(\hat{\theta}_{n})]}  \Big]}dy\non
\eeqn
\beqn
\le n^{2}\sup_{i=1,\ldots,n}\int_{A_3}J_i(\mb X,\hat{\theta}_{n} +\frac{y}{\sqrt{nc}})f(X_i,\hat{\theta}_n+\frac{y}{\sqrt{nc}})e^{n[L_{n}(\hat{\theta}_{n} +\frac{y}{\sqrt{nc}})-L_{n}(\hat{\theta}_{n})]\Big[1-\frac{\log f(X_{i},\hat{\theta}_n+\frac{y}{\sqrt{nc}})}{n[L_{n}(\hat{\theta}_{n} +\frac{y}{\sqrt{nc}})-L_{n}(\hat{\theta}_{n})]}  \Big]}dy\nonumber
\eeqn 
Notice that from Assumption \ref{ext}(b) $P_{\theta_{0}}$ almost surely $n^{-s}\int_{\R}J_i(\mb X,\hat{\theta}_{n} +\frac{y}{\sqrt{nc}})f(X_i,\hat{\theta}_n+\frac{y}{\sqrt{nc}})dy<\infty$. Now by Assumption \ref{ext} one has the exponential term to decay as $e^{-n(1-c)\epsilon}$ in probability and that term multiplied with $n$ will also goes to $0$ in probability. Rest will follow by dominated convergence theorem. For $I_{A_3}^{2}$ we have $J_{n}(\X,\hat{\theta}_{n}) \to^{P_{\theta_{0}}} J(\theta_0,\theta_0)$  from Assumption \ref{ext2}. The multiplicative parts are  the integrals $\int_{A_3}y^{\alpha} e^{-y^2} dy$ for $\alpha =0,1,2$ which under $A_3$ decays exponentially to $0$ resulting the $P_{\theta_0}$ limit of the second term $0$.

Now consider $I_{A_1}$. Denote $\frac{1}{120}\frac{y^5}{n^{3/2}c^2}L_{n}^{(5)}(\theta')$  by $M_{n}$. The first integral in region $A_1,$ can be bounded by $I_{A_1}\leq I^1_{A_1}+I^2_{A_1},$ 
where 
\beqn
 I^1_{A_1}&:=&n\int_{A_1}J_{n}(\mathbf{X},\hat{\theta}_{n} +\frac{y}{\sqrt{nc}})e^{-\frac{y^2}{2}}\Big|e^{R_n(\hat{\theta}_{n})+M_n}-1-R_{n}(\hat{\theta}_{n})-\frac{R_{n}(\hat{\theta}_{n})^2}{2}\Big|dy\s\s\s\text{and}\non\\ I^2_{A_1}&:=&n\int_{A_1} I(\X,n,y)e^{-\frac{y^2}{2}}\lt(1+R_{n}(\hat{\theta}_{n})+\frac{R_{n}(\hat{\theta}_{n})^2}{2}\rt)dy\non\s\s\s\label{sec}
\eeqn
and  $I(\X,n,y):=\big|J_{n}(\mathbf{X},\hat{\theta}_{n} +\frac{y}{\sqrt{nc}})- J_{n}(\X,\hat{\theta}_{n}) K(\theta_0,\X,y)\big|.$ Note that under $A_1,$ the quantity $$n \lt(M_n+\frac{M^2_{n}}{2}+R_{n}(\hat{\theta}_n)M_n\rt)=O_p(\frac{(\log\sqrt{n})^5}{\sqrt{n}}).$$  Also $L_{n}^{(5)}(\theta')$ is $O_p(1)$ for $\theta' \in (\theta_0-\delta,\theta_0+\delta)$ along with $L^{(3)}_{n}(\hat{\theta}_{n})$ and $L^{(4)}_{n}(\hat{\theta}_{n})$. Since $R_n+M_n$ is $O_p(\frac{\log^{3}(\sqrt{n})}{\sqrt{n}})$. Using the inequality $$ e^x-1-x-\frac{x^2}{2}\le \frac{x^3}{6(1-\frac{x}{4})}\s\text{for}\s x\in(0,4)$$ the first term of (\ref{sec}) can be written as
\beqn
I^1_{A_{1}}&\leq& n\int_{A_1}J_{n}(\mathbf{X},\hat{\theta}_{n} +\frac{y}{\sqrt{nc}})e^{-\frac{y^2}{2}}\Big|e^{R_n+M_n}-1-(R_{n}+M_n)-\frac{(R_{n}+M_n)^2}{2}\Big|dy\non\\&+&n \int_{A_1}J_{n}(\mathbf{X},\hat{\theta}_{n} +\frac{y}{\sqrt{nc}})e^{-\frac{y^2}{2}}\lt(M_n+\frac{M^2_{n}}{2}+R_{n}(\hat{\theta}_n)M_n\rt) dy\nonumber\\
&\le& n\int_{A_1}J_{n}(\mathbf{X},\hat{\theta}_{n} +\frac{y}{\sqrt{nc}})e^{-\frac{y^2}{2}}\Big|\frac{(R_n+M_n)^3}{6(1-\frac{(R_n+M_n)}{4})}\Big|dy+O_p(\frac{(\log\sqrt{n})^5}{\sqrt{n}})\sup_{y\in A_1}J_{n}(\mathbf{X},\hat{\theta}_n+ \frac{y}{\sqrt{nc}}).\nonumber
\eeqn
Now $\sup_{y\in A_1}n(R_n+M_n)^3\le \sup_{y\in A_1}\frac{y^9}{\sqrt{n}}\max{[\frac{L^{(5)}_n(\theta')}{c^{\frac{3}{2}}},\frac{L^{(3)}_n(\hat{\theta}_n)}{c^2}]}\le \frac{\log(\sqrt{n})^9}{\sqrt{n}} O_{p}(1).$
So  we have 
\beqn
I^{1}_{A_1}\le \sup_{y \in A_{1}}J_{n}(\mathbf{X},\hat{\theta}_n+ \frac{y}{\sqrt{nc}})\Big[O_p\left(\frac{\log(\sqrt{n})^9}{\sqrt{n}}\right)\int_{A_1}e^{-\frac{y^2}{2}}dy+O_p(\frac{(\log\sqrt{n})^5}{\sqrt{n}})\Big].
\eeqn

Since $\hat{\theta}_{n}\to \theta_{0}\s a.s,$ under $A_{1},$ we have   $(\hat{\theta}_{n},\hat{\theta}_{n}+\frac{y}{\sqrt{nc}})\subset (\theta_{0}-\delta,\theta_{0}+\delta)$ with  $P_{\theta_0}$ probability $1.$ We have almost surely $ \sup_{y \in A_{1}}J_{n}(\mathbf{X},\hat{\theta}_n+ \frac{y}{\sqrt{nc}})\le \sup_{\theta'\in (\theta_{0}-\delta,\theta_{0}+\delta)}J_{n}(\mathbf{X},\theta')$. From Assumption \ref{ext2} one has almost surely
$$I^{1}_{A_1} \le \sup_{\theta'\in (\theta_{0}-\delta,\theta_{0}+\delta)}J_{n}(\mathbf{X},\theta')\Big[O_p\lt (\frac{\log(\sqrt{n})^9}{\sqrt{n}}\rt )\int_{A_1}e^{-\frac{y^2}{2}}dy+O_p(\frac{(\log\sqrt{n})^5}{\sqrt{n}})\Big]. $$  
Now using Wald's theorem one has $\sup_{\theta'\in (\theta_{0}-\delta,\theta_{0}+\delta)}|J_{n}(\mathbf{X},\theta')- J(\theta_0,\theta')|\to^{a.s}0,$ resulting the following statement almost surely
$$I^{1}_{A_{1}}\le  \sup_{\theta'\in (\theta_{0}-\delta,\theta_{0}+\delta)}J(\theta_0,\theta')\Big[O_p\lt (\frac{\log(\sqrt{n})^9}{\sqrt{n}}\rt )\int_{A_1}e^{-\frac{y^2}{2}}dy+O_p(\frac{(\log\sqrt{n})^5}{\sqrt{n}})\Big].$$
which is $o_{p_{\theta_{0}}}(1).$
It follows that the second term of $I_{A_1}$ 
\beqn
 I_{A_1}^2 &:=&n\int_{A_1}I(\X,n,y)e^{-\frac{y^2}{2}}(1+R_{n})dy\non\\
&\le& \sup_{\theta' \in (\hat{\theta}_{n},\hat{\theta}_{n}+\frac{y}{\sqrt{nc}})}J_{n}(\X,\hat{\theta}_{n})\int_{A_1}\frac{1 }{n^{1/2}}\lt(\W^{(2)}_{n}(\X)\frac{y^2}{2c} +\frac{J_{n}'''(\X,\theta')}{J_{n}(\X,\hat{\theta}_{n})}\frac{y^3}{6(c)^{3/2}}\rt)e^{-\frac{y^2}{2}}(1+R_{n}(\hat{\theta}_{n}))dy.\nonumber
\eeqn

Again similarly using almost sure convergence of the event $(\hat{\theta}_{n},\hat{\theta}_{n}+\frac{y}{\sqrt{nc}})\subset (\theta_{0}-\delta,\theta_{0}+\delta)$ and Assumption \ref{ext2} on $J'''_{n}(\X,\theta')$ we get $I^{2}_{A_1}$ is of $o_{P_{\theta_0}}(1).$

 Next consider the integral $I_{A_2}$, that can bounded above by $I_{A_2}\le I_{A_2}^{1}+ I_{A_2}^{2}$ where
\beqn\label{second}
 I_{A_2}^{1}&:=&n\int_{A_2}J_{n}(\mathbf{X},\hat{\theta}_{n} +\frac{y}{\sqrt{nc}})e^{-\frac{y^2}{2}+R_n+M_n}dy,\non\\  I_{A_2}^{2}&:=& n\int_{A_2}J_{n}(\mathbf{X},\hat{\theta}_{n})e^{-\frac{y^2}{2}}\lt(1+R_{n}(\hat{\theta}_{n})+\frac{R_{n}(\hat{\theta}_{n})^2}{2}\rt) K(\theta_0,\X,y)dy. \nonumber\\
\nonumber
\eeqn
Consider the term $I_{A_2}^{2}$. Note that under $A_2$ for $n\ge e^4,$ $(\log\sqrt{n})^2>\log n$ and also using the fact $\frac{|y|}{\sqrt{n}}<\delta$ the quantity 
\beqn
R_n(\hat{\theta}_n) &\le& \frac{1}{6}\delta^{3} n \frac{L_{n}^{(3)}(\hat{\theta}_n)}{c^{3/2}}+\frac{1}{24}\delta^{4} n\frac{L^{(4)}_{n}(\hat{\theta}_n)}{c^2}=O_{P_{\theta_0}}(n)\nonumber
\eeqn
resulting $\lt(1+R_{n}(\hat{\theta}_{n})+\frac{R_{n}(\hat{\theta}_{n})^2}{2}\rt)$ is $O_{P_{\theta_0}}(n^2)$. Also from Assumption \ref{ext2} the remaining term $J_{n}(\mathbf{X},\hat{\theta}_{n})K(\theta_0,\X,y)$ is $O_{P_{\theta_0}}(1).$ So the upper bound of the second integral $I_{A_2}^{2}$ is bounded by
\beqn
 &&O_{P_{\theta_0}}(n^3)J_{n}(\mathbf{X},\hat{\theta}_{n})\bigg[1+\frac{J'(\theta_0,\hat{\theta}_{n})}{J(\theta_0,\hat{\theta}_{n})}\frac{\delta}{\sqrt{c}}\non\\&+&\frac{1}{n}\lt(\W^{(1)}_{n}(\X)\frac{\delta \sqrt{n}}{\sqrt{c}} +   \frac{J''(\theta_0,\hat{\theta}_{n})}{J(\theta_0,\hat{\theta}_{n})}\frac{n\delta^2}{2c}\rt)\bigg]. e^{-\frac{C^2}{2}\log n}\left[\delta\sqrt{n} - C\log\sqrt{n}\right]\nonumber\\
&=& O_{P_{\theta_0}}(n^{\frac{7}{2}-\frac{C^2}{2}})
\eeqn
which goes to $0$ in probability if we choose $C>\sqrt{7}.$ This result is due to convergence of $J_{n}(\mathbf{X},\hat{\theta}_{n}),J'_{n}(\mathbf{X},\hat{\theta}_{n})$
respectively to $J(\theta_0,\theta_0),J'(\theta_{0},\theta_0)$ which is validated from Assumption \ref{ext2}.
Now considering the first term of the integral (\ref{second}) we have $\frac{|y|}{\sqrt{n}}<\delta,$ We have under $A_2$ 
$$|R_n(\hat{\theta}_{n})|\le\frac{1}{6}\frac{\delta y^2 L_{n}^{(3)}(\hat{\theta}_{n})}{c^{\frac{3}{2}}}+\frac{1}{24}\delta^2 y^2\frac{L^{(4)}_{n}(\hat{\theta}_n)}{c^2},  \s \s  |M_n|=\frac{1}{120}\frac{y^5}{n^{3/2}c^2}L_{n}^{(5)}(\theta')\le\frac{1}{120}\frac{y^2 \delta^3}{c^2}L_{n}^{(5)}(\theta')$$
and since under $A_2$ the quantities $\sup_{\theta' \in (\hat{\theta}_{n},\hat{\theta}_{n}+\frac{y}{\sqrt{nc}})}L_{n}^{(4)}(\theta')$,$\frac{L^{(4)}_{n}(\hat{\theta}_n)}{c^2}$ and $\frac{L^{(3)}_{n}(\hat{\theta}_{n})}{c^{3/2}}$ are $O_{p}(1),$ given a small $\epsilon>0$ one can always choose a $\delta$ so that we can get
\beqn\label{choose}
P_{\theta_{0}}\big\{-\frac{y^2}{2}+R_n(\hat{\theta}_{n})+M_n < -\frac{y^2}{4}, \s \forall y\in A_2\big\} >1-\epsilon \s \text{for}\s n>n_{0}.
\eeqn
So with probability greater than $1-\epsilon$
\beqn
n\int_{A_2}J_{n}(\mathbf{X},\hat{\theta}_{n} +\frac{y}{\sqrt{nc}})e^{-\frac{y^2}{2}+R_n+M_n}dy&\le& \sup_{y\in A_2}J_{n}(\mathbf{X},\hat{\theta}_{n} +\frac{y}{\sqrt{nc}})\s n\int_{A_2} e^{-\frac{y^2}{4}} dy\nonumber\\
&\to^{a.s} & 0 \s\text{as} \s n\to \infty.
\eeqn
The last line follows from the fact under $A_2,\s (\hat{\theta}_{n},\hat{\theta}_{n}+\frac{y}{\sqrt{nc}})\subset (\theta_{0}-\delta,\theta_{0}+\delta)$ almost surely and then by applying Assumption \ref{ext2}, $\sup_{y\in A_2}J_{n}(\mathbf{X},\hat{\theta}_{n} +\frac{y}{\sqrt{nc}})\le \sup_{\theta \in (\theta_{0}-\delta,\theta_{0}+\delta)} J(\theta_0,\theta)$ asymptotically almost surely. The integral will converge to $0$ as $n\to \infty$ by choosing a bigger $ C.$ Choice of $\delta$ will be specified by (\ref{choose}) given a small $\epsilon>0.$
\end{proof}

Now it's obvious to conclude from Lemma \ref{lem1} that for any $A\in \mathcal{B}(\mathbb{R})$
\beqn
&&\int_{A} e^{-\frac{y^{2}}{2}}\lt(1+R_{n}(\hat{\theta}_{n})+\frac{R_{n}(\hat{\theta}_{n})^2}{2}\rt) K(\theta_0,\X,y)dy\label{co1}\\&=&\int_{A} e^{-\frac{y^{2}}{2}}\lt[1+\frac{1}{\sqrt{n}}\lt(A_{1}y+A_{3}y^3\rt)+\frac{1}{n}\lt(A_{2}y^{2}+A_{4}y^{4}+A_{6}y^{6}+W^{(1)}_{n}\frac{y}{\sqrt{c}}\rt)\rt]dy+o_{p_{{\theta_{0}}}}\lt(\frac{1}{n}\rt)\non
\eeqn
where
$$ A_{1}:= c^{-\frac{1}{2}}\frac{J'(\theta_{0},\hat{\theta}_{n})}{J(\theta_{0},\hat{\theta}_{n})}, 
A_{2}:= \frac{1}{2}c^{-1}\frac{J''(\theta_{0},\hat{\theta}_{n})}{J(\theta_{0},\hat{\theta}_{n})}, A_{3}:= \frac{1}{6}c^{-\frac{3}{2}}a ,\s A_{4}:=A_{1}A_{3}+\frac{1}{24}c^{-2}a_4, A_{6}:=\frac{1}{2}A^{2}_{3}.$$
(\ref{co1}) follows from the fact that all higher order terms will accumulate in $o_{p_{\theta_{0}}}(\frac{1}{n}).$ For an illustration taking just one cross-product term of second term of $R_{n}(\hat{\theta}_{n})$ and $\frac{A_{1}y}{\sqrt{n}},$ one has
\beqn
\frac{1}{24}\frac{y^4}{nc^2}L_{n}^{(4)}(\hat{\theta}_{n})\frac{A_{1}y}{\sqrt{n}}&=&\frac{1}{24}\frac{A_{1}y^{5}}{n\sqrt{n}c^2}L_{n}^{(4)}(\hat{\theta}_{n})1_{\{|y|\le\log n\}}+\frac{1}{24}\frac{A_{1}y^{5}}{n\sqrt{n}c^2}L_{n}^{(4)}(\hat{\theta}_{n})1_{\{|y|>\log n\}}.\non
\eeqn
Since $R_{n}(\hat{\theta}_{n}).A_{1}$ is $O_{P_{\theta_{0}}}(1)$
\beqn
\int_{\R}e^{-\frac{y^2}{2}}R_{n}(\hat{\theta}_{n})\frac{A_{1}y}{\sqrt{n}}dy&=&\frac{L_{n}^{(4)}(\hat{\theta}_{n})A_{1}}{24c^{2}n^{\frac{3}{2}}}\int_{|y|\le\log n} y^{5}e^{-\frac{y^2}{2}}dy+\frac{L_{n}^{(4)}A_{1}}{24c^{2}n^{\frac{3}{2}}}\int_{|y|>\log n} y^{5}e^{-\frac{y^2}{2}}dy\non\\
&\le& O_{P_{\theta_{0}}}\bigg(\frac{(\log n)^5}{n^{\frac{3}{2}}}\bigg)+\frac{L_{n}^{(4)}A_{1}}{24c^{2}n^{\frac{3}{2}}}\int_{|y|>\log n} y^{5}e^{-\frac{y^2}{2}}dy\label{co2}
\eeqn
Since Gamma distribution is exponentially tailed, whole R.H.S of  (\ref{co2}) is of $o_{P_{\theta_{0}}}(\frac{1}{n})$. Now note that the formula for $r$-th (even) central moment of standard normal distribution $EX^{r}:= (r-1)(r-3)\ldots1$. Dividing the quantity $J_{n}(\mathbf{X},\hat{\theta}_{n} +\frac{y}{\sqrt{nc}})e^{n[L_{n}(\hat{\theta}_{n} +\frac{y}{\sqrt{nc}})-L_{n}(\hat{\theta}_{n} )] }$ with the expansion of the denominator $\int_{\R}J_{n}(\mathbf{X},\hat{\theta}_{n} +\frac{y}{\sqrt{nc}})e^{n[L_{n}(\hat{\theta}_{n} +\frac{y}{\sqrt{nc}})-L_{n}(\hat{\theta}_{n} )] }dy$, one has the asymptotic expansion of the fiducial density (upto second order in terms of expansion with respect to $\frac{1}{\sqrt{n}}$) of $y$
\beqn
f_{\mathcal{G}}(y\mid \X)&=&\phi(y)\bigg(1+\frac{1}{\sqrt{n}}\lt(A_{1}y+A_{3}y^3\rt)+\frac{1}{n}\bigg(A_{2}(y^{2}-1)+A_{4}(y^{4}-3)+A_{6}(y^{6}-15)\non\\&+&W_{n}^{(1)}(\X)\frac{y}{\sqrt{c}}\bigg)\bigg)+o_{p_{{\theta_{0}}}}\lt(\frac{1}{n}\rt)\label{expr1}
\eeqn

where $\phi(.)$ is the density function of the normal distribution. From (\ref{co1}) we get (\ref{expr1}) using the power series expansion $\frac{1}{1+x}=\sum_{i=1}^{\infty} (-1)^{i}x^{i}$ given $|x|<1$ on first two ordered terms.

\begin{remark}The conclusion (\ref{expr1}) will remain unchanged if $W_{n}^{(1)}(\X)$ is replaced by a random variable $\widehat{W}^{(1)}_{n}(\X)$ that is $\sigma(\X)$ measurable with the property
$$P_{\theta_0}\lt[W_{n}^{(1)}(\X) \neq \widehat{W}^{(1)}_{n}(\X)\rt] = e^{-cn}\s\s\text{for some } c>0.$$  It is because the quantity $\lt(W_{n}^{(1)}(\X) - \widehat{W}^{(1)}_{n}(\X)\rt)$ multiplied with any polynomial ordered term of $n$ will remain $O_{P_{\theta_0}}(e^{-cn})$ since  it doesn't hamper in any specific polynomial order terms.
\end{remark}

Recall the classical orthogonal \textbf{Hermite polynomials} $\{H_{n}(x)\}_{n\ge 1}$ which is defined as 
$$H_{n}(x)= (-1)^{n} e^{\frac{x^{2}}{2}}\bigg[\frac{d^{n}}{dx^{n}}e^{-\frac{x^{2}}{2}}\bigg].$$ First few Hermite polynomials are 
\beqn
H_{0}(x)=1,\s H_{1}(x)=x,\s H_{2}(x)=x^{2}-1,\s H_{3}(x)=x^{3}-3x,\s H_{4}(x)=x^{4}-6x^{2}+3,\non\\
H_{5}(x)=y^5-10y^3 +15y,\s\s H_{6}(x)=y^6 - 15y^4 +45y^2 -15\non.
\eeqn
 Following properties hold where $\phi(x)$ is the density of normal distribution: For all $a\in\R,$
\beqn
\int_{-\infty}^{a}H_{1}(y)\phi(y)dy= -\phi(a)\s\text{and}\s\int_{-\infty}^{a}H_{n}(y)\phi(y)dy= -H_{n-1}(a)\phi(a)\s\s\quad\forall n\ge 2.\label{herm}
\eeqn

Expressing (\ref{expr1}) with each coefficient in terms of Hermite polynomials we get,
\beqn
f_{\mathcal{G}}(y)&=&\phi(y)\Bigg(1+\frac{1}{\sqrt{n}}\lt(G_{1}H_{1}(y)+G_{3}H_{3}(y)\rt)+\frac{1}{n}\bigg(G_{2}H_{2}(y)+G_{4}H_{4}(y)\non\\&+&G_{6}H_{6}(y)+W_{n}^{(1)}\frac{H_{1}(y)}{\sqrt{c}}\bigg)\Bigg)
+o_{p_{{\theta_{0}}}}\lt(\frac{1}{n}\rt)\non
\eeqn
where 
\beqn
G_{1}:= A_{1}+3A_{3}, \s G_{2}=A_{2}+6A_{4} +45A_{6}, \s G_{3}:= A_{3},\s G_{4}=A_{4}+15A_{6}\s G_{6}=A_{6}.\non 
\eeqn
Define further
\beqn
\beta_{1} := G_{1} + G_{3}.H_{2}(z),\s\beta_{2}:= 2z\beta_{1}G_{3}-\frac{1}{2}\beta^{2}_{1}z+G_{2}H_{1}(z)+G_{4}H_{3}(z)+G_{6}H_{5}(z)+\frac{W_{n}^{(1)}(\X)}{\sqrt{c}}\s
\eeqn

This following illustration is similar with Theorem 2.3.1 of \cite{datta2004probability} which gives  an asymptotic expansion of $(1-\alpha)$-th fiducial quantile.
\begin{Corollary}\label{Cor1}
 Denote $\theta^{(1-\alpha)}(\X,\mathcal{G}) := \hat{\theta} + (nc)^{-\frac{1}{2}} (z + n^{-\frac{1}{2}}\beta_1+n^{-1}\beta_2).$  Suppose Assumptions $(A_{1})$-$(A_{5})$ of asymptotic normality of likelihood expansion, along with Assumption \ref{str},\ref{ext} \ref{ext2} with $m=1$ hold. Then we have
\beqn
P^{\mathcal{G}}\lt[ \theta \leq \theta^{(1-\alpha)}(\X,\mathcal{G}) \bigg| \X\rt] = 1 - \alpha + o_{p_{{\theta_0}}}(n^{-1}).
\eeqn
\end{Corollary}

\begin{proof}
The concerned quantity
\beqn
&&P^{\mathcal{G}}\lt[\theta \leq \theta^{(1-\alpha)}(\X,\mathcal{G}) \bigg| \X\rt]= P\lt[\theta \leq \hat{\theta} +(nc)^{-\frac{1}{2}}\lt\{z+n^{-\frac{1}{2}}\beta_{1}+n^{-1}\beta_{2}\rt\}\bigg| \X\rt]\nonumber\\
&=& P\lt[y\leq z+ n^{-\frac{1}{2}}\beta_{1}+n^{-1}\beta_{2} \bigg| \X\rt]\nonumber\\
&=&\int_{-\infty}^{z+ n^{-\frac{1}{2}}\beta_{1}+n^{-1}\beta_{2}}\phi(y)\lt[1+\frac{1}{\sqrt{n}}\lt(G_{1}H_{1}(y)+G_{3}H_{3}(y)\rt)\rt]dy\non\\&+&\frac{1}{n}\int_{-\infty}^{z+ n^{-\frac{1}{2}}\beta_{1}+n^{-1}\beta_{2}}\phi(y)\lt[\lt(G_{2}H_{2}(y)+G_{4}H_{4}(y)+G_{6}H_{6}(y)+W_{n}^{(1)}\frac{H_{1}(y)}{\sqrt{c}}\rt)\rt]dy +o_{p}(n^{-1}).\non
\eeqn
Using the properties of Hermite polynomials on (\ref{herm}) one easily gets
\beqn
&&P^{\mathcal{G}}\lt[\theta \leq \theta^{(1-\alpha)}(F,\X,n)\big| \X\rt]=\Phi(z+ n^{-\frac{1}{2}}\beta_{1}+n^{-1}\beta_{2})- n^{-\frac{1}{2}}.\phi(z+n^{-\frac{1}{2}}\beta)\bigg[G_{1}\label{eqq1}\\&+&G_{3}H_{2}(z+n^{-\frac{1}{2}}\beta_1)\bigg]
-n^{-1}\phi(z) \lt[G_{2}H_{1}(z)+G_{4}H_{3}(z)+G_{6}H_{5}(z)+\frac{W_{n}^{(1)}(\X)}{\sqrt{c}}\rt]+o_{p}(n^{-1}).\non
\eeqn
Using Taylor's expansions of $\Phi(x),\phi(x)$ and accumulating the higher order terms into $o_{p}(n^{-1}),$ the RHS of (\ref{eqq1}) is simplified to
\beqn
P^{\mathcal{G}}\lt[\theta \leq \theta^{(1-\alpha)}(\X,\mathcal{G})\big| \X\rt]&=&\Phi(z)+n^{-\frac{1}{2}}\phi(z)\lt\{\beta_{1}-G_{1}-G_{3}H_{2}(z)\rt\}+n^{-1}\phi(z)\bigg[\beta_{2}-2z\beta_{1}G_{3}\nonumber\\ -\frac{1}{2}\beta^{2}_{1}z+\beta_{1}z\bigg\{G_{1}+G_{3}H_{2}(z)\bigg\}&-&G_{2}H_{1}(z)-G_{4}H_{3}(z)-G_{6}H_{5}(z)-\frac{W_{n}^{(1)}(\X)}{\sqrt{c}}\bigg]+o_{p}(n^{-1})\non\\
&=& 1-\alpha +o_{p}(n^{-1})\label{eq2}
\eeqn
where (\ref{eq2}) follows from the definitions of $\beta_1$ and $\beta_2$. Corollary \ref{Cor1} follows from that.
\end{proof}

Higher order asymptotics in context of Probability matching prior is an old topic and well documented in \cite{datta2004probability}. 
The idea of Shrinkage method was essentially originated from \cite{bickel1990decomposition} in context of establishing higher order asymptotics of Bertlett test statistics. In general it is used to find an expansion of $E_{\theta_{0}}\lt[g(\X,\theta)\rt]$ for any function $g(\X,\theta)$ (in our case $g(\X,\theta) := 1_{\{\theta \leq \theta^{(1-\alpha)}(\X,\mathcal{G})\}}$). Some relevant works on probability matching data dependent prior were done in  \cite{mukerjee2008data, ong2010data} but \textit{data dependence is either coming from moments or the maximum likelihood estimator. In comparison to that here the term $W_{n}^{(1)}(\X)$ is much like a ratio estimator where its higher order expansion is interestingly critical for the terms after first order which makes the following calculation relevant.} In order to implement Shrinkage method one formulates an auxiliary prior $\bar{\pi}$ with properties that it is proper, supported on a compact set, having true $\theta_{0}$ in its interior. It  vanishes on the boundary of the support while taking strictly positive values in the interior. It also satisfies all the conditions $B_{m}$ ($m=1,2$) in \cite{bickel1990decomposition} ensuring  smoothness of $\bar{\pi}$ ,and  $\log \bar{\pi}$ and and its derivatives near the boundary of the support. Basic steps of Shrinkage method (for upto second order) are following:

\begin{enumerate}[(a)]
\item \textbf{Step 1:} Start with an auxiliary prior $\bar{\pi}$ with a compact support $\subseteq \mathcal{H}$ containing $\theta_{0}$ as an interior point. We will find the expansion of  $E^{\bar{\pi}}( g(\X,\theta)| \X)$ upto $o_{P_{\theta_{0}}}(\frac{1}{n})$.

\item \textbf{Step 2:} Under the assumption that the $\X=\lt(X_{1},X_{2},X_{3}\ldots,X_{n}\rt)$ generated from $ F(\cdot \mid \theta)$ we compute :  $\lambda(\theta) := E_{\theta}E^{\bar{\pi}}( g(\X,\theta) | \X)$ upto $o(\frac{1}{n}).$

\item \textbf{Step 3:} Compute $\int \lambda(\theta) \bar{\pi}(d\theta)$ when $\bar{\pi} \leadsto \delta_{\theta_{0}}(.)$. The final quantity after taking the weak limit leads to the required expansion of $E_{\theta_{0}}\lt( g(\X,\theta_{\theta_{0}})\rt)$ upto $o(\frac{1}{n}).$
\end{enumerate}

\begin{Proposition}\label{cor2}
Note  if one observes $T(\X) := E_{\bar{\pi}}(g(\X,\theta)\mid \X)$ for an integrable function $T(\X)$ (with respect to $P_{\theta_{}}$ for  $\theta\in(\theta_{0}-\delta,\theta_{0}+\delta)$ for some $\delta>0$) after \textbf{Step 1} of Shrinkage method,  \textbf{Step 2} and \textbf{Step 3} virtually compute $E_{\theta_{0}}T(\X)$. Since $\lambda(\theta)=E_{\theta}T(\X),$ through Dominated Convergence Theorem and a consequence of the weak limit gives $$\lim_{\bar{\pi} \leadsto \delta_{\theta_{0}}(\cdot)}\int E_{\theta}\lt[T(\X)\rt]\bar{\pi}(d\theta).$$ 
A good illustration on how Shrinkage method works is given at Chapter 1 of \cite{datta2004probability}.
\end{Proposition}
The conditions $B_{m}$ in \cite{bickel1990decomposition} ensures the existence of a set $S$ which contains data $\X$ with probability $P_{\theta}$, $(1-o_{p}(n^{-1}))$ for $\theta \in$ a compact set $K$. For ensuring second order term we need to just assume $B_{2}$ for the auxiliary prior $\bar{\pi}(\cdot)$ containing true $\theta_{0}$ in interior. All the following calculation of the Shrinkage method is a consequence of those assumptions in $B_{2}.$ We will complete the proof of Theorem~\ref{uni} by translating the Shrinkage method 
to the \GFI framework.
\begin{proof}
\begin{enumerate}[(A)]
\item \textbf{Step 1:} We will construct a prior $\bar{\pi}$ with aforementioned smoothness properties and with a compact support with $\theta_{0}$ being an interior point. Now define the following quantities \begin{eqnarray}\label{1}
\bar{G}_{1}:= \bar{A}_{1}+3\bar{A}_{3}, \s \bar{G}_{2}=\bar{A}_{2}+6\bar{A}_{4} +45\bar{A}_{6}, \s \bar{G}_{3}:= \bar{A}_{3},\s \bar{G}_{4}=\bar{A}_{4}+15\bar{A}_{6}\s \bar{G}_{6}=\bar{A}_{6}\non
\eeqn
where
\beqn
\bar{A}_{1}:= c^{-\frac{1}{2}}\frac{\bar{\pi}'(\hat{\theta}_{n})}{\bar{\pi}(\hat{\theta}_{n})},\s \bar{A}_{2}:= \frac{1}{2}c^{-1}\frac{\bar{\pi}''(\hat{\theta}_{n})}{\bar{\pi}(\hat{\theta}_{n})}, \bar{A}_{3}:= A_{3} ,\s \bar{A}_{4}:=\bar{A}_{1}\bar{A}_{3}+\frac{1}{24}c^{-2}a_4, \s\bar{A}_{6}:=A_{6}.&&\non
\end{eqnarray}
By proceeding similarly like Lemma \ref{lem1} or from \cite{ghosh1994higher}, one gets a similar posterior expansion of $\bar{\pi}(y\big|\X):=\bar{\pi}(\hat{\theta}_{n} +\frac{y}{\sqrt{nc}})e^{n[L_{n}(\hat{\theta}_{n} +\frac{y}{\sqrt{nc}})-L_{n}(\hat{\theta}_{n} )] }$ like following display, 
\beqn
\int_{\R}\Bigg|\bar{\pi}(\hat{\theta}_{n} +\frac{y}{\sqrt{nc}})e^{n[L_{n}(\hat{\theta}_{n} +\frac{y}{\sqrt{nc}})-L_{n}(\hat{\theta}_{n} )] }&-& \bar{\pi}(\hat{\theta}_{n})e^{-\frac{y^2}{2}}\lt(1+\frac{\bar{\pi}'(\hat{\theta}_{n})}{\bar{\pi}(\hat{\theta}_{n})}\frac{y}{\sqrt{nc}}+\frac{1}{n}  \frac{\bar{\pi}''(\hat{\theta}_{n})}{\bar{\pi}(\hat{\theta}_{n})}\frac{y^2}{2c}\rt)\non\\ &&\lt(1+R_{n}(\hat{\theta}_{n})+\frac{R_{n}(\hat{\theta}_{n})^2}{2}\rt) \Bigg|dy = o_{P_{\theta_{0}}}\lt(\frac{1}{n}\rt)\s.\non
\eeqn
\beqn
\text{So,} \s\s f^{\bar{\pi}}(\theta \mid \X)= \frac{\bar{\pi}(\theta)e^{n[L_{n}(\theta)-L_{n}(\hat{\theta}_{n} )] }}{\int \bar{\pi}(\theta)e^{n[L_{n}(\theta)-L_{n}(\hat{\theta}_{n} )] }d\theta}=\frac{\bar{\pi}(\hat{\theta}_{n} +\frac{y}{\sqrt{nc}})e^{n[L_{n}(\hat{\theta}_{n} +\frac{y}{\sqrt{nc}})-L_{n}(\hat{\theta}_{n} )] }}{\int \bar{\pi}(\hat{\theta}_{n} +\frac{y}{\sqrt{nc}}))e^{n[L_{n}(\hat{\theta}_{n} +\frac{y}{\sqrt{nc}})-L_{n}(\hat{\theta}_{n} )] }dy}.\non
\eeqn
Using the above expression of ``auxiliary" posterior density of $y=\sqrt{nc}(\theta - \hat{\theta}_{n}),$
\beqn
f^{\bar{\pi}}(y\mid \X)&=&\phi(y)\Bigg(1+\frac{1}{\sqrt{n}}\lt(\bar{G}_{1}H_{1}(y)+\bar{G}_{3}H_{3}(y)\rt)+\frac{1}{n}\bigg(\bar{G}_{2}H_{2}(y)+\bar{G}_{4}H_{4}(y)+\bar{G}_{6}H_{6}(y)\bigg)\Bigg)
\non\\&&+o_{p_{{\theta_{0}}}}(n^{-1}).\,\non
\eeqn
Using the expansion one can write $P^{\bar{\pi}}\lt[\theta \leq \theta^{(1-\alpha)}(\X,\mathcal{G}) \big| \X\rt]$ as
\beqn
&& P\lt[\theta \leq \hat{\theta} +(nc)^{-\frac{1}{2}}\lt\{z+n^{-\frac{1}{2}}\beta_{1}+n^{-1}\beta_{2}\rt\}\bigg| \X\rt]= P^{\bar{\pi}}\lt[y\leq z+ n^{-\frac{1}{2}}\beta_{1}+n^{-1}\beta_{2} \bigg| \X\rt]\nonumber\\
&=&\int_{-\infty}^{z+ n^{-\frac{1}{2}}\beta_{1}+n^{-1}\beta_{2}}\phi(y)\lt[1+\frac{1}{\sqrt{n}}\lt(\bar{G}_{1}H_{1}(y)+\bar{G}_{3}H_{3}(y)\rt)\rt]dy\non\\&+&\frac{1}{n}\int_{-\infty}^{z+ n^{-\frac{1}{2}}\beta_{1}+n^{-1}\beta_{2}}\phi(y)\lt[\lt(\bar{G}_{2}H_{2}(y)+\bar{G}_{4}H_{4}(y)+\bar{G}_{6}H_{6}(y)\rt)\rt]dy +o_{p}(n^{-1})\non
\eeqn
Working similarly like (\ref{eqq1}-\ref{eq2}) we have $P^{\bar{\pi}}\lt[\theta \leq \theta^{(1-\alpha)}(F,\X,n) \big| \X\rt]$
\beqn
&=&\Phi(z)+n^{-\frac{1}{2}}\phi(z)\lt\{\beta_{1}-\bar{G}_{1}-\bar{G}_{3}H_{2}(z)\rt\}+n^{-1}\phi(z)\bigg[\beta_{2}-2z\beta_{1}\bar{G}_{3}\nonumber\\ &&-\frac{1}{2}\beta^{2}_{1}z+\beta_{1}z\lt\{\bar{G}_{1}+\bar{G}_{3}H_{2}(z)\rt\}-\bar{G}_{2}H_{1}(z)-\bar{G}_{4}H_{3}(z)-\bar{G}_{6}H_{5}(z)\bigg]+o_{p}(n^{-1})\non\\
&=& 1-\alpha+n^{-\frac{1}{2}}\phi(z)\lt\{G_{1} - \bar{G}_{1}\rt\}+n^{-1}\phi(z)\bigg\{\beta_{1}z\lt[\bar{G}_{1}-G_{1}\rt]+\lt[G_{2}-\bar{G}_{2}\rt]H_{1}(z)\non\\&&+\lt[G_{4}-\bar{G}_{4}\rt]H_{3}(z)+\frac{W_{n}^{(1)}(\X)}{\sqrt{c}}\bigg\}+o_{p}(n^{-1})\label{s(-1)}\\
&=& 1-\alpha+n^{-\frac{1}{2}}\phi(z)c^{-\frac{1}{2}}\lt\{\frac{J'(\theta_{0},\hat{\theta}_{n})}{J(\theta_{0},\hat{\theta}_{n})} - \frac{\bar{\pi}'(\hat{\theta}_{n})}{\bar{\pi}(\hat{\theta}_{n})}+a_{n}^{(2)}(\theta,\hat{\theta}_{n})\rt\}+n^{-1}\phi(z)z\bigg\{\bigg(c^{-1} \frac{J'(\theta_{0},\hat{\theta}_{n})}{J(\theta_{0},\hat{\theta}_{n})}\non\\&&- \frac{1}{6}c^{-2}l'''(\hat{\theta}_{n})\bigg)\bigg(\frac{\bar{\pi}'(\hat{\theta}_{n})}{\bar{\pi}(\hat{\theta}_{n})}-\frac{J'(\theta_{0},\hat{\theta}_{n})}{J(\theta_{0},\hat{\theta}_{n})}\bigg)+\frac{1}{2}c^{-1}\bigg[\frac{J''(\theta_{0},\hat{\theta}_{n})}{J(\theta_{0},\hat{\theta}_{n})}-\frac{\bar{\pi}''(\hat{\theta}_{n})}{\bar{\pi}(\hat{\theta}_{n})}\bigg]\bigg\}+n^{-1}\phi(z)\frac{a_{n}^{1}(\X,\theta)}{\sqrt{c}}\non\\&&+o_{p_{{\theta_{0}}}}(n^{-1}).\,\label{s1}
\eeqn
where (\ref{s1}) is obtained after a number of simplifications (putting values of the data dependent constants) and using following decomposition of
 $\W^{(1)}_{n}(\X)$ in \eqref{s(-1)}. 
\beqn
\W^{(1)}_{n}(\X) &:=& \sqrt{n}\lt(\frac{J^{(1)}_{n}(\X,\hat{\theta}_{n})}{J_{n}(\X,\hat{\theta}_{n})}-  \frac{J^{(1)}(\theta_0,\hat{\theta}_{n})}{J(\theta_0,\hat{\theta}_{n})}\rt)\non\\&=&\sqrt{n}\lt(\frac{J^{(1)}_{n}(\X,\hat{\theta}_{n})}{J_{n}(\X,\hat{\theta}_{n})}-  \frac{J^{(1)}(\theta,\hat{\theta}_{n})}{J(\theta,\hat{\theta}_{n})}\rt)+\sqrt{n}\lt(\frac{J^{(1)}(\theta,\hat{\theta}_{n})}{J(\theta,\hat{\theta}_{n})}-  \frac{J^{(1)}(\theta_0,\hat{\theta}_{n})}{J(\theta_0,\hat{\theta}_{n})}\rt) \non\\
&=:&a_{n}^{1}(\X,\theta)+\sqrt{n}a_{n}^{(2)}(\theta,\hat{\theta}_{n})\s\s(\text{defining the first and second term})\label{w}
\eeqn
where each of these terms will be analyzed in next step. 
\item \textbf{Step 2:} We will now compute the asymptotic value of $\lambda(\theta) = E_{\theta}P^{\bar{\pi}}( \theta \leq \theta^{(1-\alpha)}(\X,\mathcal{G}) | X).$ Note that asymptotically $a_{n}^{(2)}(\theta,\hat{\theta}_{n})$ converges to $a^{(2)}(\theta):=\frac{J^{(1)}(\theta,\theta)}{J(\theta,\theta)}-  \frac{J^{(1)}(\theta_0,\theta)}{J(\theta_0,\theta)},$ under true $\theta$ which becomes $a^{(2)}(\theta_{0})=0$ when $\theta=\theta_{0}.$ We will treat $a_{n}^{1}(\X,\theta)$ by expanding that term. We take the facility of choosing auxiliary $\bar{\pi}(\cdot)$ in a way such that the expression (\ref{s1}) holds for all data points in a compact set $\bar{S}$ in $\mathbb{R}$ that has $P_{\theta}$ of order $ (1-o(n^{-1}))$ uniformly for all $\theta \in K$, where $K$ is the compact domain of $\bar{\pi}$. Under the assumption that the limits exist and the existence of the set $S\times K$ is ensured by the condition  $B_{m}$ satisfied by $\bar{\pi}$ for $m=2$ in \cite{bickel1990decomposition}. From (\ref{s1})
\beqn
&&\lambda(\theta):=E_{\theta}\lt\{P^{\bar{\pi}}\lt[\theta \leq \theta^{(1-\alpha)}(\X,\mathcal{G}) \Big| \X\rt]\rt\}=1-\alpha+n^{-\frac{1}{2}}\phi(z)I_{\theta}^{-\frac{1}{2}}\bigg\{\frac{J'(\theta_{0},\theta)}{J(\theta_{0},\theta)} - \frac{\bar{\pi}'(\theta)}{\bar{\pi}(\theta)}\non\\&+&a^{(2)}(\theta)\bigg\}
+n^{-1}\phi(z)z\bigg\{\bigg(I_{\theta}^{-1} \frac{J'(\theta_{0},\theta)}{J(\theta_{0},\theta)}- \frac{1}{6}I_{\theta}^{-2}M_{\theta}\bigg)\bigg(\frac{\bar{\pi}'(\theta)}{\bar{\pi}(\theta)}-\frac{J'(\theta_{0},\theta)}{J(\theta_{0},\theta)}\bigg)\non\\&+&\frac{1}{2}I_{\theta}^{-1}\bigg[\frac{J''(\theta_{0},\theta)}{J(\theta_{0},\theta)}-\frac{\bar{\pi}''(\theta)}{\bar{\pi}(\theta)}\bigg]\bigg\}+n^{-1}\phi(z)\frac{E_{\theta}\lt[a_{n}^{1}(\X,\theta)\rt]}{\sqrt{I_{\theta}}}+o_{}(n^{-1})\label{s2}
\eeqn

where $M_{\theta}:=El^{(3)}(\theta|\X)$. By Assumption $a_{i}(\cdot)$ is a continuous function, so in a compact domain($\bar{\pi}$) containing $\theta_{0}$  it will always exist. It is a consequence of the Corollary \ref{cor2} but we need to show the integrability of $E_{\theta}\lt[a^{1}_{n}(\X,\theta)\rt]$ in $\theta\in \text{domain}(\bar{\pi})$. In the following we will give an expansion of $E_{\theta}\lt[a^{1}_{n}(\X,\theta)\rt]$ in terms of $a_{i}(\theta_{0}), J(\theta_{0},\theta), J'(\theta_{0},\theta).$  Now by Assumption \ref{ext}-\ref{as4} we have $a_{i}(\cdot)$ continuous function in $\Theta=\R$. Since also $J(\cdot,\theta)$ won't vanish is $\theta \in (\theta_{0}-\delta,\theta+\delta)$ one can always find a compact neighborhood of $\theta_{0}$ where the quantity $n\to \infty \s E_{\theta}\lt[a^{1}_{n}(\X,\theta)\rt]$ will remain bounded. We will take that compact neighborhood as the  $\text{domain}(\bar{\pi})$.

Now we will prove the higher order expansion of the quantity $E_{\theta}\lt[a^{1}_{n}(\X,\theta)\rt]$. Note that by Taylor's expansion on the function $f(x,y)=\frac{x}{y}$ at the point $(T_1,T_2)$ around $(g_1,g_2)$, we have for some $(g^{*}_1,g^{*}_2)\in (T_1,g_1)\times(T_2,g_2)$
\beqn
\frac{T_1}{T_2}=\frac{g_{1,\theta}}{g_{2,\theta}}+(T_1-g_{1,\theta})\frac{1}{g_{2,\theta}}- (T_2-g_{2,\theta})\frac{g_{1,\theta}}{g^{2}_{2,\theta}}+ (T_2 - g_{2,\theta})^{2}\frac{g^{*}_{1,\theta}}{g_{2,\theta}^{*3}}- (T_1-g_{1,\theta})(T_2-g_{2,\theta})\frac{1}{g_{*2,\theta}}.
\eeqn
That yields $\sqrt{n}[\frac{T_1}{T_2}-\frac{g_{1,\theta}}{g_{2,\theta}}]=\sqrt{n}(T_1-g_{1,\theta})\frac{1}{g_{2,\theta}}- \sqrt{n}(T_2-g_{2,\theta})\frac{g_{1,\theta}}{g^{2}_{2,\theta}}+\sqrt{n}(T_2-g_{2,\theta})[(T_2-g_{2,\theta})\frac{g^{*}_{1,\theta}}{g_{2,\theta}^{*3}} -(T_1-g_{1,\theta})\frac{1}{g_{*2}}].$
Now choosing $T_1:= J'_{n}(\X,\hat{\theta}_{n}),T_2:=J_{n}(\X,\hat{\theta}_{n})$ and their corresponding limits $g_{1,\theta}:= J'(\theta,\theta),g_{2,\theta}:=J(\theta_{},\theta_{})$ one gets $a^{1}_{n}(\X,\theta)$ is $O_{P_{\theta}}(1).$ Last statement is a consequence of Slutsky's theorem and  the Assumption \ref{as4}. Note that $\frac{g^{*}_{1,\theta}}{g_{2,\theta}^{*3}}, \frac{1}{g_{*2,\theta}}$ are $O_{P_{\theta_0}}(1)$ which follows from the fact $g^{*}_{1}\xrightarrow{P_{\theta_0}}g_{1,\theta},$ $g^{*}_{2,\theta}\xrightarrow{P_{\theta}}g_{2,\theta},$ and then using continuity theorem one has $\frac{g^{*}_{1,\theta}}{g_{2,\theta}^{*3}}=\frac{g^{}_{1,\theta}}{g_{2,\theta}^{3}}+o_{P_{\theta_{}}}(1),$ $\frac{1}{g_{*2,\theta}}= \frac{1}{g_{2,\theta}}+o_{P_{_\theta}}(1)\s\text{since }\s g_2\neq 0.$ From Slutsky's theorem the residual term $\sqrt{n}(T_2-g_2)[(T_2-g_2)\frac{g^{*}_{1,\theta}}{g_{2,\theta}^{*3}} -(T_1-g_{1,\theta})\frac{1}{g_{*2,\theta}}]$ will be  $o_{p_{\theta_{}}}(1).$ Our conclusion that $E_{\theta_{}}\lt[a_{n}^{1}(\X,\theta)\rt]=\frac{a_{1}(\theta_{})}{ g_{2,\theta}} - \frac{a_{2}(\theta_{})g_1}{ g^{2}_{2,\theta}}+o(1),$ will follow if we provide an additional detail on the expected residual term:  
\beqn
E_{\theta}\lt[\sqrt{n}(T_2-g_2)\lt[(T_2-g_2)\frac{g^{*}_{1}}{g_{2}^{*3}} -(T_1-g_1)\frac{1}{g_{*2}}\rt]\rt]\to 0\label{res1}
\eeqn
for all $\theta\in \text{domain}(\bar{\pi}).$ Note that 
\beqn
(T_{1}-g_{1})&=&(J'_{n}(\X,\hat{\theta}_{n})-J'(\theta_{},\hat{\theta}_{n}))+((J'(\theta_{},\hat{\theta}_{n})-J'(\theta_{},\theta_{}))\label{res2}\\
(T_{2}-g_{2})&=&(J_{n}(\X,\hat{\theta}_{n})-J(\theta_{},\hat{\theta}_{n}))+((J(\theta_{},\hat{\theta}_{n})-J(\theta_{},\theta_{})).
\eeqn
 Second term after scaling with $\sqrt{n},$ along with the Lipschitz property of $J(\theta,\cdot)$ from Assumption \ref{as4}(c) will give finiteness of the quantity $nE_{\theta}\left[(J(\theta_{},\hat{\theta}_{n})-J(\theta_{},\theta_{})\right]^2<\infty$ (from the asymptotic expansion of MLE $\hat{\theta}_{n}$ under true value $\theta$). Along with that and Assumption \ref{as4}(a) one gets $nE_{\theta}\left[T_{1}-g_{1}\right]^{2}<\infty$ and $nE_{\theta}\left[T_{2}-g_{2}\right]^{2}<\infty$ will follow similarly. Note that these results imply that the set $\{(T_{1},T_{2}):|T_{1}-g_{1}|<\epsilon,|T_{2}-g_{2}|<\epsilon\}$ denoted by $A_{n,\epsilon}$ has probability $P_{\theta}$ of order $(1-O_{P_{{\theta}}}(\frac{1}{n}))$. Since $g_{{2}}$ is away from $0,$ fixing $\epsilon \in(0,g_2)$ we can work with $\frac{g^{*}_{1}}{g_{2}^{*3}}.1_{A_{n,\epsilon}},\frac{1}{g_{*2}}.1_{A_{n,\epsilon}}$  in place of $\frac{g^{*}_{1}}{g_{2}^{*3}},\frac{1}{g_{*2}}$ in the expansion of $a_{n}^{1}(\X,\theta)$ in (\ref{s1}) for $\theta=\theta_0,$ since the residual term (that is non zero with probability $o_{P_{{\theta_0}}}(\frac{1}{\sqrt{n}}$) will be accumulated in the $o_{P_{{\theta_0}}}(\frac{1}{\sqrt{n}})$ term.  Now note $\bigg|\frac{g^{*}_{1}}{g_{2}^{*3}}.1_{A_{n,\epsilon}}\bigg|<\frac{g_{1}+\epsilon}{(g_{{2}}-\epsilon)^{3}},\bigg|\frac{1}{g_{*2}}.1_{A_{n,\epsilon}}\bigg|<\frac{1}{g_{{2}}-\epsilon}\s\text{a.s}.$ Using these along with $nE_{\theta}\left[T_{1}-g_{1}\right]^{2}<\infty, \s nE_{\theta}\left[T_{2}-g_{2}\right]^{2}<\infty,$ (\ref{res1}) will follow by breaking and  analyzing each of the two terms.

\item \textbf{Step 3:} The last step comes from computing $\int \lambda(\theta)\bar{\pi}(d\theta)$ when $\bar{\pi}(\theta) \to \delta_{\theta_{0}}(\theta).$  
Note that if $\alpha(\theta):=\frac{1}{6}I^{-2}_{\theta}M_{\theta} - I_{\theta}^{-1}\frac{J'(\theta_{0},\theta)}{J(\theta_{0},\theta)},$ it follows from (\ref{s2}) that
\beqn
\lambda(\theta)&=&1-\alpha+n^{-\frac{1}{2}}\phi(z)I_{\theta}^{-\frac{1}{2}}\lt\{\frac{J'(\theta_{0},\theta)}{J(\theta_{0},\theta)} - \frac{\bar{\pi}'(\theta)}{\bar{\pi}(\theta)}+a^{(2)}(\theta)\rt\}+n^{-1}\phi(z)z\bigg\{\alpha(\theta)\bigg(\frac{J'(\theta_{0},\theta)}{J(\theta_{0},\theta)}-\frac{\bar{\pi}'(\theta)}{\bar{\pi}(\theta)}\bigg)\non\\&+&\frac{1}{2}I_{\theta}^{-1}\bigg[\frac{J''(\theta_{0},\theta)}{J(\theta_{0},\theta)}-\frac{\bar{\pi}''(\theta)}{\bar{\pi}(\theta)}\bigg]\bigg\}+n^{-1}\phi(z)\frac{E_{\theta}\lt[a_{n}^{1}(\X,\theta)\rt]}{\sqrt{I_{\theta}}}+o_{}(n^{-1}).
\eeqn
From properties of distribution theory one has if $\bar{\pi}(\theta) \rightarrow \delta_{\theta_{0}}(\theta),$ then $$\int f(\theta) .\bar{\pi}(d\theta) \rightarrow f(\theta_{0}), \s\s \int f(\theta) .\frac{\bar{\pi}^{(m)}(d\theta)}{\bar{\pi}(\theta)} \rightarrow (-1)^{m}f^{(m)}(\theta_{0})$$  where for the second result $f$ is an $m$-times differentiable at a neighborhood of  $\theta=\theta_{0}.$ Note that $a^{(2)}(\theta_{0})=0.$ So after taking the weak limit of $\int\lambda(\theta)\bar{\pi}(\theta)$ as $\bar{\pi}(\theta) \rightarrow \delta_{\theta_{0}}(\theta),$ one has
\beqn
P_{\theta_{0}}\lt[\theta_{0}\le \theta^{(1-\alpha)}(\X,\mathcal{G})\rt]&=&1-\alpha+n^{-\frac{1}{2}}\phi(z)\Delta_{1}(\theta_{0})+n^{-1}\phi(z)z\Delta_{2}(\theta_{0})+o_{}(n^{-1})\non
\eeqn
where
\beqn
\Delta_{1}(\theta_{0})&=&I_{\theta_{0}}^{-\frac{1}{2}} \frac{\frac{\partial}{\partial \theta}J(\theta_0,\theta)}{J(\theta_{0},\theta_{0})}\bigg|_{\theta_{0}} + \frac{\partial I^{-\frac{1}{2}}_{\theta}}{\partial \theta}\bigg|_{\theta_{0}}, \non\\
\Delta_{2}(\theta_{0})&=& \bigg\{\bigg(\alpha'(\theta)+\alpha(\theta)\frac{J'(\theta_{0},\theta)}{J(\theta_{0},\theta)}\bigg)+\frac{1}{2}I_{\theta}^{-1}\frac{J''(\theta_{0},\theta)}{J(\theta_{0},\theta)}-\frac{d^{2}}{d\theta^{2}}\bigg[\frac{1}{2}I_{\theta}^{-1}\bigg]\bigg\}\Bigg|_{\theta_{0}}+\frac{E_{\theta_{0}}\lt[a_{n}^{1}(\X,\theta_{0})\rt]}{z\sqrt{I_{\theta_{0}}}}\s\s\label{s31}
\eeqn
\end{enumerate}
$E_{\theta_{0}}\lt[a_{n}^{1}(\X,\theta_{0})\rt]=\frac{a_{1}(\theta_{0})}{ g_2} - \frac{a_{2}(\theta_{0})g_1}{ g^{2}_{2}}+o(1)$
Note that
First term of $\Delta_{2}(\theta_{0})$ in (\ref{s31})
\beqn
&&\bigg(\alpha'(\theta)+\alpha(\theta)\frac{J'(\theta_{0},\theta)}{J(\theta_{0},\theta)}\bigg)+\frac{1}{2}I_{\theta}^{-1}\frac{J''(\theta_{0},\theta)}{J(\theta_{0},\theta)}-\frac{d^{2}}{d\theta^{2}}\bigg[\frac{1}{2}I_{\theta}^{-1}\bigg]\non\\
&=&\frac{1}{J(\theta_{0},\theta)}\bigg[\frac{d}{d\theta}\lt[\alpha(\theta)J(\theta_{0,\theta})\rt]\bigg]+\frac{1}{2}J(\theta_{0},\theta)^{-1}\frac{d}{d\theta}\bigg\{I^{-1}_{\theta}J'(\theta_{0},\theta)-J(\theta_{0},\theta)\big(\frac{d}{d\theta}I^{-1}_{\theta}\big)\bigg\}\non\\
&=&J(\theta_{0},\theta)^{-1}\bigg[\frac{d}{d\theta}\bigg\{\alpha(\theta)J(\theta_{0},\theta)+\frac{1}{2}I^{-1}_{\theta}J'(\theta_{0},\theta) -\frac{1}{2} J(\theta_{0},\theta)\frac{d}{d\theta}[I_{\theta}^{-1}]\bigg\}\bigg]\label{s32}
\eeqn
Using definition of $\alpha(\theta)$ the R.H.S of (\ref{s32}) becomes
\beqn
&&J(\theta_{0},\theta)^{-1}\bigg[\frac{d}{d\theta}\bigg\{\frac{1}{6}I^{-2}_{\theta}M_{\theta}J(\theta_{0},\theta) -\frac{1}{2}I^{-1}_{\theta}J'(\theta_{0},\theta) -\frac{1}{2} J(\theta_{0},\theta)\frac{d}{d\theta}[I_{\theta}^{-1}]\bigg\}\bigg]\non\\
&=&J(\theta_{0},\theta)^{-1}\bigg[\frac{d}{d\theta}\bigg\{\frac{1}{6}I^{-2}_{\theta}M_{\theta}J(\theta_{0},\theta) -\frac{d}{d\theta}\bigg[\frac{1}{2} I_{\theta}^{-1}J(\theta_{0},\theta)\bigg]\bigg\}\bigg]\label{s3l}
\eeqn

Combining two estimates from (\ref{s31}) and (\ref{s3l}) with taking the limit at $\theta=\theta_{0}$ one gets the second order term and the conclusion of Theorem~\ref{uni} follows.\end{proof}


\bibliographystyle{plain}
\bibliography{bib}

\begin{table}[]
\caption{Bivariate Normal correlation $\rho$: (One sided coverage of $(1-\alpha)$th quantile)\label{t:rhoOne}}
\begin{center}
\resizebox{18cm}{4 cm}{
\begin{tabular}{c c c cccccc c cccccc  c cccccc}
\hline\hline
&& \multicolumn{6}{c}{$\rho=0.05$}& &\multicolumn{6}{c}{$\rho=0.25$} & &\multicolumn{6}{c}{$\rho=0.5$}  \\
\cline{4-9}  \cline{11-16} \cline{18-23}  \\
&&\multicolumn{6}{c}{Sample size}& &\multicolumn{6}{c}{Sample size} & &\multicolumn{6}{c}{Sample size} \\
\cline{4-9}  \cline{11-16} \cline{18-23}  \\
$1-\alpha$ &&& 2&3&4&5&10&100& &2&3&4&5&10&100& &2&3&4&5&10&100  \\
\hline
&&& &&&&& && &&&&& && &&&&&\\ 
0.025 &&  FS &0.00575 & 0.0096 & 0.0106 & 0.0148 & 0.02 & 0.0243   &&  0.0028 & 0.0088 & 0.0097 & 0.0124 & 0.0176 & 0.022  &&   0.0019 & 0.0051 & 0.0088 & 0.0112 & 0.017 & 0.022 \\ 
       && F1 &  0.0555 & 0.0496 & 0.0394 & 0.0399 & 0.0364 & 0.0257  &&   0.0499 & 0.044 & 0.0375 & 0.0352 & 0.0346 & 0.0238   &&  0.036 & 0.031 & 0.034 & 0.0325 & 0.0294 & 0.0242   \\
&& BJ &  0.036 & 0.0332 & 0.0276 & 0.0297 & 0.0307 & 0.0255  &&    0.0341 & 0.0332 & 0.0281 & 0.0285 & 0.0309 & 0.0235   &&   0.0265 & 0.0254 & 0.0298 & 0.0292 & 0.0274 & 0.0242\\
&& B2 & 0.03275 & 0.0299 & 0.0252 & 0.0271 & 0.0282 & 0.0249  && 0.0316 & 0.032 & 0.0268 & 0.0277 & 0.0287 & 0.0234  &&  0.0255 & 0.0246 & 0.0289 & 0.0284 & 0.0267 & 0.0241\\
&&& &&&&& && &&&&& && &&&&&\\

0.05&& FS & 0.0175& 0.0248& 0.0254& 0.032& 0.0437& 0.048 && 0.0105& 0.0221& 0.0242& 0.0287& 0.041& 0.0479&& 0.0053& 0.0128& 0.0233& 0.0277& 0.038& 0.0451\\
&& F1 & 0.09625& 0.0885& 0.0768& 0.0725& 0.0675& 0.0515 && 0.0906& 0.0802& 0.0692& 0.0655& 0.0646& 0.051 && 0.077& 0.0647& 0.0602& 0.0589& 0.0558& 0.0487\\
&& BJ& 0.06375& 0.063& 0.0573& 0.0591& 0.0602& 0.0504 && 0.0646& 0.0613& 0.056& 0.0549& 0.0576& 0.0501 && 0.0565& 0.0524& 0.0538& 0.0539& 0.0523& 0.0482\\
&& B2& 0.058& 0.0575& 0.0511& 0.0529& 0.0562& 0.0494 && 0.0609& 0.0574& 0.0528& 0.0512& 0.0553& 0.0495 && 0.055& 0.0513& 0.0531& 0.0518& 0.0519& 0.048\\

&&& &&&&& && &&&&& && &&&&&\\

0.50 && FS & 0.4855& 0.4975& 0.4866& 0.4827& 0.4827& 0.5114 && 0.4611& 0.4626& 0.4606& 0.4617& 0.4774& 0.4866 && 0.4167& 0.4216& 0.4394& 0.4401& 0.4654& 0.4899\\
&& F1 & 0.4985& 0.5074& 0.4941& 0.4926& 0.4882& 0.5131 && 0.5156& 0.5122& 0.5032& 0.5009& 0.5077& 0.4956 && 0.5213& 0.5094& 0.5144& 0.5024& 0.504& 0.5016 \\
&& BJ & 0.4945& 0.5057& 0.491& 0.4901& 0.487& 0.5131 && 0.5001& 0.498& 0.4931& 0.491& 0.501& 0.4956 && 0.4956& 0.49& 0.4984& 0.4884& 0.4986& 0.5013\\
&& B2 & 0.4875& 0.5018& 0.4887& 0.4882& 0.4852& 0.5131 && 0.4923& 0.4908& 0.4865& 0.4851& 0.4979& 0.4953 && 0.4914& 0.4859& 0.495& 0.4884& 0.4975& 0.5015\\

&&& &&&&& && &&&&& && &&&&&\\

0.95&& FS & 0.9795 & 0.9714 & 0.9698 & 0.9649 & 0.9562 & 0.9508 && 0.9727 & 0.9635 & 0.9605 & 0.9517 & 0.9505 & 0.9478 && 0.9602 & 0.9485 & 0.9499 & 0.9425 & 0.9438 & 0.9448\\
&& F1 & 0.8935 & 0.9054 & 0.9201 & 0.9229 & 0.9317 & 0.949 && 0.8865 & 0.9074 & 0.9136 & 0.9142 & 0.9361 & 0.948 && 0.8897 & 0.9069 & 0.9207 & 0.9228 & 0.9392 & 0.9471 \\
&& BJ & 0.92825 & 0.9345 & 0.9423 & 0.9414 & 0.9424 & 0.9496 && 0.9281 & 0.9352 & 0.938 & 0.9333 & 0.9442 & 0.9483 && 0.9265 & 0.93 & 0.9384 & 0.9348 & 0.9437 & 0.9474\\
&& B2 & 0.93675 & 0.9419 & 0.9483 & 0.9474 & 0.9478 & 0.9507 && 0.937 & 0.9436 & 0.9471 & 0.9439 & 0.9493 & 0.949 && 0.9411 & 0.9399 & 0.9472 & 0.9431 & 0.9485 & 0.9479\\
&&& &&&&& && &&&&& && &&&&&\\
0.975&& FS & 0.993 & 0.9888 & 0.9865 & 0.985 & 0.9781 & 0.9745 && 0.99 & 0.985 & 0.983 & 0.9777 & 0.9768 & 0.9729 && 0.9846 & 0.9768 & 0.9747 & 0.9725 & 0.9708 & 0.9725 \\
&& F1 & 0.93175 & 0.9458 & 0.9556 & 0.9584 & 0.9637 & 0.9721 && 0.9255 & 0.9426 & 0.9504 & 0.9516 & 0.9637 & 0.973 && 0.9283 & 0.9408 & 0.9534 & 0.9532 & 0.9664 & 0.9737 \\
&& BJ & 0.959 & 0.9646 & 0.9703 & 0.9691 & 0.9694 & 0.973 && 0.96 & 0.9642 & 0.9668 & 0.9651 & 0.9706 & 0.9736 && 0.9612 & 0.9607 & 0.9663 & 0.9667 & 0.9699 & 0.9737\\
&& B2 & 0.96325 & 0.9677 & 0.9733 & 0.9729 & 0.9728 & 0.9743 && 0.966 & 0.9696 & 0.9716 & 0.9698 & 0.9736 & 0.9748 && 0.9694 & 0.9695 & 0.9718 & 0.9713 & 0.9732 & 0.9743\\
\hline
\hline
\end{tabular}
}
\end{center}

\begin{center}
\resizebox{14cm}{4 cm}{
\begin{tabular}{c c c cccccc c cccccc }
\hline\hline
&& \multicolumn{6}{c}{$\rho=0.75$}& &\multicolumn{6}{c}{$\rho=0.9$}  \\
\cline{4-9}  \cline{11-16}   \\
&&\multicolumn{6}{c}{Sample size}& &\multicolumn{6}{c}{Sample size}  \\
\cline{4-9}  \cline{11-16}  \\
$1-\alpha$ &&& 2&3&4&5&10&100& &2&3&4&5&10&100  \\
\hline
&&& &&&&& && &&&&& \\
0.025&& FS & 0.001 & 0.0051 & 0.0086 & 0.0133 & 0.0215 & 0.026 && 0.0003 & 0.0086 & 0.0137 & 0.0193 & 0.013 & 0.0244 \\
&& F1 & 0.0293 & 0.0292 & 0.0235 & 0.0293 & 0.0266 & 0.0265 && 0.0273 & 0.0214 & 0.0248 & 0.0273 & 0.017 & 0.025 \\
&& BJ & 0.0253 & 0.0264 & 0.0219 & 0.0277 & 0.0263 & 0.0265 && 0.025 & 0.0208 & 0.0242 & 0.0273 & 0.017 & 0.025 \\
&& B2 & 0.025 & 0.0262 & 0.0223 & 0.0281 & 0.0267 & 0.0265 && 0.02567 & 0.0208 & 0.0247 & 0.0277 & 0.017 & 0.025 \\
&&& &&&&& && &&&&& \\
0.05&& FS & 0.0043 & 0.0169 & 0.0207 & 0.0314 & 0.0427 & 0.0492 && 0.0027 & 0.0184 & 0.0317 & 0.0413 & 0.0391 & 0.0484 \\
&& F1 & 0.0603 & 0.054 & 0.0492 & 0.0537 & 0.0547 & 0.051 && 0.058 & 0.0462 & 0.0526 & 0.055 & 0.0435 & 0.0492 \\
&& BJ & 0.051 & 0.0484 & 0.0466 & 0.0518 & 0.0539 & 0.0509 && 0.0537 & 0.044 & 0.0521 & 0.055 & 0.0435 & 0.0492 \\
&&B2 & 0.0506 & 0.0489 & 0.0468 & 0.0526 & 0.0542 & 0.0508 && 0.054 & 0.0448 & 0.052 & 0.055 & 0.0445 & 0.0493\\
&&& &&&&& && &&&&& \\
0.50&& FS & 0.3927 & 0.4016 & 0.4256 & 0.4435 & 0.4763 & 0.4938 && 0.377 & 0.4212 & 0.4426 & 0.4651 & 0.4825 & 0.4958 \\
&& F1 & 0.5222 & 0.4976 & 0.5025 & 0.5072 & 0.5085 & 0.5009 && 0.5067 & 0.498 & 0.4966 & 0.5003& 0.5001 & 0.4995 \\
&& BJ & 0.4953 & 0.4797 & 0.4875 & 0.4989 & 0.5066 & 0.5008 && 0.4896 & 0.4882 & 0.4908 & 0.4981 & 0.5013 & 0.4996 \\
&& B2 & 0.4947 & 0.4814 & 0.4925 & 0.5014 & 0.5079 & 0.5006 && 0.4953 & 0.4926 & 0.496 & 0.4998 & 0.5015 & 0.4995 \\
&&& &&&&& && &&&&& \\
0.95 && FS & 0.953 & 0.9388 & 0.9391 & 0.9383 & 0.9407 & 0.9525 && 0.9267 & 0.9316 & 0.9396 & 0.947 & 0.943 & 0.9524 \\
&& F1 & 0.913 & 0.9271 & 0.9369 & 0.9387 & 0.9453 & 0.954 && 0.927 & 0.9362 & 0.9483 & 0.9543 & 0.9471 & 0.9534 \\
&& BJ & 0.939 & 0.9374 & 0.9413 & 0.9422 & 0.9461 & 0.954 && 0.9322 & 0.9381 & 0.948 & 0.953 & 0.9473 & 0.9532 \\
&& B2 & 0.945 & 0.9429 & 0.9438 & 0.9447 & 0.9468 & 0.9543 && 0.9337 & 0.9396 & 0.9498 & 0.954 & 0.948 & 0.9533\\
&&& &&&&& && &&&&& \\
0.975&& FS & 0.9773 & 0.9705 & 0.9685 & 0.9682 & 0.9684 & 0.9781 && 0.9663 & 0.962 & 0.9706 & 0.9703 & 0.9721 & 0.9762 \\
&& F1 & 0.942 & 0.9554 & 0.9621 & 0.9666 & 0.9706 & 0.9793 && 0.9554 & 0.9642 & 0.9727 & 0.9723 & 0.9747 & 0.9763 \\
&& BJ & 0.967 & 0.9664 & 0.9676 & 0.9697 & 0.9712 & 0.9797 && 0.9633 & 0.966 & 0.9733 & 0.973 & 0.9745 & 0.9763 \\
&& B2 & 0.9712 & 0.9707 & 0.9698 & 0.9713 & 0.9719 & 0.9796 && 0.9653 & 0.967 & 0.9747 & 0.9741 & 0.9755 & 0.9763 \\
\hline
\hline
\end{tabular}
}
\end{center}
\end{table}

\begin{table}[]
\caption{Two sided coverage of $(1-\alpha)$th quantile for correlation coefficient $\rho$ in Bivariate Normal:\label{t:rhoTwo} }
\begin{center}
\resizebox{18cm}{2.4 cm}{
\begin{tabular}{c c c cccccc c cccccc  c cccccc}
\hline\hline
&&& \multicolumn{6}{c}{$\rho=0.05$}& &\multicolumn{6}{c}{$\rho=0.25$} & &\multicolumn{6}{c}{$\rho=0.5$}  \\
\cline{4-9}  \cline{11-16} \cline{18-23}  \\
&&&\multicolumn{6}{c}{Sample size}& &\multicolumn{6}{c}{Sample size} & &\multicolumn{6}{c}{Sample size} \\
\cline{4-9}  \cline{11-16} \cline{18-23}  \\
$1-\alpha$ &&& 2&3&4&5&10&100& &2&3&4&5&10&100& &2&3&4&5&10&100  \\
\hline
&&& &&&&& && &&&&& && &&&&&\\ 
0.95 && FS&  0.98725 & 0.9792 & 0.9759 & 0.9702 & 0.9581 & 0.9502 && 0.9872 & 0.9762 & 0.9733 & 0.9653 & 0.9592 & 0.9509 && 0.9827 & 0.9717 & 0.9659 & 0.9613 & 0.9538 & 0.9505\\
&& F1& 0.87625 & 0.8962 & 0.9162 & 0.9185 & 0.9273 & 0.9464 && 0.8756 & 0.8986 & 0.9129 & 0.9164 & 0.9291 & 0.9492 && 0.8923 & 0.9098 & 0.9194 & 0.9207 & 0.937 & 0.9495\\
&& BJ& 0.923 & 0.9314 & 0.9427 & 0.9394 & 0.9387 & 0.9475 && 0.9259 & 0.931 & 0.9387 & 0.9366 & 0.9397 & 0.9501 && 0.9347 & 0.9353 & 0.9365 & 0.9375 & 0.9425 & 0.9495 \\
&& B2& 0.9305 & 0.9378 & 0.9481 & 0.9458 & 0.9446 & 0.9494 && 0.9344 & 0.9376 & 0.9448 & 0.9421 & 0.9449 & 0.9514 && 0.9439 & 0.9449 & 0.9429 & 0.9429 & 0.9465 & 0.9502\\
&&& &&&&& && &&&&& && &&&&&\\ 
0.90&& FS& 0.962 & 0.9466 & 0.9444 & 0.9329 & 0.9125 & 0.9028 && 0.9622 & 0.9414 & 0.9363 & 0.923 & 0.9095 & 0.8999 && 0.9549 & 0.9357 & 0.9266 & 0.9148 & 0.9058 & 0.8997\\
&& F1& 0.79725 & 0.8169 & 0.8433 & 0.8504 & 0.8642 & 0.8975 && 0.7959 & 0.8272 & 0.8444 & 0.8487 & 0.8715 & 0.897 && 0.8127 & 0.8422 & 0.8605 & 0.8639 & 0.8834 & 0.8984\\
&& BJ& 0.8645 & 0.8715 & 0.885 & 0.8823 & 0.8822 & 0.8992 && 0.8635 & 0.8739 & 0.882 & 0.8784 & 0.8866 & 0.8982 && 0.87 & 0.8776 & 0.8846 & 0.8809 & 0.8914 & 0.8992\\
&& B2& 0.87875 & 0.8844 & 0.8972 & 0.8945 & 0.8916 & 0.9013 && 0.8761 & 0.8862 & 0.8943 & 0.8927 & 0.894 & 0.8995 && 0.8861 & 0.8886 & 0.8941 & 0.8913 & 0.8966 & 0.8999\\
\hline
\hline
\end{tabular}
}
\end{center}
\begin{center}
\resizebox{14cm}{2.4 cm}{
\begin{tabular}{c c c cccccc c cccccc }
\hline\hline
&&& \multicolumn{6}{c}{$\rho=0.75$}& &\multicolumn{6}{c}{$\rho=0.9$}  \\
\cline{3-8}  \cline{10-15}   \\
&&&\multicolumn{6}{c}{Sample size}& &\multicolumn{6}{c}{Sample size}  \\
\cline{3-8}  \cline{10-15}  \\
$1-\alpha$ &&& 2&3&4&5&10&100& &2&3&4&5&10&100  \\
\hline
&&& &&&&& && &&&&& \\
0.95&& FS& 0.9763 & 0.9654 & 0.9599 & 0.9547 & 0.9469 & 0.9521 && 0.966 & 0.9534 & 0.9569 & 0.951 & 0.959 & 0.9516 \\
&& F1& 0.9126 & 0.9262 & 0.9385 & 0.9373 & 0.944 & 0.9528 && 0.9276 & 0.9428 & 0.9479 & 0.945 & 0.9575 & 0.9513 \\
&& BJ& 0.9416 & 0.94 & 0.9457 & 0.942 & 0.9449 & 0.9532 && 0.938 & 0.9452 & 0.9491 & 0.9456 & 0.9575 & 0.9513 \\
&& B2& 0.946 & 0.9445 & 0.9475 & 0.9432 & 0.9452 & 0.9531 && 0.9397 & 0.9462 & 0.95 & 0.9463 & 0.9585 & 0.9513\\
&&& &&&&& && &&&&& \\
0.90&& FS& 0.949 & 0.9219 & 0.9184 & 0.9069 & 0.898 & 0.9033 && 0.924 & 0.9132 & 0.9079 & 0.9056 & 0.904 & 0.904 \\
&& F1& 0.853 & 0.8731 & 0.8877 & 0.885 & 0.8906 & 0.903 && 0.869 & 0.89 & 0.8957 & 0.8993 & 0.9035 & 0.9038\\
&& BJ& 0.8883 & 0.889 & 0.8947 & 0.8904 & 0.8922 & 0.9031 && 0.8783 & 0.894 & 0.8959 & 0.898 & 0.9035 & 0.9038\\
&& B2& 0.8943 & 0.894 & 0.897 & 0.8921 & 0.8926 & 0.9035 && 0.8796 & 0.8948 & 0.898 & 0.899 & 0.9035 & 0.9038 \\
\hline
\hline
\end{tabular}
}
\end{center}
\end{table}

\begin{table}[]
\caption{Length of two sided credible region for correlation coefficient $\rho$ of bivariate normal:\label{t:rhoLength}}
\begin{center}
\resizebox{18cm}{2.1 cm}{
\begin{tabular}{c c c cccccc c cccccc  c cccccc}
\hline\hline
&&& \multicolumn{6}{c}{$\rho=0.05$}& &\multicolumn{6}{c}{$\rho=0.25$} & &\multicolumn{6}{c}{$\rho=0.5$}  \\
\cline{3-8}  \cline{10-15} \cline{17-22}  \\
&&&\multicolumn{6}{c}{Sample size}& &\multicolumn{6}{c}{Sample size} & &\multicolumn{6}{c}{Sample size} \\
\cline{4-9}  \cline{11-16} \cline{18-23}  \\
$1-\alpha$ &&& 2&3&4&5&10&100& &2&3&4&5&10&100& &2&3&4&5&10&100  \\
\hline
&&& &&&&& && &&&&& && &&&&&\\ 
0.95 &&FS& 1.49364 & 1.36899 & 1.27441 & 1.19953 & 0.96714 & 0.37729 && 1.49257 & 1.3594 & 1.26564 & 1.18705 & 0.95007 & 0.35126 && 1.49018 & 1.34738 & 1.23445 & 1.14837 & 0.87824 & 0.26944\\
&& F1&1.30301 & 1.23683 & 1.17742 & 1.12299 & 0.93075 & 0.37544 && 1.29118 & 1.2228 & 1.16132 & 1.10211 & 0.90792 & 0.3488 && 1.26177 & 1.17863 & 1.09868 & 1.03532 & 0.81741 & 0.26678\\
&& BJ&1.39732 & 1.30477 & 1.2284 & 1.16429 & 0.95194 & 0.37701 && 1.39049 & 1.28988 & 1.21382 & 1.14514 & 0.93002 & 0.35009 && 1.36781 & 1.25445 & 1.15682 & 1.08357 & 0.84155 & 0.26733 \\
&& B2& 1.42509 & 1.33146 & 1.25336 & 1.18746 & 0.96828 & 0.37859 && 1.41846 & 1.31631 & 1.2386 & 1.16844 & 0.94603 & 0.35132 && 1.39541 & 1.28107 & 1.18146 & 1.10641 & 0.85599 & 0.26777\\
&&& &&&&& && &&&&& && &&&&&\\ 
0.90&& FS&1.38194 & 1.24521 & 1.14665 & 1.07151 & 0.84631 & 0.31902 && 1.38013 & 1.23462 & 1.13722 & 1.0573 & 0.8284 & 0.29614 && 1.3768 & 1.21853 & 1.0989 & 1.01209 & 0.75305 & 0.22558\\
&& F1& 1.14599 & 1.08902 & 1.03593 & 0.98588 & 0.80906 & 0.3174 && 1.13371 & 1.07536 & 1.01918 & 0.96417 & 0.78609 & 0.29403 && 1.10074 & 1.02695 & 0.95368 & 0.89452 & 0.69585 & 0.22337\\
&& BJ& 1.25607 & 1.16406 & 1.091 & 1.02994 & 0.83007 & 0.31877 && 1.2454 & 1.14935 & 1.07534 & 1.00907 & 0.80757 & 0.29513 && 1.21894 & 1.10679 & 1.01296 & 0.94236 & 0.71756 & 0.22381\\
&&B2&1.29028 & 1.19528 & 1.11933 & 1.05547 & 0.84668 & 0.32015 && 1.27944 & 1.18027 & 1.10321 & 1.03449 & 0.82359 & 0.29617 && 1.2522 & 1.137 & 1.03972 & 0.96621 & 0.73094 & 0.22417\\
\hline
\hline
\end{tabular}
}
\end{center}

\begin{center}
\resizebox{14cm}{2.4 cm}{
\begin{tabular}{c c c cccccc c cccccc }
\hline\hline
&&& \multicolumn{6}{c}{$\rho=0.75$}& &\multicolumn{6}{c}{$\rho=0.9$}  \\
\cline{4-9}  \cline{11-16}   \\
&&&\multicolumn{6}{c}{Sample size}& &\multicolumn{6}{c}{Sample size}  \\
\cline{4-9}  \cline{11-16}   \\
$1-\alpha$ &&& 2&3&4&5&10&100& &2&3&4&5&10&100  \\
\hline
&&& &&&&& && &&&&& \\
0.95&& FS& 1.46525 & 1.28939 & 1.14658 & 1.01888 & 0.64729 & 0.14245 && 1.42347 & 1.15463 & 0.911 & 0.71763 & 0.29692 & 0.05759\\
&& F1& 1.14556 & 1.03852 & 0.93649 & 0.84194 & 0.57003 & 0.14144 && 0.95594 & 0.77545 & 0.61772 & 0.50665 & 0.26153 & 0.05742 \\
&& BJ& 1.27695 & 1.13221 & 1.00879 & 0.89941 & 0.58929 & 0.1415 && 1.11493 & 0.88124 & 0.68985 & 0.55331 & 0.26556 & 0.05742\\
&& B2& 1.30573 & 1.15912 & 1.03267 & 0.91951 & 0.59647 & 0.14153 && 1.14417 & 0.90456 & 0.70529 & 0.56209 & 0.26509 & 0.05742\\
&&& &&&&& && &&&&& \\
0.90&& FS&1.34259 & 1.14445 & 0.99232 & 0.86329 & 0.52451 & 0.11868 && 1.28042 & 0.97149 & 0.7207 & 0.54661 & 0.22448 & 0.04794\\
&& F1& 0.97127 & 0.8764 & 0.78256 & 0.69586 & 0.46211 & 0.11786 && 0.76891 & 0.60795 & 0.47136 & 0.38253 & 0.20199 & 0.0478\\
&& BJ& 1.11145 & 0.96876 & 0.84997 & 0.7469 & 0.47662 & 0.1179 && 0.92724 & 0.70058 & 0.52871 & 0.41713 & 0.20448 & 0.0478 \\
&& B2& 1.14355 & 0.99663 & 0.87273 & 0.76473 & 0.48159 & 0.11793 && 0.95648 & 0.71972 & 0.5389 & 0.42164 & 0.20389 & 0.0478 \\
\hline
\hline
\end{tabular}
}
\end{center}
\end{table}

\begin{table}[]
\caption{Mean Absolute Deviation of the medians from true parameter value for Bivariate Normal\label{t:rhoMAD}}
\begin{center}
\resizebox{18cm}{1.5 cm}{
\begin{tabular}{c c c cccccc c cccccc  c cccccc}
\hline\hline
&&& \multicolumn{6}{c}{$\rho=0.05$}& &\multicolumn{6}{c}{$\rho=0.25$} & &\multicolumn{6}{c}{$\rho=0.5$}  \\
\cline{4-9}  \cline{11-16} \cline{18-23}  \\
&&\multicolumn{6}{c}{Sample size}& &\multicolumn{6}{c}{Sample size} & &\multicolumn{6}{c}{Sample size} \\
\cline{4-9}  \cline{11-16} \cline{18-23}  \\
 &&& 2&3&4&5&10&30& &2&3&4&5&10&30& &2&3&4&5&10&30  \\
\hline
&&& &&&&& && &&&&& && &&&&&\\
&& FS&0.435843 & 0.38597 & 0.344296 & 0.317791 & 0.241437 & 0.078172 && 0.42649 & 0.374332 & 0.336432 & 0.312826 & 0.229622 & 0.072592 && 0.403899 & 0.348557 & 0.301868 & 0.276695 & 0.187183 & 0.053905\\
&& F1& 0.524691 & 0.447443 & 0.388507 & 0.352943 & 0.257415 & 0.0789 && 0.503459 & 0.423748 & 0.37091 & 0.339238 & 0.240099 & 0.072854 && 0.442377 & 0.365753 & 0.309883 & 0.280768 & 0.185597 & 0.053688\\
&& BJ& 0.501272 & 0.431655 & 0.377397 & 0.34465 & 0.254559 & 0.07888 && 0.483887 & 0.411335 & 0.362684 & 0.333394 & 0.238553 & 0.072852 && 0.432762 & 0.361719 & 0.308469 & 0.280536 & 0.186311 & 0.053694\\
&& B2& 0.496309 & 0.42824 & 0.374909 & 0.342918 & 0.254039 & 0.078887 && 0.482603 & 0.410959 & 0.362827 & 0.333595 & 0.239031 & 0.07286 && 0.43626 & 0.364774 & 0.311403 & 0.282947 & 0.187164 & 0.053691\\
\hline
\hline
\end{tabular}
}
\end{center}

\begin{center}
\resizebox{14cm}{1.5 cm}{
\begin{tabular}{c c c cccccc c cccccc }
\hline\hline
&&& \multicolumn{6}{c}{$\rho=0.75$}& &\multicolumn{6}{c}{$\rho=0.9$}  \\
\cline{4-9}  \cline{11-16}   \\
&&\multicolumn{6}{c}{Sample size}& &\multicolumn{6}{c}{Sample size}  \\
\cline{4-9}  \cline{11-16}   \\
&&& 2&3&4&5&10&100& &2&3&4&5&10&100  \\
\hline
&&& &&&&& && &&&&& \\
&& FS&0.330269 & 0.257887 & 0.2049 & 0.174615 & 0.103761 & 0.027943 && 0.207926 & 0.129982 & 0.087098 & 0.070213 & 0.038846 & 0.011204\\
&&F1&0.31468 & 0.237052 & 0.186925 & 0.160447 & 0.097869 & 0.027811 && 0.166243 & 0.10787 & 0.074405 & 0.062108 & 0.037389 & 0.01118 \\
&& BJ& 0.318566 & 0.241772 & 0.190785 & 0.163314 & 0.098572 & 0.027812 && 0.173736 & 0.111831 & 0.076352 & 0.06314 & 0.03745 & 0.01118\\
&& B2& 0.321822 & 0.243293 & 0.191266 & 0.163436 & 0.098132 & 0.027808 && 0.17303 & 0.110691 & 0.075421 & 0.062301 & 0.037248 & 0.01118\\
\hline
\hline
\end{tabular}
}
\end{center}
\end{table}

\begin{table}[]
\caption{Scaled Normal N$(\mu,\mu^{q}), \quad q=1$: (One sided coverage of $(1-\alpha)$th quantile)\label{t:qOne}}
\begin{center}
\resizebox{18cm}{4 cm}{
\begin{tabular}{c c c cccccc c cccccc  c cccccc}
\hline\hline
&& \multicolumn{6}{c}{$\mu=0.1$}& &\multicolumn{6}{c}{$\mu=0.5$} & &\multicolumn{6}{c}{$\mu=1$}  \\
\cline{4-9}  \cline{11-16} \cline{18-23}  \\
&& \multicolumn{6}{c}{Sample size}& &\multicolumn{6}{c}{Sample size} & &\multicolumn{6}{c}{Sample size} \\
\cline{4-9}  \cline{11-16} \cline{18-23}  \\
$\alpha$ &&& 2&3&4&5&10&30& &2&3&4&5&10&30& &2&3&4&5&10&30  \\
\hline
&&& &&&&& && &&&&& && &&&&&\\
0.025 && FS&  0.0298 & 0.0263 & 0.0308 & 0.0263 & 0.0248 & 0.02325 && 0.0261 & 0.0288 & 0.0274 & 0.0283 & 0.0262 & 0.03 &&  0.0254 & 0.0257 & 0.0264 & 0.0263 & 0.0219 & 0.0306\\
       && F1&  0.0256 & 0.0228 & 0.0282 & 0.0243 & 0.0234 & 0.022  && 0.0241 & 0.0261 & 0.0255 & 0.0261 & 0.0247 & 0.0296 && 0.0225 & 0.0243 & 0.0247 & 0.0249 & 0.0208 & 0.0293   \\
&& BJ&  0.0254 & 0.0236 & 0.0282 & 0.0234 & 0.0232 & 0.02225&& 0.0247 & 0.0265 & 0.0261 & 0.0266 & 0.0248 & 0.03 && 0.0241 & 0.0245 & 0.0256 & 0.0253 & 0.021 & 0.0293\\
&& B2& 0.0303 & 0.0263 & 0.0291 & 0.0236 & 0.023 & 0.02225 && 0.0271 & 0.0272 & 0.0267 & 0.027 & 0.0246 & 0.03 && 0.0243 & 0.0246 & 0.0255 & 0.0249 & 0.0209 & 0.0293\\
&&& &&&&& && &&&&& && &&&&&\\

0.05&& FS& 0.0574 & 0.0578 & 0.0568 & 0.0563 & 0.052 & 0.048 && 0.0525 & 0.0584 & 0.0542 & 0.0537 & 0.0511 & 0.0596 && 0.0533 & 0.0527 & 0.0518 & 0.0513 & 0.046 & 0.0553\\
&& F1& 0.0506 & 0.0499 & 0.0523 & 0.0501 & 0.0495 & 0.0495 && 0.0471 & 0.0513 & 0.0495 & 0.0497 & 0.0483 & 0.0567 && 0.049 & 0.0492 & 0.0474 & 0.0479 & 0.0426 & 0.0553\\
&& BJ& 0.0507 & 0.0514 & 0.053 & 0.0501 & 0.05 & 0.04575 && 0.0487 & 0.0523 & 0.0503 & 0.0504 & 0.0488 & 0.057 && 0.0509 & 0.0506 & 0.0485 & 0.0483 & 0.0435 & 0.055 \\
&& B2& 0.0585 & 0.0548 & 0.0551 & 0.0512 & 0.0503 & 0.04600 && 0.0525 & 0.0539 & 0.0516 & 0.0509 & 0.0487 & 0.056 && 0.0519 & 0.0512 & 0.0487 & 0.0484 & 0.0431 & 0.05467\\
&&& &&&&& && &&&&& && &&&&&\\

0.5 && FS& 0.5485 & 0.5372 & 0.5402 & 0.5202 & 0.5122 & 0.5230&& 0.5452 & 0.5352 & 0.5322 & 0.5283 & 0.5102 & 0.5153 && 0.5194 & 0.5239 & 0.5275 & 0.5178 & 0.5213 & 0.5064 \\
&& F1& 0.4986 & 0.5017 & 0.511 & 0.4966 & 0.5 & 0.494 && 0.5165 & 0.4963 & 0.4965 & 0.4986 & 0.4896 & 0.5017 && 0.4849 & 0.4957 & 0.5044 & 0.4966 & 0.5051 & 0.496 \\
&& FS& 0.4886 & 0.4995 & 0.5083 & 0.4952 & 0.4986 & 0.5170 && 0.494 & 0.4923 & 0.4965 & 0.4993 & 0.4906 & 0.5 && 0.4797 & 0.4938 & 0.5027 & 0.4958 & 0.5054 & 0.4963 \\
&& FS& 0.5126 & 0.5097 & 0.5145 & 0.4985 & 0.4997 & 0.5175 && 0.5531 & 0.5211 & 0.5117 & 0.5083 & 0.493 & 0.501 && 0.5171 & 0.5097 & 0.5128 & 0.5021 & 0.5069 & 0.4973\\
&&& &&&&& && &&&&& && &&&&&\\

0.95&& FS& 0.9921 & 0.9757 & 0.9648 & 0.9592 & 0.9516 & 0.9527 && 0.9901 & 0.9734 & 0.9695 & 0.962 & 0.9595 & 0.95367 && 0.9818 & 0.9681 & 0.9589 & 0.9601 & 0.9577 & 0.9507\\
&& F1& 0.9726 & 0.9586 & 0.9523 & 0.9502 & 0.9467 & 0.9500 && 0.9673 & 0.9537 & 0.955 & 0.9472 & 0.9508 & 0.9497 && 0.9589 & 0.9503 & 0.9437 & 0.9488 & 0.951 & 0.9487\\
&& BJ& 0.9466 & 0.9524 & 0.9487 & 0.9471 & 0.9459 & 0.9497 && 0.9464 & 0.9449 & 0.9508 & 0.9453 & 0.9507 & 0.95 && 0.9442 & 0.9458 & 0.9414 & 0.9461 & 0.9508 & 0.9483\\
&& B2& 0.9489 & 0.9538 & 0.9496 & 0.9476 & 0.9459 & 0.94975 && 0.9632 & 0.9543 & 0.9573 & 0.9505 & 0.9526 & 0.9507 && 0.9664 & 0.9597 & 0.952 & 0.9548 & 0.9528 & 0.949 \\
&&& &&&&& && &&&&& && &&&&&\\

0.975 && FS& 0.9982 & 0.9922 & 0.9856 & 0.9809 & 0.977 & 0.9755 && 0.9983 & 0.9905 & 0.9868 & 0.9836 & 0.9808 & 0.976 && 0.9954 & 0.9877 & 0.9809 & 0.9806 & 0.9797 & 0.975 \\
&& F1& 0.9902 & 0.9832 & 0.9783 & 0.9759 & 0.9749 & 0.9742 && 0.9868 & 0.9784 & 0.9769 & 0.973 & 0.9763 & 0.974 && 0.9834 & 0.9744 & 0.9701 & 0.9744 & 0.9757 & 0.97367 \\
&& BJ& 0.974 & 0.9762 & 0.975 & 0.975 & 0.9742 & 0.9740 && 0.9725 & 0.9703 & 0.9745 & 0.9718 & 0.9763 & 0.974 &&0.9704 & 0.9711 & 0.9679 & 0.9739 & 0.975 & 0.97367 \\ 
&& B2& 0.9747 & 0.9768 & 0.9754 & 0.9754 & 0.9742 & 0.9743 && 0.9813 & 0.9774 & 0.979 & 0.975 & 0.9775 & 0.9743 &&  0.9848 & 0.9825 & 0.9745 & 0.9776 & 0.9769 & 0.97367\\
\hline
\hline
\end{tabular}
}
\end{center}

\begin{center}
\resizebox{14cm}{4.3 cm}{
\begin{tabular}{c c c cccccc c cccccc }
\hline\hline
&& \multicolumn{6}{c}{$\mu=3$}& &\multicolumn{6}{c}{$\mu=5$}  \\
\cline{4-9}  \cline{11-16}   \\
&&\multicolumn{6}{c}{Sample size}& &\multicolumn{6}{c}{Sample size}  \\
\cline{4-9}  \cline{11-16}  \\
$1-\alpha$ &&& 2&3&4&5&10&30& &2&3&4&5&10&30  \\
\hline
&&& &&&&& && &&&&& \\
0.025&& FS& 0.0262 & 0.0262 & 0.0275 & 0.0256 & 0.0248 & 0.0256 && 0.0235 & 0.0259 & 0.0251 & 0.025 & 0.0232 & 0.0237 \\
&& F1& 0.0256 & 0.026 & 0.027 & 0.0246 & 0.0246 & 0.025 && 0.0234 & 0.0255 & 0.0241 & 0.0246 & 0.0237 & 0.0236 \\
&& BJ& 0.0268 & 0.0264 & 0.0275 & 0.0251 & 0.0246 & 0.025 && 0.024 & 0.0258 & 0.025 & 0.025 & 0.0231 & 0.0234\\
&& B2& 0.0262 & 0.0262 & 0.0271 & 0.0251 & 0.0246 & 0.025 && 0.0235 & 0.0257 & 0.0247 & 0.0247 & 0.0233 & 0.0234 \\
&&& &&&&& && &&&&& \\
0.05&& FS& 0.0521 & 0.05 & 0.0547 & 0.0512 & 0.051 & 0.057 && 0.0488 & 0.0486 & 0.049 & 0.0504 & 0.0489 & 0.051\\
&& F1& 0.0505 & 0.0495 & 0.0534 & 0.0496 & 0.05 & 0.057 && 0.0481 & 0.048 & 0.0482 & 0.0501 & 0.0481 & 0.051\\
&& BJ& 0.0523 & 0.0502 & 0.0541 & 0.0508 & 0.0503 & 0.057 && 0.0492 & 0.0486 & 0.049 & 0.0503 & 0.0485 & 0.051\\
&& B2& 0.0516 & 0.0497 & 0.0537 & 0.0501 & 0.0503 & 0.057 && 0.0486 & 0.0483 & 0.0484 & 0.0503 & 0.0482 & 0.051\\
&&& &&&&& && &&&&& \\
0.5&& FS& 0.5053 & 0.5093 & 0.5101 & 0.5064 & 0.5021 & 0.501 && 0.5062 & 0.4988 & 0.4961 & 0.5034 & 0.4978 & 0.509 \\
&& F1& 0.4933 & 0.5014 & 0.5018 & 0.4994 & 0.4979 & 0.499 && 0.5004 & 0.4943 & 0.4922 & 0.5000 & 0.4953 & 0.5086 \\
&& BJ& 0.4926 & 0.5013 & 0.5016 & 0.4993 & 0.4978 & 0.4987 && 0.501 & 0.494 & 0.4924 & 0.5001 & 0.4952 & 0.5086\\
&& B2& 0.5005 & 0.5041 & 0.5037 & 0.5019 & 0.4983 & 0.4986 && 0.504 & 0.4959 & 0.4935 & 0.5005 & 0.4952 & 0.5083\\
&&& &&&&& && &&&&& \\
0.95&& FS& 0.9587 & 0.9558 & 0.9552 & 0.9528 & 0.9506 & 0.9453 && 0.9513 & 0.9495 & 0.953 & 0.9535 & 0.9491 & 0.955\\
&& F1& 0.9492 & 0.9492 & 0.9509 & 0.9497 & 0.948 & 0.945 && 0.9472 & 0.9472 & 0.9519 & 0.952 & 0.9485 & 0.9543\\
&& BJ& 0.9445 & 0.9456 & 0.9494 & 0.9483 & 0.947 & 0.9447 && 0.9446 & 0.9458 & 0.9508 & 0.9509 & 0.9481 & 0.954\\
&& B2& 0.9633 & 0.9552 & 0.954 & 0.9512 & 0.9484 & 0.9447 && 0.9539 & 0.9493 & 0.9526 & 0.9531 & 0.9486 & 0.9543\\
&&& &&&&& && &&&&& \\
0.975&& FS& 0.9801 & 0.979 & 0.9772 & 0.9748 & 0.9768 & 0.968 && 0.9766 & 0.9752 & 0.9775 & 0.9762 & 0.9752 & 0.973\\
&& F1& 0.9732 & 0.9752 & 0.9741 & 0.9722 & 0.975 & 0.968 && 0.9741 & 0.9734 & 0.9764 & 0.9748 & 0.9742 & 0.973\\
&& BJ& 0.9704 & 0.9733 & 0.9723 & 0.9709 & 0.9747 & 0.968 && 0.9716 & 0.972 & 0.9753 & 0.9741 & 0.9739 & 0.9727\\
&& B2& 0.9846 & 0.9797 & 0.9777 & 0.9733 & 0.9755 & 0.968 && 0.979 & 0.9759 & 0.977 & 0.9754 & 0.9743 & 0.973\\
\hline
\hline
\end{tabular}
}
\end{center}
\end{table}

\begin{table}[]
\caption{Scaled Normal N$(\mu,\mu^{q}), \quad q=1$: (Two sided coverage of $(1-\alpha)$th quantile)\label{t:qTwo}}
\begin{center}
\resizebox{18cm}{2.3 cm}{
\begin{tabular}{c c c cccccc c cccccc  c cccccc}
\hline\hline
&&& \multicolumn{6}{c}{$\mu=0.1$}& &\multicolumn{6}{c}{$\mu=0.5$} & &\multicolumn{6}{c}{$\mu=1$}  \\
\cline{4-9}  \cline{11-16} \cline{18-23}  \\
&&&\multicolumn{6}{c}{Sample size}& &\multicolumn{6}{c}{Sample size} & &\multicolumn{6}{c}{Sample size} \\
\cline{4-9}  \cline{11-16} \cline{18-23}  \\
$\alpha$ &&& 2&3&4&5&10&30& &2&3&4&5&10&30& &2&3&4&5&10&30  \\
\hline 
&&& &&&&& && &&&&& && &&&&&\\
0.95&& FS& 0.9684 & 0.9659 & 0.9548 & 0.9546 & 0.9522 & 0.95225 && 0.9722 & 0.9617 & 0.9594 & 0.9553 & 0.9546 & 0.946 && 0.97 & 0.962 & 0.9545 & 0.9543 & 0.9578 & 0.9443\\
&& F1& 0.9646 & 0.9604 & 0.9501 & 0.9516 & 0.9515 & 0.95225 && 0.9627 & 0.9523 & 0.9514 & 0.9469 & 0.9516 & 0.9443 && 0.9609 & 0.9501 & 0.9454 & 0.9495 & 0.9549 & 0.9443\\
&& BJ& 0.9486 & 0.9526 & 0.9468 & 0.9516 & 0.951 & 0.95175 && 0.9478 & 0.9438 & 0.9484 & 0.9452 & 0.9515 & 0.944 && 0.9463 & 0.9466 & 0.9423 & 0.9486 & 0.954 & 0.9443\\
&& B2&0.9444 & 0.9505 & 0.9463 & 0.9518 & 0.9512 & 0.95175 && 0.9542 & 0.9502 & 0.9523 & 0.948 & 0.9529 & 0.944 && 0.9605 & 0.9579 & 0.949 & 0.9527 & 0.956 & 0.9441\\
&&& &&&&& && &&&&& && &&&&&\\
0.90&& FS& 0.9347 & 0.9179 & 0.908 & 0.9029 & 0.8996 & 0.90475 && 0.9376 & 0.915 & 0.9153 & 0.9083 & 0.9084 & 0.894 && 0.9285 & 0.9154 & 0.9071 & 0.9088 & 0.9117 & 0.8953\\
&& F1& 0.922 & 0.9087 & 0.9 & 0.9001 & 0.8972 & 0.9045 && 0.9202 & 0.9024 & 0.9055 & 0.8975 & 0.9025 & 0.893 && 0.9099 & 0.9011 & 0.8963 & 0.9009 & 0.9084 & 0.893\\
&& BJ& 0.8959 & 0.901 & 0.8957 & 0.897 & 0.8959 & 0.904 && 0.8977 & 0.8926 & 0.9005 & 0.8949 & 0.9019 & 0.893 && 0.8933 & 0.8952 & 0.8929 & 0.8978 & 0.9073 & 0.893\\
&& B2& 0.8904 & 0.899 & 0.8945 & 0.8964 & 0.8956 & 0.90375 && 0.9107 & 0.9004 & 0.9057 & 0.8996 & 0.9039 & 0.8947 && 0.9145 & 0.9085 & 0.9033 & 0.9064 & 0.9097 & 0.8943\\
\hline
\hline
\end{tabular}
}
\end{center}

\begin{center}
\resizebox{14cm}{2.4 cm}{
\begin{tabular}{c c c cccccc c cccccc }
\hline\hline
&&& \multicolumn{6}{c}{$\mu=3$}& &\multicolumn{6}{c}{$\mu=5$}  \\
\cline{4-9}  \cline{11-16}   \\
&&\multicolumn{6}{c}{Sample size}& &\multicolumn{6}{c}{Sample size}  \\
\cline{4-9}  \cline{11-16}  \\
$1-\alpha$ && 2&3&4&5&10&100& &2&3&4&5&10&100  \\
\hline
&&& &&&&& && &&&&& \\
0.95&& FS& 0.9539 & 0.9528 & 0.9497 & 0.9492 & 0.952 & 0.9423 && 0.9531 & 0.9493 & 0.9524 & 0.9512 & 0.9522 & 0.9493\\
&& F1& 0.9476 & 0.9492 & 0.9471 & 0.9476 & 0.9504 & 0.943 && 0.9507 & 0.9479 & 0.9523 & 0.9502 & 0.9512 & 0.9493\\
&& BJ& 0.9436 & 0.9469 & 0.9448 & 0.9458 & 0.9501 & 0.943 && 0.9476 & 0.9462 & 0.9503 & 0.9491 & 0.9509 & 0.949\\
&& B2& 0.9584 & 0.9535 & 0.9506 & 0.9482 & 0.9509 & 0.943 && 0.9555 & 0.9502 & 0.9523 & 0.9507 & 0.9513 & 0.9493\\
&&& &&&&& && &&&&& \\
0.90&& FS& 0.9066 & 0.9058 & 0.9005 & 0.9016 & 0.8996 & 0.8883&& 0.9025 & 0.9009 & 0.904 & 0.9031 & 0.9002 & 0.904\\
&& F1& 0.8987 & 0.8997 & 0.8975 & 0.9001 & 0.898 & 0.888 && 0.8991 & 0.8992 & 0.9037 & 0.9019 & 0.9004 & 0.903 \\
&& BJ& 0.8922 & 0.8954 & 0.8953 & 0.8975 & 0.8967 & 0.8876 && 0.8954 & 0.8972 & 0.9018 & 0.9006 & 0.8996 & 0.903 \\
&& B2& 0.9117 & 0.9055 & 0.9003 & 0.9011 & 0.8981 & 0.8877 && 0.9053 & 0.901 & 0.9042 & 0.9028 & 0.9004 & 0.903 \\
\hline
\hline
\end{tabular}
}
\end{center}
\end{table}

\begin{table}[]
\caption{Length of two sided credible region for $\mu$ in Scaled Normal N$(\mu,\mu^{q}), \quad q=1$:\label{t:qLength}}
\begin{center}
\resizebox{18cm}{2 cm}{
\begin{tabular}{c c c cccccc c cccccc  c cccccc}
\hline\hline
&& \multicolumn{6}{c}{$\mu=0.1$}& &\multicolumn{6}{c}{$\mu=0.5$} & &\multicolumn{6}{c}{$\mu=1$}  \\
\cline{4-9}  \cline{11-16} \cline{18-23}  \\
&&\multicolumn{6}{c}{Sample size}& &\multicolumn{6}{c}{Sample size} & &\multicolumn{6}{c}{Sample size} \\
\cline{4-9}  \cline{11-16} \cline{18-23}  \\
$\alpha$ &&& 2&3&4&5&10&30& &2&3&4&5&10&30& &2&3&4&5&10&30  \\
\hline
&&& &&&&& && &&&&& && &&&&&\\
0.95&& FS& 1.35388 & 0.81608 & 0.55837 & 0.42236 & 0.21588 & 0.10209 && 2.45075 & 1.68304 & 1.3317 & 1.12244 & 0.69992 & 0.37435 && 3.10638 & 2.29652 & 1.88748 & 1.63664 & 1.08409 & 0.59789\\
&& F1& 1.07899 & 0.65473 & 0.4709 & 0.37123 & 0.20545 & 0.10089 && 2.18481 & 1.53574 & 1.23959 & 1.05883 & 0.68 & 0.37095 && 2.91414 & 2.19116 & 1.82143 & 1.59063 & 1.06911 & 0.59519\\
&& BJ& 0.90065 & 0.59879 & 0.44976 & 0.36167 & 0.20458 & 0.10085 && 2.02621 & 1.48871 & 1.2191 & 1.04753 & 0.67744 & 0.37054 && 2.79285 & 2.14955 & 1.79932 & 1.57636 & 1.06472 & 0.59442\\
&& B2& 1.1766 & 0.69552 & 0.48718 & 0.37729 & 0.20595 & 0.10092 && 2.60951 & 1.6889 & 1.30849 & 1.09641 & 0.68508 & 0.37121 && 3.27241 & 2.31822 & 1.8791 & 1.6226 & 1.07371 & 0.59543\\
&&& &&&&& && &&&&& && &&&&&\\
0.90&& FS& 0.98973 & 0.58489 & 0.40703 & 0.31504 & 0.17057 & 0.08402 && 1.89682 & 1.32286 & 1.05959 & 0.90102 & 0.57339 & 0.31147 && 2.47091 & 1.85236 & 1.53532 & 1.33879 & 0.89761 & 0.49922\\
&&F1& 0.76831 & 0.47154 & 0.34842 & 0.28053 & 0.16307 & 0.08307 && 1.68005 & 1.20546 & 0.98658 & 0.85033 & 0.5572 & 0.30865 && 2.3109 & 1.76538 & 1.48078 & 1.30063 & 0.88509 & 0.49696\\
&&BJ& 0.64081 & 0.4361 & 0.33551 & 0.27464 & 0.16249 & 0.08304 && 1.5632 & 1.17205 & 0.97204 & 0.84228 & 0.5553 & 0.30832 && 2.21963 & 1.73431 & 1.46415 & 1.28979 & 0.88163 & 0.49633\\
&& B2& 0.81998 & 0.49647 & 0.35854 & 0.28417 & 0.16339 & 0.08309 && 1.97825 & 1.31223 & 1.03464 & 0.87685 & 0.56087 & 0.30884 && 2.5644 & 1.85554 & 1.52181 & 1.32356 & 0.88842 & 0.49714\\
\hline
\hline
\end{tabular}
}
\end{center}

\begin{center}
\resizebox{14cm}{2.6 cm}{
\begin{tabular}{c c c cccccc c cccccc }
\hline\hline
&&& \multicolumn{6}{c}{$\mu=3$}& &\multicolumn{6}{c}{$\mu=5$}  \\
\cline{4-9}  \cline{11-16}   \\
&&&\multicolumn{6}{c}{Sample size}& &\multicolumn{6}{c}{Sample size}  \\
\cline{4-9}  \cline{11-16}  \\
$1-\alpha$ &&& 2&3&4&5&10&100& &2&3&4&5&10&100  \\
\hline
&&& &&&&& && &&&&& \\
0.95&& FS& 4.93805 & 3.89768 & 3.31892 & 2.93372 & 2.03048 & 1.15517 && 6.28185 & 5.02079 & 4.30686 & 3.8347 & 2.67314 & 1.53541\\
&&F1& 4.86896 & 3.86389 & 3.29744 & 2.91908 & 2.02564 & 1.1543 && 6.2494 & 5.0041 & 4.2966 & 3.82761 & 2.67069 & 1.53497 \\
&& BJ& 4.79247 & 3.82639 & 3.27446 & 2.90282 & 2.02016 & 1.15328 && 6.18417 & 4.97077 & 4.27542 & 3.81265 & 2.66557 & 1.534 \\
&& B2& 5.01579 & 3.9107 & 3.3196 & 2.93084 & 2.02755 & 1.1544 && 6.32213 & 5.02842 & 4.30762 & 3.83365 & 2.67174 & 1.53502\\
&&& &&&&& && &&&&& \\
0.90&& FS& 4.05274 & 3.22139 & 2.75313 & 2.439 & 1.69592 & 0.96712 && 5.19936 & 4.17415 & 3.58887 & 3.19992 & 2.23699 & 1.28628\\
&& F1& 3.99422 & 3.19282 & 2.73498 & 2.42664 & 1.69185 & 0.96637 && 5.17182 & 4.16004 & 3.58021 & 3.19395 & 2.23493 & 1.2859\\
&& BJ& 3.93367 & 3.16275 & 2.71639 & 2.41341 & 1.68732 & 0.9655 && 5.11898 & 4.13278 & 3.56278 & 3.18159 & 2.23068 & 1.28511\\
&& B2& 4.09773 & 3.22577 & 2.75069 & 2.43495 & 1.6932 & 0.96644 && 5.22244 & 4.17701 & 3.58788 & 3.19815 & 2.23567 & 1.28592\\
\hline
\hline
\end{tabular}
}
\end{center}

\end{table}

\begin{table}[]
\caption{Mean Absolute Deviation of the medians from true parameter value for Scaled Normal\label{t:qMAD}}
\begin{center}
\resizebox{18cm}{1.5 cm}{
\begin{tabular}{c c c cccccc c cccccc  c cccccc}
\hline\hline
&& \multicolumn{6}{c}{$\mu=0.1$}& &\multicolumn{6}{c}{$\mu=0.5$} & &\multicolumn{6}{c}{$\mu=1$}  \\
\cline{4-9}  \cline{11-16} \cline{18-23}  \\
&&\multicolumn{6}{c}{Sample size}& &\multicolumn{6}{c}{Sample size} & &\multicolumn{6}{c}{Sample size} \\
\cline{4-9}  \cline{11-16} \cline{18-23}  \\
 &&& 2&3&4&5&10&30& &2&3&4&5&10&30& &2&3&4&5&10&30  \\
\hline
&&& &&&&& && &&&&& && &&&&&\\
&&FS& 0.093022 & 0.070158 & 0.058998 & 0.051342 & 0.034347 & 0.018777 && 0.299045 & 0.242821 & 0.206852 & 0.184339 & 0.127275 & 0.074172 && 0.475668 & 0.383737 & 0.334208 & 0.294802 & 0.207019 & 0.12207\\
&& F1& 0.082243 & 0.064151 & 0.05517 & 0.048795 & 0.033557 & 0.018628 && 0.287522 & 0.235675 & 0.201921 & 0.180667 & 0.126157 & 0.073927 && 0.472906 & 0.382297 & 0.333217 & 0.294453 & 0.206771 & 0.122045\\
&& BJ& 0.07962 & 0.063413 & 0.05494 & 0.048625 & 0.033538 & 0.018642 && 0.285706 & 0.235153 & 0.201856 & 0.180757 & 0.126087 & 0.073904 && 0.473698 & 0.382369 & 0.332932 & 0.294322 & 0.206728 & 0.122044\\
&& B2& 0.091499 & 0.067335 & 0.056539 & 0.049261 & 0.033613 & 0.018646 && 0.31046 & 0.240836 & 0.204446 & 0.182431 & 0.126208 & 0.073907 && 0.475886 & 0.382178 & 0.33241 & 0.293996 & 0.206723 & 0.122048\\
\hline
\hline
\end{tabular}
}
\end{center}

\begin{center}
\resizebox{14cm}{1.6 cm}{
\begin{tabular}{c c c cccccc c cccccc }
\hline\hline
&&& \multicolumn{6}{c}{$\mu=3$}& &\multicolumn{6}{c}{$\mu=5$}  \\
\cline{4-9}  \cline{11-16}   \\
&&&\multicolumn{6}{c}{Sample size}& &\multicolumn{6}{c}{Sample size}  \\
\cline{4-9}  \cline{11-16}  \\
 &&& 2&3&4&5&10&100& &2&3&4&5&10&100  \\
\hline
&&& &&&&& && &&&&& \\
&& FS&0.915859 & 0.743588 & 0.641391 & 0.57712 & 0.404544 & 0.237739 && 1.216561 & 0.987328 & 0.846028 & 0.759754 & 0.538429 & 0.304847\\
&& F1& 0.923632 & 0.747219 & 0.644017 & 0.578878 & 0.405205 & 0.237835 && 1.221673 & 0.990262 & 0.847762 & 0.760816 & 0.538839 & 0.304877\\
&& BJ& 0.923667 & 0.747361 & 0.643863 & 0.578883 & 0.405175 & 0.237835 && 1.221856 & 0.990214 & 0.847798 & 0.760852 & 0.538848 & 0.304872\\
&& B2&0.915956 & 0.745368 & 0.642904 & 0.578482 & 0.405086 & 0.237842 && 1.217969 & 0.989117 & 0.847358 & 0.760632 & 0.538809 & 0.304882\\
\hline
\hline
\end{tabular}
}
\end{center}

\end{table}

\end{document}